\documentclass[11pt]{amsart}

\usepackage{amssymb}
\usepackage{amsmath}
\usepackage{amsfonts}
\usepackage{graphicx}
\usepackage{subfigure}
\usepackage{amsthm}
\usepackage{enumerate}
\usepackage[mathscr]{eucal}
\usepackage{mathrsfs}
\usepackage{verbatim}
\usepackage{yhmath}
\usepackage{epstopdf}
\usepackage{color}
\usepackage{hyperref} 

\makeatletter
\@namedef{subjclassname@2020}{%
  \textup{2020} Mathematics Subject Classification}
\makeatother

\numberwithin{equation}{section}
\numberwithin{figure}{section}

\theoremstyle{plain}
\newtheorem{theorem}{Theorem}[section]
\newtheorem{lemma}[theorem]{Lemma}
\newtheorem{proposition}[theorem]{Proposition}
\newtheorem{corollary}[theorem]{Corollary}

\theoremstyle{plain}

\theoremstyle{remark}
\newtheorem{remark}[theorem]{Remark}

\allowdisplaybreaks[4]

\begin{document}
\date{}

\title[Bessel functions and Weyl's law]{Bessel functions and Weyl's law for balls and spherical shells}

\author{Jingwei Guo}
\address{School of Mathematical Sciences\\
University of Science and Technology of China\\
Hefei, 230026\\ P.R. China}
\email{jwguo@ustc.edu.cn}

\author{Tao Jiang}
\address{School of Mathematics and Statistics\\
Anhui Normal University\\
Wuhu, 241002\\ P.R. China}
\email{tjiang@ahnu.edu.cn}

\author{Zuoqin Wang}
\address{School of Mathematical Sciences\\
University of Science and Technology of China\\
Hefei, 230026\\ P.R. China}
\email{wangzuoq@ustc.edu.cn}

\author{Xuerui Yang}
\address{Department of Mathematics\\
University of Illinois at Urbana-Champaign\\
Urbana, IL, 61801\\USA}
\email{xueruiy3@illinois.edu}



\begin{abstract}
The purpose of this paper is twofold. One is to investigate the properties of the zeros of cross-products of Bessel functions or derivatives of ultraspherical Bessel functions, as well as the properties of the zeros of the derivative of the first-kind ultraspherical Bessel function. The properties we study include asymptotics (with uniform and nonuniform remainder estimates), upper and lower bounds and so on. In addition, we provide the number of zeros of a certain cross-product within a large circle and show that all its zeros are real and simple.
These results may be of independent interest.

The other is to investigate the Dirichlet/Neumann Laplacian on balls and spherical shells in $\mathbb{R}^d$  ($d\geq 2$) and the remainder of the associated Weyl's law. We obtain new upper bounds in all dimensions, both in the Dirichlet and Neumann cases. The proof relies on our studies of Bessel functions and the latest development in the Gauss circle problem, which was driven by the application of the emerging decoupling theory of harmonic analysis.
\end{abstract}

\maketitle

\tableofcontents

\section{Introduction} \label{intro}

Consider the Laplacian associated with a bounded Euclidean domain $\mathscr{D}\subset \mathbb{R}^d$ ($d\geq 2$), with either Dirichlet or Neumann boundary conditions. Denote by  $\mathscr{N}_\mathscr{D}(\mu)=\#\{j\ | \lambda_j \le \mu\}$ the corresponding eigenvalue counting function, where $\lambda_j^2$ are the Dirichlet/Neumann eigenvalues. The Weyl remainder $\mathscr{R}_\mathscr{D}(\mu)$ is defined to be the quantity in the following expression:
\begin{equation*}
	\mathscr{N}_\mathscr{D}(\mu)=\frac{\omega_d}{(2\pi)^{d}} \left|\mathscr{D}\right|\mu^d\mp \frac{\omega_{d-1}}{4(2\pi)^{d-1}}\left|\partial\mathscr{D}\right| \mu^{d-1}+\mathscr{R}_\mathscr{D}(\mu),
\end{equation*}
where $\omega_k$ denotes the volume of the unit ball in $\mathbb{R}^k$ and the sign ``$-$'' (resp., ``$+$'') refers to the Dirichlet (resp., Neumann) boundary condition. The study of such an eigenvalue counting function was initiated by Weyl \cite{weyl:1912, weyl:1913}.  Weyl's conjecture claims that  the remainder $\mathscr{R}_\mathscr{D}(\mu)$ is of order $o(\mu^{d-1})$ as $\mu \to \infty$. Melrose \cite{Mel:1980} and Ivrii \cite{Ivrii1980} proved this conjecture for manifolds with boundary under certain ``non-periodicity condition" on the billiard flow, namely the set of periodic billiard trajectories has measure zero. It is still unknown whether the non-periodicity condition holds for any bounded Euclidean domain (with sufficiently nice boundary).

A natural question is to study the asymptotic order of the remainder $\mathscr{R}_\mathscr{D}(\mu)$.
It is known that there is no universal constant $\kappa<d-1$  so that $\mathscr{R}_\mathscr{D}(\mu)=O(\mu^\kappa)$ holds for all $\mathscr{D}\subset \mathbb{R}^d$. In fact, for each $\kappa <1$ Lazutkin and Terman \cite{Lazu:1982} had constructed convex planar domains with specific billiard dynamics so that the remainder is not $O(\mu^\kappa)$.

On the other direction, however, the remainder can be much smaller than $o(\mu^{d-1})$ for specific domains in $\mathbb{R}^d$. For example, for the Dirichlet Laplacian on disks in $\mathbb{R}^2$, Kuznetsov and  Fedosov \cite{kuz:1965} and Colin de Verdi\`ere \cite{colin:2011} showed that the remainder is $O(\mu^{2/3})$. This bound was recently improved to $O(\mu^{2/3-1/495})$ in \cite{GWW2018} and to
\begin{equation*}
	O\left(\mu^{\frac{131}{208}}(\log \mu)^{\frac{18627}{8320}}\right)
\end{equation*}
in \cite{GMWW:2019}. Huxley \cite{Huxley:2024} obtained the same bound for both the Dirichlet and Neumann Laplacian on disks.  Kuznetsov proved the bound $O(\mu^{2/3})$ for ellipses in \cite{Kuznecov:1965} and  considered planar domains of separable variable type in \cite{Kuznecov:1966}. In \cite{GMWW:2019} we studied  the Dirichlet Laplacian on annuli and obtained the bound $O(\mu^{2/3})$ in general and the bound $O(\mu^{131/208}(\log \mu)^{18627/8320})$ under the rationality assumption on the ``slope" $\arccos(r/R)/\pi$, where $r$ and $R$ are the inner and outer radii of the annulus. The Dirichlet Laplacian on balls in $\mathbb{R}^d$ ($d\geq 3$) was studied in \cite{Guo} and a bound
\begin{equation*}
	O\left(\mu^{d-2+\frac{131}{208}}(\log \mu)^{\frac{18627}{8320}} \right)
\end{equation*}
was obtained.

In this paper, we study the Dirichlet/Neumann Laplacian on balls and spherical shells in $\mathbb{R}^d$. Throughout this paper we denote by
\begin{equation*}
	\mathbb{S}=\mathbb{S}_{r,R}^d=\{x\in\mathbb{R}^d : r<|x|<R \}  
\end{equation*}
a spherical shell in $\mathbb{R}^d$, where $r$ and $R$ are positive numbers with $r<R$, by
\begin{equation*}
	\mathbb{B}=\mathbb{B}_R^d=\{x\in\mathbb{R}^d : |x|<R\} 
\end{equation*}
a ball in $\mathbb{R}^d$, and by
\begin{equation*}
\theta^*=0.3144831759741\cdots
\end{equation*}
the opposite of the unique solution in the interval $[-0.35,-0.3]$ to the equation \eqref{definition of theta} (see Theorem \ref{expo sum} below). We obtain the following bounds.

\begin{theorem}\label{specthm}
	Let $\mathscr{R}_\mathscr{D}(\mu)$ be the  Weyl remainder associated with the domain $\mathscr{D}$ for either   Dirichlet or  Neumann eigenvalues.  For  $d\geq 2$ and $\epsilon>0$, we have
\begin{enumerate}
\item
	\begin{equation*}
		\mathscr{R}_\mathbb{B}(\mu)=O_{\epsilon}\left(\mu^{d-2+2\theta^*+\epsilon}\right);
	\end{equation*}
\item
	\begin{equation*}
		\mathscr{R}_\mathbb{S}(\mu)=O\left(\mu^{d-2+\frac{2}{3}}\right)
	\end{equation*}
	which can be improved to
	\begin{equation*}
		\mathscr{R}_\mathbb{S}(\mu)=O_{\epsilon}\left(\mu^{d-2+2\theta^*+\epsilon}\right)
	\end{equation*}
	if $\pi^{-1}\arccos(r/R)\in \mathbb{Q}$.
\end{enumerate}
\end{theorem}

\begin{remark}
For comparison between these bounds and previous ones, it is worth noting that $2\theta^*=0.628966\cdots$, whereas $131/208=0.629807\cdots$.

By using Huxley's bound in \cite[Proposition 3]{Huxley:2003} (which has previously been used in \cite{GMWW:2019, Guo, Huxley:2024}), and combining the work presented in Sections \ref{zeros}--\ref{sec5}, we can already obtain new results. Specifically, we can extend the results in \cite{GMWW:2019} from annuli to spherical shells in any dimension, covering both the Dirichlet and Neumann scenarios. Furthermore, we can extend the results in \cite{Guo} for balls to include the Neumann case.

In fact, we can go even further. By using the latest development in the Gauss circle problem, which was driven by the application of the emerging decoupling theory of harmonic analysis, rather than relying on Huxley's \cite[Proposition 3]{Huxley:2003}, we can obtain new bounds in all dimensions,  both in the Dirichlet and Neumann cases, by improving the exponent from $131/208$ to $2\theta^*$.
\end{remark}

We would like to take this chance to explain one difference between the cases $d=2$ and $d>2$.
It is well known that the eigenvalue counting problem for certain Euclidean domains (whose billiard flows are completely integrable) can be converted into some ``almost lattice point problems'' (with each problem associated with a different domain of the same dimension), essentially involving the counting of lattice points subject to certain translations. Even thought the planar domain for the corresponding lattice point problem could be bad (non-convex, with cusp points, etc.) and despite the potential presence of complicated translations, it is still possible to adapt the methods/arguments developed in number theory and harmonic analysis for the Gauss circle problem (which counts lattice points within disks) over the past 100 years, and achieve a satisfactory asymptotic bound. As a result, one gets an equally nice bound for  the eigenvalue counting problem, as alluded to above. For the case $d=2$, this idea was used by many authors (c.f.  \cite{kuz:1965, Kuznecov:1965, Kuznecov:1966, colin:2011, GWW2018, GMWW:2019, FLPS:2023, Huxley:2024}) to study the Dirichlet/Neumann eigenvalue counting function  of disks, annuli, ellipses, etc. In \cite{RWY} Rudnick, Wigman and Yesha also studied the  Robin eigenvalue counting function of the disk via this method.

For the case  $d>2$, one may hope to mimic the aforementioned strategy to convert the eigenvalue counting problem for balls and spherical shells to some almost lattice point problems associated with specific domains in $\mathbb R^d$, and then by adapting existing techniques in handling lattice point problems within $d$-dimensional domains (like balls, ellipsoids, etc.) to get equally nice asymptotic bounds (e.g., achieving $O(\mu^{d-2})$ for $d>4$). However, while it is true that the eigenvalues still correspond to ``almost lattice points'', it seems that the existing techniques employed to address lattice point problems within d-dimensional domains are not applicable to the specific domains encountered here. In fact, according to Eswarathasan, Polterovich and Toth's \cite[Proposition 1.8]{EPT:2016}, there is a ``lower bound on average" for balls $\mathbb{B}$ in $\mathbb{R}^d$ ($d\geq 2$),
\begin{equation*}
	\frac{1}{\mu}\int_{\mu}^{2\mu}\left|\mathscr{R}_\mathbb{B}(\tau) \right|\,\textrm{d}\tau\ge c \mu^{d-2+\frac{1}{2}},
\end{equation*}
which is much larger than the well-known ``lattice point counting remainder" for balls (e.g., $O_\varepsilon(\mu^{d-2+\varepsilon})$ for $d=4$ and $O(\mu^{d-2})$ for $d>4$). An interesting question would be whether the true order of the error term  for the eigenvalue counting problem for balls is, for any $\varepsilon>0$,
\begin{equation*}
	O_{\varepsilon}\left(\mu^{d-2+\frac{1}{2}+\varepsilon} \right).
\end{equation*}

So instead of searching for a $d$-dimensional solution, we will convert the eigenvalue counting problems under consideration to some  ``weighted'' planar   lattice point  counting problems. This method was used by the first author \cite{Guo} to study the Dirichlet eigenvalue counting function for balls, and by Filonov, Levitin,  Polterovich and Sher \cite{FLPS:2023} in confirming Polya's conjecture for balls of dimension $d \ge 3$ (in the Dirichlet case).

More specifically, the procedure of proving Theorem \ref{specthm} is as follows. We first derive approximations of eigenvalues with \textit{uniform} error estimates based on our investigation on the zeros of various expressions of Bessel related functions. Following this, we relate the eigenvalue counting problems to certain weighted lattice point counting problems associated with two types of special planar domains, with weights coming from multiplicities of eigenvalues. At last we count lattice points by applying estimates obtained from analytic number theory and harmonic analysis, and conclude the problems with satisfactory  remainder estimates.

One major difficulty we encounter lies in approximating eigenvalues with \textit{uniform} error terms. The eigenvalues we aim to study can be determined as the squares of the zeros of certain cross-products of Bessel functions or  derivatives of ultraspherical Bessel functions, as well as the squares of the zeros of the derivative of the ultraspherical Bessel function of the first kind. Notice that Bessel functions and ultraspherical Bessel functions are widely used in mathematics, physics and engineering science to analyze boundary value problems with spherical or cylindrical geometry. Extensive research results have been obtained regarding them. To mention a few examples, McMahon \cite{mcmahon:1894} in 1894 gave asymptotics of the zeros of the Bessel and certain related functions (see also \cite[P. 371 and 441]{abram:1972}); Cochran \cite{cochran:1964, Cochran:1966a, Cochran:1966} in the 1960s examined properties (including asymptotics, analyticity, etc.) of the zeros of cross-products of Bessel functions and their derivatives (see also \cite[P. 374]{abram:1972}); Filonov, Levitin, Polterovich and Sher \cite{FLPS:2024} very recently obtain some nice results on uniform enclosures for the phase and zeros of Bessel functions and their derivatives. However, to our knowledge the cross-product of derivatives of ultraspherical Bessel functions  has not been well studied so far, although its application in physics becomes increasingly important. We also find that known asymptotics of aforementioned zeros, if any, are not of the type that we require with uniform error terms.

In this paper, the properties of zeros we study include asymptotics (with both uniform and nonuniform remainder estimates), upper and lower bounds and so on. Here we will briefly explain some of the results we have obtained by taking the zeros $x''_{\nu,k}$ of the cross-product of derivatives of ultraspherical Bessel functions $j_{\nu,\delta}'(Rx)y_{\nu,\delta}'(rx)-j_{\nu,\delta}'(rx)y_{\nu,\delta}'(Rx)$ (see \eqref{eigenequation1}) as an example. For detailed results, please refer to the subsequent sections. For any fixed $\nu\geq |\delta|$ and sufficiently large $k$, we give in Theorem \ref{thm2.20} that
\begin{equation*}
	x''_{\nu,k}=\frac{\pi}{R-r}k+O_{\nu}\left(\frac{1}{k} \right).
\end{equation*}
(See Theorem \ref{thm4.12} for an analogous result of the zeros of the derivative of the ultraspherical Bessel function.) In the special case when $\delta=0$ (that is, ultraspherical Bessel functions are simply the usual Bessel functions), this asymptotics can be derived directly from \cite[9.5.28--9.5.31 on P. 374]{abram:1972}. The implicit constant in the error term is not uniform in $\nu$. However, the uniformity of the implicit constant in $\nu$ and $k$ is vital for the eigenvalue counting problems. After a somewhat long and technical computation, we manage to achieve such a uniformity. For example, we give in Theorem \ref{approximation} that
\begin{equation*}
	x''_{\nu, k}=F(\nu,k)+O\left((\nu+k)^{-1}\right)
\end{equation*}
in certain range of $x''_{\nu,k}$, where $F$ is a fixed function homogeneous of degree one. Theorem \ref{approximation} contains results in all ranges of $x''_{\nu,k}$. See Theorem \ref{thm4.11} for an asymptotics of the zeros of the derivative of the ultraspherical Bessel function with uniform error terms.

Besides asymptotics, we also obtain upper and lower bounds of zeros. See Propositions \ref{case0}, \ref{prop2}, and \ref{prop4.10}. In addition, we provide in Section \ref{sec3} the number of the zeros of the cross-product of derivatives of ultraspherical Bessel functions within a large circle and show that all its zeros are real and simple.

Apart from the difficulty of approximating eigenvalues, another difficulty of extending planar results (like those in \cite{GMWW:2019}) to high dimensional ones lies in handling the varying multiplicities of eigenvalues. Our resolution to this (as did in \cite{Guo}) is to transfer the multiplicities to weights of lattice points correspondingly, that is, different lattice points may be counted for different numbers of times. As a result, we have to deal with certain weighted planar lattice point counting problems. We then solve them by decomposing them into finitely many standard lattice point counting problems without weights but associated with planar domains of decreasing sizes. For details, please refer to Section \ref{reduction-sec}.

The novel exponent $2\theta^*$ in Theorem \ref{specthm} arises from the application  of the latest development in the Gauss circle and Dirichlet divisor problems, achieved by Li and the last author in \cite{LY2023}, to the lattice point counting problems encountered in  Section \ref{sec5}. Inspired by new ideas presented by Bourgain in \cite{Bourgain:2017} and by Bourgain and Watt in \cite{BW2018, BWpreprint}, Li and the last author combined recent advancements of the decoupling theory, made by Guth and Maldague \cite{GM:2022}, with results on some diophantine counting problems to improve results on the first spacing problem of the circle and divisor problems. Furthermore, by incorporating Huxley's work in \cite{Huxley:2003} on the second spacing problem, they obtained their improved exponential sum estimates in \cite[Theorem 4.2]{LY2023}.  See Section \ref{sec6} for more elaboration. Based on their work, we have formulated an estimate for the rounding error sums, which holds within a limited range but under weaker assumptions. This result, presented independently in Section \ref{sec6}, is particularly applicable to our problems.

%%%%%%%%%%%%%%%%%%%%%%%%%%%%%%%%%%%%%%%%%%%%%%%%%%%%%%%%%%%%%%%%%%%%%%%%%%%%%%%

\emph{Notations:} For functions $f$ and $g$ with $g$ taking nonnegative real values,
$f\lesssim g$ means $|f|\leqslant Cg$ for some constant $C$. If $f$
is nonnegative, $f\gtrsim g$ means $g\lesssim f$. The notation
$f\asymp g$ means that $f\lesssim g$ and $g\lesssim f$. If we write a subscript (for instance $\lesssim_{\sigma}$), we emphasize that the implicit constant depends on that specific subscript. We set $\mathbb{Z}_+:=\mathbb{N}\cup\{0\}$.

%%%%%%%%%%%%%%%%%%%%%%%%%%%%%%%%%%%%%%%%%%%%%%%%%%%%%%%%%%%%%%%%%%%%%%%%%%%%%%%%%%%%%%%%%%%%%%%
%%%%%%%%%%%%%%%%%%%%%%%%%%%%%%%%%%%%%%%%%%%%%%%%%%%%%%%%%%%%%%%%%%%%%%%%%%%%%%%%%%%%%%%%%%%%%%%
%%%%%%%%%%%%%%%%%%%%%%%%%%%%%%%%%%%%%%%%%%%%%%%%%%%%%%%%%%%%%%%%%%%%%%%%%%%%%%%%%%%%%%%%%%%%%%%%%%%%%%%%%%%%%%%%%%%%%%

\section{Zeros of cross-products of Bessel functions}\label{zeros}

Let $0<r<R<\infty$ be two given numbers. For any $\nu\geq 0$ we would like to study positive zeros of cross-product combinations of Bessel functions
\begin{equation}
\mathfrak{f}_{\nu}(x):=J_{\nu}(Rx)Y_{\nu}(rx)-J_{\nu}(rx)Y_{\nu}(Rx),  \label{eigenequation}
\end{equation}
\begin{equation}
\mathfrak{g}_{\nu}(x):=J_{\nu}'(Rx)Y_{\nu}'(rx)-J_{\nu}'(rx)Y_{\nu}'(Rx)  \label{eigenequation2}
\end{equation}
and
\begin{equation}
\mathfrak{h}_{\nu,\delta}(x):=j_{\nu}'(Rx)y_{\nu}'(rx)-j_{\nu}'(rx)y_{\nu}'(Rx),   \label{eigenequation1}
\end{equation}
where $J_{\nu}$ and $Y_{\nu}$ are Bessel functions of the first and second kind of order ${\nu}$,
\begin{equation*}
j_{\nu}(x)=j_{\nu,\delta}(x):=x^{-\delta}J_{\nu}(x)
\end{equation*}
and
\begin{equation*}
y_{\nu}(x)=y_{\nu,\delta}(x):=x^{-\delta}Y_{\nu}(x)
\end{equation*}
with $\delta\in\mathbb{R}$ and $\nu\geq |\delta|$.\footnote{See Remark \ref{rm111} for the reason why we only consider the case $\nu\geq |\delta|$.}  In particular when $\delta=0$ the functions $j_{\nu}$ and $y_{\nu}$ coincide with the Bessel functions hence $\mathfrak{g}_{\nu}=\mathfrak{h}_{\nu,0}$; when $\delta=d/2-1$ and $\nu=n+d/2-1$ with integer $n\geq 0$ and dimension $d\geq 3$, the functions $j_{\nu}$ and $y_{\nu}$ are ultraspherical Bessel functions of the first and second kind that we will deal with in the eigenvalue counting problems.

The motivation for studying the cross-product $\mathfrak{h}_{\nu,\delta}$ (resp., $\mathfrak{f}_{\nu}$) lies in the Neumann (resp., Dirichlet) Laplacian on spherical shells. In this section, we primarily focus on the study of $\mathfrak{h}_{\nu,\delta}$  (including $\mathfrak{g}_{\nu}$), as $\mathfrak{f}_{n}$ with integer $n$ has already been investigated in \cite{GMWW:2019}, and the generalization of those results from $\mathfrak{f}_{n}$ to $\mathfrak{f}_{\nu}$ is essentially the same in nature. For completeness, we still list below results for $\mathfrak{f}_{\nu}$, though without proofs.

One main goal of this section is to find approximations of positive zeros of $\mathfrak{f}_{\nu}(x)$ and $\mathfrak{h}_{\nu,\delta}(x)$ with uniform error terms, which are vital in our study of the two-term Weyl's law. To achieve this goal, we will put much effort into establishing asymptotics of the aforementioned cross-products. The desired approximations will be presented in Theorem \ref{approximation}. We will also provide approximations with nonuniform error terms, upper and lower bounds and so on.

The study of the Dirichlet/Neumann Laplacian on balls is relatively easier. One needs to investigate the zeros of the derivative of the first-kind ultraspherical Bessel function. Analogous results will be presented in Section  \ref{subsec4.2}.

Throughout this paper we denote
\begin{equation*}
g(x)=\left(\sqrt{1-x^2}-x\arccos x\right)/\pi,    
\end{equation*}
\begin{equation*}
G(x)=\left\{
\begin{aligned}
&Rg(x/R)-rg(x/r)\;\; &\mathrm{for}&\;0\leq x\leq r,\\
&Rg(x/R)\;\; &\mathrm{for}&\;r\leq x\leq R,
\end{aligned}
\right.
\end{equation*}
and
\begin{equation*}
\mathcal{G}_\nu(x)=x G(\nu/x).
\end{equation*}
These functions naturally arise in the asymptotics of Bessel functions and their cross-products, respectively. See Lemma \ref{app-1} and the following lemmas.

%%%%%%%%%%%%%%%%%%%%%%%%%%%%%%%%%%%%%%%%%%%%%%%%%%%%%%%%%%%%%%%%%%%%%%%%%%%%%%%%%%%%%%%%%%%%%%%%%%%%%%%%%%%

\subsection{Asymptotics of cross-products}

We first study $\mathfrak{f}_{\nu}(x)$ and $\mathfrak{g}_{\nu}(x)$, then $\mathfrak{h}_{\nu,\delta}(x)$ (based on results of $\mathfrak{g}_{\nu}(x)$).

\begin{lemma}\label{case111}
For any $c>0$ and all $\nu\ge0$, if $rx\geq \max\{(1+c)\nu, 10\}$ then
\begin{equation}\label{case111-1}
\mathfrak{f}_{\nu}(x)=-\frac{2}{\pi}\frac{\sin\left( \pi \mathcal{G}_\nu(x)\right)+O_c\left(x^{-1}\right)}{\left(\left(Rx\right)^2-\nu^2\right)^{1/4}
\left(\left(rx\right)^2-\nu^2\right)^{1/4}}
\end{equation}
and
\begin{equation}\label{case111-1NC}
\mathfrak{g}_{\nu}(x)=-\frac{2}{\pi Rr}\frac{\sin\left( \pi \mathcal{G}_\nu(x)\right)+O_c\left(x^{-1}\right)}{x^2\left(\left(Rx\right)^2-\nu^2\right)^{-1/4}
\left(\left(rx\right)^2-\nu^2\right)^{-1/4}}.
\end{equation}
\end{lemma}

\begin{proof}
We apply  the asymptotics of Bessel functions from Lemma \ref{app-1}  to all factors in $\mathfrak{f}_{\nu}(x)$ and $\mathfrak{g}_{\nu}(x)$, and subsequently utilize the angle difference formula for the sine function.
\end{proof}

%%%%%%%%%%%%%%%%%%%%%      case 2       %%%%%%%%%%%%%%%%%%%%%%%%%%%%%%%%%%%%%%%%%%%%%

\begin{lemma}\label{case222}
There exists a constant $c\in (0,1)$ such that for any $\varepsilon>0$ and all sufficiently large $\nu$, if $\nu+\nu^{1/3+\varepsilon}\leq rx\leq (1+c)\nu$ then
\begin{equation}
\mathfrak{f}_{\nu}(x)=-\frac{2}{\pi}\frac{\sin\left( \pi \mathcal{G}_\nu(x)\right)+O\left(z^{-3/2}\right)}{\left(\left(Rx\right)^2-\nu^2\right)^{1/4}
\left(\left(rx\right)^2-\nu^2\right)^{1/4}}           \label{case222-1}
\end{equation}
and
\begin{equation} \label{case222-1NC}
\mathfrak{g}_{\nu}(x)=-\frac{2}{\pi Rr}\frac{\sin\left( \pi \mathcal{G}_\nu(x)\right)+O\left(z^{-3/2}\right)}{x^2 \left(\left(Rx\right)^2-\nu^2\right)^{-1/4}
\left(\left(rx\right)^2-\nu^2\right)^{-1/4}},
\end{equation}
where $z$ is determined by the equation $rx=\nu+z \nu^{1/3}$.
\end{lemma}

\begin{proof}
Notice that $Rx>rx\geq \nu+\nu^{1/3+\varepsilon}$ implies that $Rx\geq (1+c')\nu$ with some constant $c'>0$. If $c$ is small then $\nu/rx$ is close to $1$,
\begin{equation*}
rxg\left(\frac{\nu}{rx}\right)\asymp rx\left( 1-\frac{\nu}{rx}\right)^{3/2}=\left(\frac{\nu}{rx} \right)^{1/2}z^{3/2} \asymp z^{3/2}\geq \nu^{\frac{3}{2}\varepsilon}
\end{equation*}
and
\begin{equation*}
x^{-1}\lesssim z^{-3/2}.
\end{equation*}

Applying Lemma \ref{app-1} to  $J_{\nu}(Rx)$, $Y_{\nu}(Rx)$, $J_{\nu}'(Rx)$ and $Y_{\nu}'(Rx)$,
Lemma \ref{app-2} to $J_{\nu}(rx)$, $Y_{\nu}(rx)$, $J_{\nu}'(rx)$ and $Y_{\nu}'(rx)$ and then the angle difference formula readily yields the desired asymptotics.
\end{proof}

%%%%%%%%%%%%%%%%%%%%%%%%%%%%%%%%   case  2.5   %%%%%%%%%%%%%%%%%%%%%%%%%%%%%%%%

\begin{lemma} \label{case2.5}
There exist strictly decreasing real-valued $C^1$ functions $\psi_i$: $\mathbb{R} \rightarrow (0, 1/4)$, $i=1,2$, such that $\psi_i(0)=1/12$, $\lim_{x\rightarrow -\infty}\psi_i(x)=1/4$, $\lim_{x\rightarrow \infty}\psi_i(x)=0$ and the images of $\psi'_i$ are bounded intervals. For any $\varepsilon>0$ and all sufficiently large $\nu$, if $\nu-\nu^{1/3+\varepsilon}\leq rx \leq \nu+\nu^{1/3+\varepsilon}$ then
\begin{equation}
\mathfrak{f}_\nu(x)=-\frac{2^{5/6}}{\pi^{1/2}}\frac{\sin\left(\pi \mathcal{G}_\nu(x)+\pi \psi_1\left(z\right)\right)+O\left(\nu^{-2/3+2.5\varepsilon}\right)}{\nu^{1/3}\left(\left(Rx\right)^2-\nu^2\right)^{1/4}
\left(\mathrm{Ai}^2+\mathrm{Bi}^2\right)^{-1/2}\left(-2^{1/3}z\right)}\label{case2.5-1}
\end{equation}
and
\begin{equation}\label{case2.5-1NC}
\mathfrak{g}_\nu(x)\!=-\frac{2^{7/6}}{\pi^{1/2}}\frac{\sin\left(\pi \mathcal{G}_\nu(x)-\pi \psi_2\left(z\right)\right)+O\left(\nu^{-2/3+2.75\varepsilon}\right)}{\nu^{2/3}Rx\left(\left(Rx\right)^2-\nu^2\right)^{-1/4}
\left(\mathrm{Ai'}^2+\mathrm{Bi'}^2\right)^{-1/2}\left(-2^{1/3}z\right)},
\end{equation}
where $z$ is determined by the equation $rx=\nu+z \nu^{1/3}$.
\end{lemma}

\begin{proof}
We only prove the $\mathfrak g_{\nu}$ part; for the $\mathfrak{f}_{\nu}$ part see \cite[Lemma 2.3]{GMWW:2019}.

Notice that $Rx>rx\geq \nu-\nu^{1/3+\varepsilon}$ implies $Rx>(1+c')\nu$ for some constant $c'>0$ whenever $\nu$ is sufficiently large. Denote
\begin{equation*}
rx=\nu+z\nu^{1/3} \quad  \textrm{with $-\nu^\varepsilon\leq z\leq \nu^\varepsilon$}.
\end{equation*}
Applying Lemma \ref{app-1} to $J'_\nu(Rx)$ and $Y'_\nu(Rx)$ and
Lemma \ref{9.3.4analogue} to $J'_\nu(rx)$ and $Y'_\nu(rx)$ yields
\begin{align}
\mathfrak{g}_\nu(x)&=-\frac{2^{7/6}\left(\left(Rx\right)^2-\nu^2\right)^{1/4}\sqrt{\mathrm{Ai}'^2+\mathrm{Bi}'^2}(-2^{1/3}z)}
{\pi^{1/2}Rx\nu^{2/3}} \cdot \nonumber \\
       &\bigg[\sin\!\left(\! \pi Rx g\!\left(\frac{\nu}{Rx}\right)\!-\frac{3\pi}{4}\!\right)
       \! \frac{\mathrm{Ai}'}{\sqrt{\mathrm{Ai}'^2+\mathrm{Bi}'^2}}\left(-2^{1/3}z\right)+ \label{equ1NC}\\
       & \ \cos\!\left(\! \pi R x g\!\left(\frac{\nu}{Rx}\right)\!-\frac{3\pi}{4}\!\right)
      \! \frac{\mathrm{Bi}'}{\sqrt{\mathrm{Ai}'^2+\mathrm{Bi}'^2}}\left(-2^{1/3}z\right)\!+O\left(\nu^{-2/3+2.75\varepsilon}\right)\!\bigg], \label{equ2NC}
\end{align}
where we have used facts that $\mathrm{Ai}'^2(x)+\mathrm{Bi}'^2(x)$ has an absolute positive lower bound (see \cite[10.4.10 and 10.4.80]{abram:1972}) and $x\asymp \nu$.

Let $t_0=-\infty$ and $t_m$ ($m\in\mathbb{N}$) be the $m$th zero of the function $\textrm{Ai}'(-x)$. Let
\begin{equation*}
\mathcal{A}(x)=\left\{
\begin{array}{ll}
-(m-2)\pi+\arctan\left( \frac{\mathrm{Bi'}}{\mathrm{Ai'}}(-x)\right),  & \textrm{$x\in (t_{m-1}, t_m)$, $m\in\mathbb{N}$,}\\
-(m-2)\pi-\frac{1}{2}\pi,  & \textrm{$x=t_m$, $m\in\mathbb{N}$,}
\end{array}\right.
\end{equation*}
be a continuous branch of the inverse tangent function $\arctan\left(\frac{\mathrm{Bi'}}{\mathrm{Ai'}}(-x)\right)$ and
\begin{equation*}
\beta(z)=\frac{1}{\pi} \mathcal{A}\left(2^{1/3}z\right).
\end{equation*}
We can then use this function $\beta$ and the angle sum formula to rewrite \eqref{equ1NC} and \eqref{equ2NC} as
\begin{equation}
\left[\sin\left( \pi R x g\left(\frac{\nu}{Rx}\right)+\pi \beta(z)-\frac{3}{4}\pi\right)+O\left(\nu^{-2/3+2.75\varepsilon}\right)\right].\label{case2.5-3NC}
\end{equation}
Set
\begin{equation*}
\psi_2(z)=\left\{ \begin{array}{ll}
-\beta(z)-\frac{2\sqrt{2}}{3\pi}z^{3/2}+\frac{3}{4},  & \textrm{$z\geq 0$},\\
-\beta(z)+\frac{3}{4},        & \textrm{$z\leq 0$}.
\end{array}\right.
\end{equation*}
By rewriting \eqref{case2.5-3NC} with this $\psi_2$ and the function $\mathcal{G}_\nu$ and using the asymptotics
\begin{equation*}
rx g\left(\frac{\nu}{rx}\right)=\frac{2\sqrt{2}}{3\pi}z^{3/2}+O\left(z^{2.5}\nu^{-2/3}\right) \quad \textrm{for $z\geq 0$},
\end{equation*}
we get \eqref{case2.5-1NC}.

It remains to check the properties of $\psi_2$. We first show that $\psi_2'(z)\leq 0$ with the equality holding only at $z=0$. Indeed, by using properties of Airy functions (10.4.1 and 10.4.10 in \cite[P. 446]{abram:1972}) we have
\begin{equation*}
\mathcal{A}'(x)=-\frac{1}{\pi}\frac{x}{(\mathrm{Ai}'^2+\mathrm{Bi}'^2 )(-x)}.
\end{equation*}
Hence $\psi_2'(z)\leq 0$ if $z\leq 0$ while the equality holds whenever $z=0$.
For $z>0$, $\psi_2'(z)<0$ is equivalent to
\begin{equation*}
\pi z^{-1/2} \left(\mathrm{Ai}'^2+\mathrm{Bi}'^2 \right)(-z)>1 \quad \textrm{for all $z>0$}
\end{equation*}
which follows from the following two facts. Firstly,
\begin{equation*}
\lim_{z\rightarrow+\infty} \pi z^{-1/2} \left(\mathrm{Ai}'^2+\mathrm{Bi}'^2 \right)(-z)=1,
 \end{equation*}
which is an easy consequence of the asymptotics of $\mathrm{Ai}'$ and $\mathrm{Bi}'$ (see \cite[P. 449]{abram:1972}). Secondly, for $z>0$
\begin{equation*}
 z^{-1/2}\left(\mathrm{Ai}'^2+\mathrm{Bi}'^2\right)(-z)=\frac{1}{2}\xi\left(J_{2/3}^2+Y_{2/3}^2\right)(\xi), \quad \textrm{with $\xi=\frac{2}{3}z^{3/2}$},
\end{equation*}
is a decreasing function of $z$ (see  \S 7.3 in \cite[P. 342]{olver:1997}). Note that the above identity follows from 9.1.3 in \cite[P. 358]{abram:1972} and 10.4.28 in \cite[P. 447]{abram:1972}.

The continuity of $\psi_2'$ is obvious. The limit of $\psi_2$ at $\infty$ is easy to get by using asymptotics \cite[10.4.81]{abram:1972} of $\mathcal{A}(x)$,  while the limit at $-\infty$ can be obtained by straightforward computation. Since $\psi_2'(z)\rightarrow 0$ as $|z|\rightarrow \infty$, its image must be a bounded interval.
\end{proof}

%%%%%%%%%%%%%%%%%%%%%%%%   case 3 %%%%%%%%%%%%%%%%%%%%%%%%%%%%%%%%%%%%

\begin{lemma}\label{case3}
For any $\varepsilon>0$ and all sufficiently large $\nu$, if $r\nu/R<rx\leq \nu-\nu^{1/3+\varepsilon}$ then
\begin{equation}\label{case3-1}
\mathfrak f_\nu(x)=Y_\nu(rx)\frac{\mathrm{Ai}\!\left(-\left(\frac{3\pi}{2} \mathcal{G}_\nu(x)\right)^{2/3}\right)+O\!\left(\nu^{-4/3}\max\!\left\{1, \mathcal{G}_\nu(x)^{1/6}\right\}\!\right)}
{\left(\left(Rx\right)^2-\nu^2\right)^{1/4}\left(12\pi \mathcal{G}_\nu(x)\right)^{-1/6}}
\end{equation}
and
\begin{equation}\label{case3-1NC}
\mathfrak g_\nu(x)=-2Y'_\nu(rx)\frac{\mathrm{Ai'}\!\left(-\left(\frac{3\pi}{2} \mathcal{G}_\nu(x)\right)^{2/3}\right)
\!+\!O\!\left(\nu^{-2/3}\min\!\left\{1, \mathcal{G}_\nu(x)^{-1/6}\right\}\!\right)}
{Rx\left(\left(Rx\right)^2-\nu^2\right)^{-1/4}\left(12\pi \mathcal{G}_\nu(x)\right)^{1/6}},
\end{equation}
where $Y_\nu(rx)<0$ and  $Y'_\nu(rx)>0$. If we further assume that $\mathcal{G}_\nu(x)>1$, then
\begin{equation}
\mathfrak f_\nu(x)=\sqrt{\frac{2}{\pi}} Y_\nu(rx)\frac{\sin\left( \pi \mathcal{G}_\nu(x)+\frac{\pi}{4}\right)+O\left(\mathcal{G}_\nu(x)^{-1} \right)}{\left(\left(Rx\right)^2-\nu^2\right)^{1/4}}    \label{case3-2}
\end{equation}
and
\begin{equation}
\mathfrak g_\nu(x)=-\sqrt{\frac{2}{\pi}} Y'_\nu(rx)\frac{\sin\left( \pi \mathcal{G}_\nu(x)-\frac{\pi}{4}\right)+O\left(\mathcal{G}_\nu(x)^{-1} \right)}{Rx\left(\left(Rx\right)^2-\nu^2\right)^{-1/4}}.    \label{case3-2NC}
\end{equation}
\end{lemma}

%%%%%%%%%%%%%%%%%%%%%%%%%%%%%%%%%%%%%  v3  %%%%%%%%%%%%%%%%%%%%%%%%%%%%%%%%%%%%%%%%%%%%%%%%%%%%%%%%%%%

\begin{remark}\label{333}
It is trivial to follow from the asymptotics \eqref{case3-1NC} and \eqref{case3-2NC} to get that if  $\mathcal{G}_\nu(x)\leq 1$ then
\begin{equation}\label{case3-1hnuNC}
	\mathfrak g_\nu(x)=-2Y'_\nu(rx)\frac{\mathrm{Ai'}\!\left(-\left(\frac{3\pi}{2} \mathcal{G}_\nu(x)\right)^{2/3}\right)
		\!+\!O\!\left(\nu^{-2/3}\right)}
	{Rx\left(\left(Rx\right)^2-\nu^2\right)^{-1/4}\left(12\pi \mathcal{G}_\nu(x)\right)^{1/6}},
\end{equation}
and, if $\mathcal{G}_\nu(x)>1$ then
\begin{equation}
\mathfrak g_\nu(x)=-\sqrt{\frac{2}{\pi}} Y'_\nu(rx)\frac{\sin\left( \pi \mathcal{G}_\nu(x)-\frac{\pi}{4}\right)+O\left(\max\left\{\mathcal{G}_\nu(x)^{-1}, \nu^{-\frac{2}{3}-\frac{\varepsilon}{2}}\right\} \right)}{Rx\left(\left(Rx\right)^2-\nu^2\right)^{-1/4}}    \label{111}
\end{equation}
Below we will use \eqref{case3-1hnuNC} and \eqref{111} instead of \eqref{case3-1NC} and \eqref{case3-2NC}. In particular the error term of \eqref{111} is trivially weaker than that of \eqref{case3-2NC} but good enough for later application. The reason for this treatment is that we will reduce the study of $\mathfrak{h}_{\nu,\delta}(x)$ to that of $\mathfrak g_\nu(x)$, but the generalization of \eqref{case3-2NC} from $\mathfrak g_\nu(x)$ to $\mathfrak{h}_{\nu,\delta}(x)$ is not valid while that of  \eqref{111} is. We will illustrate this point later.
\end{remark}

\begin{proof}[Proof of Lemma \ref{case3}]
We focus on $\mathfrak{g}_{\nu}$ below; for $\mathfrak{f}_{\nu}$ see \cite[Lemma 2.6]{GMWW:2019}. Notice that
\begin{equation*}
\mathfrak{g}_\nu(x)=Y'_\nu(rx)\left(J'_\nu(Rx)-\frac{J'_\nu(rx)}{Y'_\nu(rx)}Y'_\nu(Rx)\right).
\end{equation*}
We will find the asymptotics of $J'_\nu(Rx)$ and show  $(J'_\nu(rx)/Y'_\nu(rx))Y'_\nu(Rx)$ is relatively small.

For the convenience of using Olver's asymptotics \eqref{jnuse111NC} and \eqref{ynuse111NC}, we denote
\begin{equation*}
Rx=\nu z_{R} \quad \textrm{and}\quad  rx=\nu z_{r}.
\end{equation*}
We have two useful estimates. First, since $1<z_R< R/r$, the number $\zeta_R:=\zeta(z_R)$, determined by \eqref{def-zeta1}, is negative such that
\begin{equation*}
0<(-\zeta_R)^{3/2}\lesssim 1.
\end{equation*}
Second, it follows from $r/R< z_r\leq 1-\nu^{-2/3+\varepsilon}$ and \eqref{bound-zeta-} that the number $\zeta_r:=\zeta(z_r)$, determined by \eqref{def-zeta2}, is positive such that
\begin{equation*}
\nu^{-1+1.5\varepsilon}\lesssim \zeta_r^{3/2}\lesssim 1
\end{equation*}
whenever $\nu$ is sufficiently large.

With
\begin{equation*}
\nu^{\varepsilon}\lesssim \nu^{2/3}\zeta_r\lesssim \nu^{2/3},
\end{equation*}
applying \eqref{jnuse111NC}, \eqref{ynuse111NC} and asymptotics of Airy functions with positive arguments (see \cite[P. 448--449]{abram:1972})  yields
\begin{equation*}
J'_\nu(rx)=\left(2\pi\right)^{-1/2}(rx)^{-1}\left(\nu^2-(rx)^2\right)^{1/4}e^{-\frac{2}{3}\nu\zeta_r^{3/2}} \left(1+O\left(\nu^{-1}\zeta_r^{-3/2}\right)\right)
\end{equation*}
and
\begin{equation*}
Y'_\nu(rx)=\left(2/\pi\right)^{1/2}(rx)^{-1}\left(\nu^2-(rx)^2\right)^{1/4}e^{\frac{2}{3}\nu\zeta_r^{3/2}} \left(1+O\left(\nu^{-1}\zeta_r^{-3/2}\right)\right).
\end{equation*}
Thus  $Y'_\nu(rx)>0$ and
\begin{equation*}
\frac{J'_\nu(rx)}{Y'_\nu(rx)}=\frac{1}{2}e^{-\frac{4}{3}\nu\zeta_r^{3/2}}\left(1+O\left(\nu^{-1}\zeta_r^{-3/2}\right)\right)
=O\left(e^{-\nu^{\varepsilon}}\right).
\end{equation*}
Therefore
\begin{equation*}
\mathfrak{g}_\nu(x)=Y'_\nu(rx)\left(J'_\nu(\nu z_{R})+Y'_\nu(\nu z_{R})O\left(e^{-\nu^{\varepsilon}}\right)\right).
\end{equation*}

We next apply \eqref{jnuse111NC} and \eqref{ynuse111NC} to $J'_\nu(\nu z_{R})$ and $Y'_\nu(\nu z_{R})$ respectively. By using a simple identity
\begin{equation}\label{222}
\nu^{2/3}\left(-\zeta_R\right)=\left(\frac{3\pi}{2} \mathcal{G}_\nu(x)\right)^{2/3},
\end{equation}
we readily obtain the main term in \eqref{case3-1NC}. Since we only have
\begin{equation*}
0<\nu^{2/3}(-\zeta_R)\lesssim \nu^{2/3},
\end{equation*}
the quantity $\nu^{2/3}(-\zeta_R)$ is not necessarily large. To obtain the error term in \eqref{case3-1NC}, we discuss depending on whether $\nu^{2/3}(-\zeta_R)>1$ or not. If it is greater than $1$, using bounds of Airy functions with negative arguments (see \cite[P. 448--449]{abram:1972}) yields a bound $O(\nu^{-2/3}(\nu^{2/3}|\zeta_R|)^{-1/4})$. Otherwise, using trivial bounds of Airy functions yields a bound $O(\nu^{-2/3})$. With these two bounds we obtain \eqref{case3-1NC}.

Applying the asymptotics of $\mathrm{Ai}'(-r)$ to \eqref{case3-1NC}  yields \eqref{case3-2NC}.
\end{proof}

%%%%%%%%%%%%%%%%%%%%%%%%%%%%%%%%%%%%%%%%%%%%%%%%%%%%%%%%%%%%%%%%%%%%%%%%%%%%%%%%%

We now study $\mathfrak{h}_{\nu,\delta}(x)$. Expanding derivatives in $\mathfrak{h}_{\nu,\delta}(x)$ and factoring out $x^{-2\delta}$ yields
\begin{equation}
\mathfrak{h}_{\nu,\delta}(x)=(Rr)^{-\delta}x^{-2\delta}\widetilde{\mathfrak{h}}_{\nu,\delta}(x),                       \label{444}
\end{equation}
where
\begin{equation}
\widetilde{\mathfrak{h}}_{\nu,\delta}(x):=\mathfrak{g}_{\nu}(x)+\mathscr{E}_{\nu,\delta}(x) \label{555}
\end{equation}
with
\begin{align*}
\mathscr{E}_{\nu,\delta}(x)=&\delta^2 (Rr)^{-1}x^{-2}\mathfrak{f}_{\nu}(x)-\delta R^{-1}x^{-1}\left(J_{\nu}(Rx)Y_{\nu}'(rx)-J_{\nu}'(rx)Y_{\nu}(Rx) \right)\\
&-\delta r^{-1}x^{-1}\left(J_{\nu}'(Rx)Y_{\nu}(rx)- J_{\nu}(rx)Y_{\nu}'(Rx)\right).
\end{align*}
It is obvious that positive zeros of $\mathfrak{h}_{\nu,\delta}(x)$ are exactly positive zeros of $\widetilde{\mathfrak{h}}_{\nu,\delta}(x)$.

One can check that the remainder $\mathscr{E}_{\nu,\delta}(x)$ can be absorbed into error terms of asymptotics \eqref{case111-1NC}, \eqref{case222-1NC}, \eqref{case2.5-1NC}, \eqref{case3-1hnuNC} and \eqref{111} of $\mathfrak{g}_{\nu}(x)$. The computation is routine and tedious.
For instance, we provide partial computation related to the term $|J_{\nu}'(Rx)Y_{\nu}(rx)|x^{-1}$. If $x\geq \max\{(1+c)\nu, 10\}$ then $J_{\nu}(x)$, $Y_{\nu}(x)$, $J_{\nu}'(x)$ and $Y_{\nu}'(x)$ are all of size $O(x^{-1/2})$. Under assumptions of Lemma \ref{case111} we have $|J_{\nu}'(Rx)Y_{\nu}(rx)|x^{-1}=O(x^{-2})$, which can be absorbed in the error term $O(x^{-1})$ of the asymptotics \eqref{case111-1NC} of $\mathfrak{g}_{\nu}(x)$.  Under assumptions of Lemma \ref{case3} together with $\mathcal{G}_\nu(x)\leq 1$, we can use \eqref{bound-zeta+} and \eqref{222} to obtain that
\begin{equation*}
	\mathcal{G}_\nu(x)\asymp \nu^{-1/2}(Rx-\nu)^{3/2},
\end{equation*}
$0<Rx-\nu\lesssim \nu^{1/3}$ and $0<rx-r\nu/R\lesssim \nu^{1/3}$. With these estimates it is easy to show that $|J_{\nu}'(Rx)Y_{\nu}(rx)|x^{-1}$ can be absorbed in the error term of \eqref{case3-1hnuNC}. Under assumptions of Lemma \ref{case3} but with $\mathcal{G}_\nu(x)>1$, we have
\begin{equation*}
\left(\left(Rx\right)^2-\nu^2\right)^{-1/4}\frac{|J_{\nu}'(Rx)Y_{\nu}(rx)|}{|Y'_\nu(rx)|}\lesssim \left(\nu^2-(rx)^2\right)^{-1/2}.
\end{equation*}
This may not be majorized by the error $O(\mathcal{G}_\nu(x)^{-1})$ in  \eqref{case3-2NC}. Hence we add its trivial bound $O(\nu^{-2/3-\varepsilon/2})$  into \eqref{case3-2NC} to rewrite it into the form \eqref{111}.

\begin{lemma} \label{777}
The asymptotics \eqref{case111-1NC}, \eqref{case222-1NC}, \eqref{case2.5-1NC}, \eqref{case3-1hnuNC} and \eqref{111} in Lemmas 2.1--2.4 still hold with $\mathfrak{g}_{\nu}(x)$ replaced by $\widetilde{\mathfrak{h}}_{\nu,\delta}(x)$.
\end{lemma}

\begin{remark}
Obviously error terms of asymptotics of $\widetilde{\mathfrak{h}}_{\nu,\delta}(x)$ may also depend on $\delta$.
\end{remark}

Through Lemma \ref{777}, roughly speaking, the study of the positive zeros of $\mathfrak{h}_{\nu,\delta}(x)$  can be reduced to that of $\mathfrak{g}_{\nu}(x)$.

\subsection{Properties of zeros of cross-products} \label{subsec2.2}
We study positive zeros of $\mathfrak{f}_{\nu}$, $\mathfrak{g}_{\nu}$ and $\mathfrak{h}_{\nu,\delta}$ in this subsection. It is well-known that  $\mathfrak{f}_{\nu}$, $\mathfrak{g}_{\nu}$  are even functions whose zeros are all real and simple. See Cochran \cite{cochran:1964}.

For each nonnegative $\nu$, we denote the sequence of positive zeros of $\mathfrak{f}_{\nu}$ by $0<x_{\nu, 1}<\cdots<x_{\nu, k}<\cdots$, and similarly denote positive zeros of $\mathfrak{g}_{\nu}$ by $x'_{\nu, k}$ (with the convention of beginning with $k=0$ rather than with $k=1$ if $\nu>0$). We know that
\begin{equation*}
Rx_{\nu, k}>\nu \textrm{ and }  R x'_{\nu, k}>\nu.
\end{equation*}
The former is an extension of \cite[Lemma 2.5]{GMWW:2019} from integer $n$ to nonnegative $\nu$. The latter is a consequence of \cite[Lemma 5]{Cochran:1966}. In fact we know that $x_{\nu, k}$ and $x'_{\nu, k}$ both go to infinity as $\nu$ (or $k$) goes to infinity. This will be useful later.

\begin{proposition} \label{case0}
For all real $\nu\geq 0$ and integer $k\geq 0$ with $\nu+k\geq 1$, we have
\begin{equation*}
	x_{\nu,k}, x'_{\nu,k}>\frac{1}{R}\sqrt{\nu^2+\pi^2\left(k-\frac{1}{4} \right)^2}.
\end{equation*}
\end{proposition}

\begin{proof}
See \cite[Lemma 2.5]{GMWW:2019} for the proof of the lower bound of $x_{\nu,k}$. Concerning $x'_{0, k}$,  since $J'_0=-J_1$ and $Y'_0=-Y_1$ (\cite[P. 361]{abram:1972}), it is also a zero of $\mathfrak f_1$ which implies
\begin{equation*}
	Rx'_{0,k}>\sqrt{1+\pi^2\left(k-\frac{1}{4}\right)^2}>\sqrt{\pi^2\left(k-\frac{1}{4}\right)^2},
\end{equation*}
as desired. It remains to consider $x'_{\nu, k}$ with $\nu>0$. If $\mathtt{j}'_{\nu,k+1}$ denotes the $(k+1)$-th positive zero of $J'_\nu$, we have
\begin{equation*}
x'_{\nu, k}\geq \mathtt{j}'_{\nu,k+1}/R,
\end{equation*}
as a consequence of results in \cite[Theorem 4]{Cochran:1966} and \cite[P. 38]{Kline:1948}. Note that
\begin{equation*}
\mathtt{j}'_{\nu,k+1}> \mathtt{j}_{\nu,k}
\end{equation*}
for $k\geq 1$ (\cite[P. 370]{abram:1972}) with $\mathtt{j}_{\nu,k}$ the $k$-th positive zero of $J_\nu$. Combining the above two inequalities with McCann's \cite[Corollary, P. 102]{McCann:1977} gives the desired bound for $k\geq 1$. When $k=0$ we have
\begin{equation*}
Rx'_{\nu, 0}\geq \mathtt{j}'_{\nu,1}>\sqrt{\nu(\nu+2)}>\sqrt{\nu^2+\left(\frac{\pi}{4}\right)^2}
\end{equation*}
as desired, where we have used the inequality (3) in \cite[P. 486]{watson:1966} and the assumption on the sum of $\nu$ and $k$.
\end{proof}

%%%%%%%%%%%%%%%%%%%%%%%%%%%%%%%%%%%%%%%%%%%%%%%%%%%%%%%%%%%%%%%%%%%%%%%%%%%%%%%%%%%%%%%%

The following propositions reveal certain correspondence between the zero $x_{\nu,k}$ ($x'_{\nu,k}$) and the index $k$.

\begin{proposition} \label{thm111}
There exists a constant $c\in (0,1)$ such that for any $\varepsilon>0$ and all sufficiently large $\nu$ the positive zeros of $\mathfrak{f}_{\nu}$ and $\mathfrak{g}_\nu$ satisfy the following
\begin{equation}\label{thm111-1}
\mathcal{G}_{\nu}(x_{\nu,k})=
\left\{\begin{array}{ll}
\!\! k+O(x_{\nu,k}^{-1}),   &   \textrm{if $rx_{\nu,k}\ge (1+c)\nu$,}\\
\!\! k+O(z_{\nu,k}^{-\frac{3}{2}}),  &\textrm{if $\nu+\nu^{\frac{1}{3}+\varepsilon} \le rx_{\nu,k}<(1+c)\nu$,}\\
\!\! k-\psi_1(z_{\nu,k})+O(\nu^{-\frac{2}{3}+3\varepsilon}),& \textrm{if $\nu-\nu^{\frac{1}{3}+\varepsilon}\!<\!rx_{\nu,k}<\nu+\nu^{\frac{1}{3}+\varepsilon}$,}\\
\!\! k-\frac{1}{4}+E_{\nu,k},& \textrm{if $rx_{\nu,k}\le \nu-\nu^{\frac{1}{3}+\varepsilon}$,}
\end{array}\right.
\end{equation}
and
\begin{equation}\label{thm111-1NC}
\mathcal{G}_{\nu}(x'_{\nu,k})=
\left\{\begin{array}{ll}
\!\! k+O((x'_{\nu,k})^{-1}),   &  \textrm{if $rx'_{\nu,k}\ge (1+c)\nu$,}\\
\!\! k+O((z'_{\nu,k})^{-\frac{3}{2}}),  &\textrm{if $\nu+\nu^{\frac{1}{3}+\varepsilon} \!\le\! rx'_{\nu,k}\!<\!(1+c)\nu$,}\\
\!\! k+\psi_2(z'_{\nu,k})+O(\nu^{-\frac{2}{3}+3\varepsilon}),   &\textrm{if $\nu-\nu^{\frac{1}{3}+\varepsilon}\!<\!rx'_{\nu,k}\!<\!\nu+\nu^{\frac{1}{3}+\varepsilon}$,}\\
\!\! k+\frac{1}{4}+E'_{\nu,k},    & \textrm{if $rx'_{\nu,k}\le \nu-\nu^{\frac{1}{3}+\varepsilon}$,}
\end{array}\right.
\end{equation}
where $z_{\nu,k}$ and $z'_{\nu,k}$ are determined by $rx_{\nu,k}=\nu+z_{\nu,k} \nu^{1/3}$ and $rx'_{\nu,k}=\nu+z'_{\nu,k} \nu^{1/3}$, $\psi_1$ and $\psi_2$ are the functions appearing in Lemma \ref{case2.5}, and remainders $E_{\nu,k}$ and $E'_{\nu,k}$ satisfy
\begin{equation*}
|E_{\nu,k}|<\min\left\{3/8, O\left(\mathcal{G}_{\nu}(x_{\nu,k})^{-1}\right)\right\},\quad k\in\mathbb{N},
\end{equation*}
\begin{equation*}
|E'_{\nu,0}|<1/8,
\end{equation*}
%%%%%%%%%%%%%%%%%%%%%%%%%%%%%%%%%   v3 %%%%%%%%%%%%%%%%%%%%%%%%%%%%%%%%%%%%%%%%
 %%%%%%%%%%%%%%%%%%%%%%%%%%%%%%%%%%%%%%%%%%%%%%
\begin{equation*}
|E'_{\nu,k}|<\min\left\{3/8, O\left(\mathcal{G}_{\nu}(x'_{\nu,k})^{-1}+\nu^{-\frac{2}{3}-\frac{\varepsilon}{2}}\right)\right\}, \quad k\in\mathbb{N}.
\end{equation*}
\end{proposition}

\begin{proof}
%%%%%%%%%%%%%%%%%%%%%%%%%%%%%%%%%%%%%%%%% NC boundary condition%%%%%%%%%%%%%%%%%%%%%%%%%%%%%%%%%%%%%%%%%%%%%%%%%%%%%%%%%%%%%%%%%%%%%%%%%%%%%%%%%%%%%%%%

We only prove the $\mathfrak{g}_\nu$ part. It is not hard to check that the function
\begin{equation*}
\mathcal{G}_\nu: [\nu/R, \infty)\rightarrow [0, \infty)
\end{equation*}
is continuous, strictly increasing and mapping $(\nu/R, (s+1/2)\pi/(R-r))$ onto $(0, s+1/2+O(\nu^{-1}))$ for any integer $s>\nu^3$. Therefore for each integer $0\leq k\leq s$ there exists an interval $(a_k, b_k)\subset (\nu/R, (s+1/2)\pi/(R-r))$ such that $\mathcal{G}_\nu$ maps $(a_0, b_0)$ to $(1/6, 3/8)$ and maps $(a_k, b_k)$ to $(k-1/8, k+3/8)$ for $1\le k\le s$ bijectively. All these intervals $(a_k, b_k)$'s are clearly disjoint.

We claim that if $\nu$ is sufficiently large then for each $0\leq k\leq s$
\begin{equation}
\mathfrak g_\nu(a_k)\mathfrak g_\nu(b_k)<0.\label{IVT-conditionNC}
\end{equation}
Assuming \eqref{IVT-conditionNC}, we know that there exists at least one zero of $\mathfrak g_\nu$ in each $(a_k, b_k)$. On the other hand Cochran's \cite[Theorem on P. 583]{cochran:1964} shows that there are exactly $2s+2$ zeros of $\mathfrak g_\nu$ within the disk $B(0,(s+1/2)\pi/(R-r))$ if $s$ is sufficiently large. Since $\mathfrak g_\nu$ is even, we conclude that $\mathfrak g_\nu$ has one and only one zero in each $(a_k, b_k)$, which is $x'_{\nu,k}$ by definition.

At each $x'_{\nu,k}\in (a_k, b_k)$ we have $\mathfrak g_\nu(x'_{\nu,k})=0$. We thus have $\mathcal{G}_\nu(x'_{\nu,k})\in (\mathcal{G}_\nu(a_k), \mathcal{G}_\nu(b_k))$ which implies $|\mathcal{G}_\nu(x'_{\nu,0})-1/4|<1/8$ and $|\mathcal{G}_\nu(x'_{\nu,k})-(k+1/4)|<3/8$ for $k\in\mathbb{N}$. In particular we have $\mathcal{G}_\nu(x'_{\nu,k})>7/8$. We apply either \eqref{case111-1NC},  \eqref{case222-1NC}, \eqref{case2.5-1NC} or \eqref{111}, and conclude that each factor involving both the sine function and $\mathcal{G}_\nu(x'_{\nu,k})$ equals zero. Since $\mathcal{G}_\nu(x'_{\nu,k})-\Delta$ is contained in the interval $[k-3/8, k+3/8]$ for any $0\leq \Delta\leq 1/4$, we apply the arcsine function to obtain all desired asymptotics in \eqref{thm111-1NC}.

It remains to verify \eqref{IVT-conditionNC}. When $(a_k, b_k)$ is contained in $\{x>\nu/R : \mathcal{G}_\nu(x)>C+0.5\}$ with a sufficiently large constant $C\in\mathbb{N}$, we can readily verify \eqref{IVT-conditionNC} by using the asymptotics \eqref{case111-1NC},  \eqref{case222-1NC}, \eqref{case2.5-1NC} or \eqref{111} since
\begin{equation*}
	\mathcal{G}_\nu(a_k)-\Delta_1 \!\in\! \left[k-\frac{3}{8}, k-\frac{1}{8}\right]
	\textrm{ and }\mathcal{G}_\nu(b_k)-\Delta_2\!\in \!\left[k+\frac{1}{8}, k+\frac{3}{8}\right]
\end{equation*}
for all $0\leq \Delta_1, \Delta_2\leq 1/4$. When $(a_k, b_k)$ is contained in $\{x>\nu/R : \mathcal{G}_\nu(x)\leq C+0.5\}$,  we use the asymptotics \eqref{case3-1hnuNC}. The sign of $\mathfrak g_\nu$ depends on that of
\begin{equation}
\mathrm{Ai}'\left(-\left(3\pi \mathcal{G}_\nu(x)/2 \right)^{2/3}\right)+O_C\left(\nu^{-2/3}\right). \label{theorem-1NC}
\end{equation}
If $k=0$ then
\begin{equation*}
\mathrm{Ai}'\left(-\left(3\pi \mathcal{G}_\nu(a_0)/2 \right)^{2/3}\right)=\mathrm{Ai}'\left(-\left(\pi /4 \right)^{2/3}\right)<0 
\end{equation*}
but
\begin{equation*}
\mathrm{Ai}'\left(-\left(3\pi \mathcal{G}_\nu(b_0)/2 \right)^{2/3}\right)=\mathrm{Ai}'\left(-\left(9\pi /16 \right)^{2/3}\right)>0.  
\end{equation*}
If $k\ge 1$,
as in the proof of Lemma \ref{case2.5}, we denote by $t_k$ the $k$th zero of the function $\textrm{Ai}'(-x)$. One can derive that
\begin{equation*}
t_k=\left[\frac{3\pi}{2}\left(k-\frac{3}{4}+\beta'_k\right)\right]^{2/3}
\end{equation*}
with $|\beta'_k|<0.05$ by using results in \cite[P. 214 \& 405]{olver:1997}). Therefore
\begin{align}
t_k\in &\left(\left[\frac{3\pi}{2}\left(k-0.8\right)\right]^{2/3}, \left[\frac{3\pi}{2}\left(k-0.7\right)\right]^{2/3} \right) \nonumber \\
&\subsetneq \left(\left[\frac{3\pi}{2}\mathcal{G}_\nu(a_k)\right]^{2/3}, \left[\frac{3\pi}{2}\mathcal{G}_\nu(b_k)\right]^{2/3} \right). \label{theorem-2NC}
\end{align}
Since $\textrm{Ai}'(-x)$ oscillates around zero for positive $x$ and the intervals in \eqref{theorem-2NC} are pairwise disjoint, the signs of \eqref{theorem-1NC} at $x=a_k$ and $b_k$ must be opposite whenever $\nu$ is sufficiently large, which ensures \eqref{IVT-conditionNC} in this case.
\end{proof}

%%%%%%%%%%%%%%%%%%%%%%%%%%%%%%%%%% theorem  for small n %%%%%%%%%%%%%%%%%%%%%%%%%%%%%%%%%%%%%%%%%%%%%%%%%%%%5

For relatively small $\nu$ we have the following.

\begin{proposition} \label{thm222}
For any $V\in\mathbb{N}$ there exists a constant $K>0$ such that if $0\leq \nu\leq V$ and $k\geq K$ then the positive zeros  of $\mathfrak f_\nu$ and $\mathfrak g_\nu$ satisfy
\begin{equation}
\mathcal{G}_\nu(x_{\nu,k})=k+O\left(x_{\nu,k}^{-1}\right) \label{thm222-2}
\end{equation}
and
\begin{equation}\label{thm222-2NC}
\mathcal{G}_\nu(x'_{\nu,k})=k+O\left((x'_{\nu,k})^{-1}\right).
\end{equation}
\end{proposition}

\begin{proof}
We only prove the $\mathfrak{g}_\nu$ part. Note that the asymptotics \eqref{case111-1NC} of $\mathfrak g_\nu$ is valid for all nonnegative $\nu$. For $0\leq \nu\leq V$ and $x>C_V$ for a sufficiently large constant $C_V$ we apply \eqref{case111-1NC} with $c=1$ and its error term $O_c(x^{-1})$ less than $1/100$, in particular on the interval
\begin{equation}
\left[\frac{(k-1/2)\pi}{R-r}, \frac{(k+1/2)\pi}{R-r}\right), \label{thm222-1NC}
\end{equation}
for any sufficiently large $k$. We observe that the interval $(a_k, b_k)$ (appearing in the proof of Proposition \ref{thm111}) is a subinterval of \eqref{thm222-1NC} if $k$ is sufficiently large. Indeed, this follows from $\mathcal{G}_\nu((a_k, b_k))=(k-1/8, k+3/8)$ and  $\mathcal{G}_\nu((k\pm 1/2)\pi/(R-r))=k\pm 1/2+O(V^2/k)$.

With \eqref{case111-1NC} it is obvious that $\mathfrak g_\nu(a_k)\mathfrak g_\nu(b_k)<0$. Hence $\mathfrak g_\nu$ has at least one zero in $(a_k, b_k)$. On the other hand, Cochran \cite{cochran:1964} tells us that there  exists at most one zero in the interval \eqref{thm222-1NC}. Thus $\mathfrak g_\nu$ has exactly one zero in $(a_k, b_k)$, which is $x'_{\nu,k}$. Thus
\begin{equation*}
\sin\left( \pi\mathcal{G}_\nu(x'_{\nu,k})\right)+O\left((x'_{\nu,k})^{-1}\right)=0.
\end{equation*}
Applying the arcsine function  yields the desired result.
\end{proof}

\begin{remark}
A by-product of the above argument is the following bounds: if $k$ is sufficiently large then
\begin{equation*}
	\frac{\pi(k-1/2)}{R-r}<\mathcal{G}^{-1}_\nu\left(k-\frac{3}{8}\right)<x_{\nu,k}<\mathcal{G}^{-1}_\nu\left(k+\frac{1}{8}\right)<\frac{\pi(k+1/2)}{R-r}.
\end{equation*}
and
\begin{equation*}
\frac{\pi(k-1/2)}{R-r}<\mathcal{G}^{-1}_\nu\left(k-\frac{1}{8}\right)<x'_{\nu,k}<\mathcal{G}^{-1}_\nu\left(k+\frac{3}{8}\right)<\frac{\pi(k+1/2)}{R-r}.
\end{equation*}
In particular, this result is better than Proposition \ref{case0} as $k$ goes to infinity.
\end{remark}

By Theorem \ref{thm444} (that will be proved independently in Section \ref{sec3}), for each fixed real $\delta$ and all $\nu\geq |\delta|$, zeros of $\widetilde{\mathfrak{h}}_{\nu,\delta}$ are all real and simple. We denote positive zeros of $\widetilde{\mathfrak{h}}_{\nu,\delta}$ (which are exactly positive zeros of $\mathfrak{h}_{\nu,\delta}$) by $x''_{\nu, k}$ (with the convention of beginning with $k=0$ if $\nu>|\delta|$ and with $k=1$ if $\nu=|\delta|$).

The following two propositions reveal correspondence between the zero $x''_{\nu,k}$ and the index $k$. Their proofs rely on Theorem \ref{thm333} and are essentially the same as those of Propositions \ref{thm111} and \ref{thm222} since $\mathfrak{g}_{\nu}$ and $\widetilde{\mathfrak{h}}_{\nu,\delta}$ have formally the same asymptotics (by  Lemma \ref{777}).

\begin{proposition}\label{estofhx''}
There exists a constant $c\in (0,1)$ such that for any $\varepsilon>0$, $\delta\in\mathbb{R}$ and all sufficiently large $\nu$,
\begin{equation}\label{thm111-1hnuNC}
\mathcal{G}_{\nu}(x''_{\nu,k})=
\left\{\begin{array}{ll}
				\!\! k+O((x''_{\nu,k})^{-1}),   &  \!\textrm{if $rx''_{\nu,k}\ge (1+c)\nu$,}\\
				\!\! k+O((z''_{\nu,k})^{-\frac{3}{2}}),  & \!\textrm{if $\nu+\nu^{\frac{1}{3}+\varepsilon} \!\le\! rx''_{\nu,k}\!<\!(1+c)\nu$,}\\
				\!\! k+\psi_2(z''_{\nu,k})+O(\nu^{-\frac{2}{3}+3\varepsilon}),   & \! \textrm{if $\nu-\nu^{\frac{1}{3}+\varepsilon}\!<\!rx''_{\nu,k}\!<\!\nu+\nu^{\frac{1}{3}+\varepsilon}$,}\\
				\!\! k+\frac{1}{4}+E''_{\nu,k},    & \textrm{if $rx''_{\nu,k}\le \nu-\nu^{\frac{1}{3}+\varepsilon}$,}
\end{array}\right.
\end{equation}
where $z''_{\nu,k}$ is determined by $rx''_{\nu,k}=\nu+z''_{\nu,k} \nu^{1/3}$, $\psi_2$ is the function appearing in Lemma \ref{case2.5}, and the remainder $E''_{\nu,k}$ satisfies $|E''_{\nu,0}|<1/8$ and
\begin{equation*}
				|E''_{\nu,k}|<\min\left\{3/8, O\left(\mathcal{G}_{\nu}(x''_{\nu,k})^{-1}+\nu^{-\frac{2}{3}-\frac{\varepsilon}{2}}\right)\right\}, \quad k\in \mathbb{N}.
\end{equation*}
\end{proposition}

\begin{proposition}\label{prop1}
For any $\delta\in\mathbb{R}$ and $V\in\mathbb{N}$ there exists a constant $K>0$ such that if $|\delta|\leq \nu\leq V$ and $k\geq K$ then
	\begin{equation}\label{thm222-2hnuNC}
		\mathcal{G}_{\nu}(x''_{\nu,k})=k+O\left((x''_{\nu,k})^{-1}\right).
	\end{equation}
\end{proposition}

As a by-product of the argument, we have the following.
\begin{proposition}\label{prop2}
We have the following bounds for $x''_{\nu,k}$.
\begin{enumerate}
\item If  $\nu\geq|\delta|$, $k\geq 2$ and $\max\{\nu, k\}$ is sufficiently large, then
\begin{equation*}
	x''_{\nu,k}>\frac{1}{R}\sqrt{\nu^2+\pi^2\left(k-\frac{5}{4} \right)^2}.
\end{equation*}

\item	If $\nu\geq|\delta|$ and $\nu$ is sufficiently large, then
\begin{equation*}
	x''_{\nu,k}>\nu/R.
\end{equation*}	

\item If $\nu\geq|\delta|$ and  $k$ is sufficiently large, then
\begin{equation*}
	\frac{\pi(k-1/2)}{R-r}<\mathcal{G}^{-1}_\nu\left(k-\frac{1}{8}\right)<x''_{\nu,k}<\mathcal{G}^{-1}_\nu\left(k+\frac{3}{8}\right)<\frac{\pi(k+1/2)}{R-r}.
\end{equation*}
\end{enumerate}
\end{proposition}

\begin{proof}
We observe that
\begin{equation*}
x''_{\nu,k}>x'_{\nu,k-1}
\end{equation*}
since $x''_{\nu,k}\in(a_k, b_k)$ and $x'_{\nu,k-1}\in (a_{k-1}, b_{k-1})$ where $\mathcal{G}_\nu((a_k, b_k))=(k-1/8, k+3/8)$.  The first inequality follows immediately from Proposition \ref{case0}. The second inequality holds since $x''_{\nu,k}>a_k>a_0>\nu/R$. The third inequality  holds since the interval $(a_k, b_k)$ is a subset of the interval \eqref{thm222-1NC} if $k$ is sufficiently large.
\end{proof}

%%%%%%%%%%%%%%%%%%%%%%%%%%%%%%%%%%%%%%%%%%%%%%%%%%%%%%%%%%%%%%%%%%%%%%%%%%%%%%%%%%%%%%%%%%%%%%%%%%%%%%%%%%%

As a consequence of the six propositions above, we can readily derive the following bounds on the gap between adjacent zeros of $\mathfrak f_{\nu}$, $\mathfrak g_{\nu}$ and  $\mathfrak{h}_{\nu,\delta}$ respectively.

\begin{corollary}\label{cor1}
Given any sufficiently large $\nu$ and any constant $\sigma$ with $0<\sigma<R$, for all $x_{\nu,k}$'s that are greater than $\nu/\sigma$ we have
\begin{equation*}
1 \lesssim x_{\nu,k+1}-x_{\nu,k}\lesssim_{\sigma} 1,
\end{equation*}
and the dependence of the implicit constant on $\sigma$ can be removed if $0<\sigma\leq r$. For any $V\in\mathbb{N}$ if $0\leq \nu\leq V$ and $k$ is sufficiently large then
\begin{equation*}
x_{\nu,k+1}-x_{\nu,k}\asymp 1.
\end{equation*}

The same type of results also hold for $x'_{\nu,k}$ and $x''_{\nu,k}$.
\end{corollary}

We have the same type of results as those stated in \cite[Corollary 2.11]{GMWW:2019}.

\begin{corollary}\label{cor2}
The error terms in \eqref{thm111-1} are of size
\begin{align*}
&O\left(\frac{1}{\nu+k}\right) \quad \textrm {if $rx_{\nu,k}\ge (1+c)\nu$};\\
&O\left(\nu^{1/2}\left(k-\frac{G(r)}{r}\nu\right)^{-3/2}\right)\quad \textrm {if $\nu+\nu^{1/3+\varepsilon}\le rx_{\nu,k}< (1+c)\nu$};\\
& O\left(\nu^{-2/3+3\varepsilon}\right) \quad \textrm {if $\nu-\nu^{1/3+\varepsilon}< rx_{\nu,k}< \nu+\nu^{1/3+\varepsilon}$};\\
&O\left(\frac{1}{k}\right)\quad \textrm {if $rx_{\nu,k}\le\nu-\nu^{1/3+\varepsilon}$}.
\end{align*}
The error terms in \eqref{thm111-1NC} and \eqref{thm111-1hnuNC} have the same bounds as above except that
\begin{equation*}
E'_{\nu,k}, E''_{\nu,k}=O\left(\frac{1}{k+1}+\nu^{-\frac{2}{3}-\frac{\varepsilon}{2}}\right),\quad k\geq 0.
\end{equation*}
The error terms in \eqref{thm222-2}, \eqref{thm222-2NC} and \eqref{thm222-2hnuNC} are all of size $O\left((\nu+k)^{-1}\right)$.
\end{corollary}

\begin{remark}\label{cor2-3}
The second bound is small when $\nu$ is large, since
\begin{equation*}
	\nu^{-1/3}\left(k-\frac{G(r)}{r}\nu\right)\asymp z_{\nu,k}\geq\nu^{\varepsilon}.
\end{equation*}
\end{remark}

Let us now discuss the approximations of zeros. We already know that
 approximations of zeros $x_{\nu,k}$ and $x'_{\nu,k}$ can be derived directly  from  \cite[9.5.27--9.5.31 on P. 374]{abram:1972}, provided that we allow the implicit constants in the error terms to depend on $\nu$. Here we provide  an analogous one-term asymptotics for the zeros $x''_{\nu,k}$, which is an easy consequence of Propositions \ref{estofhx''} and \ref{prop1} combined with a Taylor expansion of the function $G$ at $0$.

\begin{theorem} \label{thm2.20}
	For any fixed $\nu\geq |\delta|$ and sufficiently large $k$, we have
	\begin{equation*}
		x''_{\nu,k}=\frac{\pi}{R-r}k+O_{\nu}\left(\frac{1}{k} \right).
	\end{equation*}
\end{theorem}

\begin{remark}
	See Theorem \ref{thm4.12} below for a similar result of the zeros of $j'_{\nu,\delta}$ and also Remark \ref{remark4.13} for a discussion on generalizing this type of one-term asymptotics to McMahon-type asymptotics.
\end{remark}

However, the type of approximations discussed above is insufficient for our applications to the eigenvalue counting problems. It is crucial for us to derive approximations where the error terms include implicit constants that are independent of both $\nu$ and $k$. Indeed, we have obtained such approximations for zeros of $\mathfrak{f}_{\nu}$ and $\mathfrak{h}_{\nu,\delta}$ in the following theorem. This result generalizes \cite[Corollary 2.14]{GMWW:2019}.

Let $F: [0, \infty)\times [0, \infty)\setminus \{O\}\rightarrow \mathbb{R}$ be the function homogeneous of degree $1$ satisfying $F\equiv1$ on the graph of $G$, that is, $F$ is the Minkowski functional of the graph of $G$. Recall that we have computed and estimated two partial derivatives of $F$ in \cite[Lemma 2.13]{GMWW:2019}.

\begin{theorem}\label{approximation}
	There exists a constant $c\in (0,1)$ such that for any $\varepsilon>0$ there exists a positive integer $V$ such that if $\nu>V$ then the positive zeros of $\mathfrak f_\nu$ satisfy
	\begin{equation}\label{approximation1}
		x_{\nu, k}=F(\nu,k-\tau_{\nu,k})+R_{\nu,k},
	\end{equation}
	where
	\begin{equation*}
		\tau_{\nu,k}=\left\{
		\begin{array}{ll}
			0,                         & \textrm{if $rx_{\nu,k}\geq \nu+\nu^{1/3+\varepsilon}$,}\\
			\psi_1\left(z_{\nu,k}\right),  & \textrm{if $\nu-\nu^{1/3+\varepsilon}< rx_{\nu,k}< \nu+\nu^{1/3+\varepsilon}$,}\\
			1/4,                       & \textrm{if $rx_{\nu,k}\leq \nu-\nu^{1/3+\varepsilon}$,}
		\end{array}
		\right. 
	\end{equation*}
and
	\begin{equation*}
		R_{\nu,k}\!=\!\left\{
		\begin{array}{ll}
			\!O\left((\nu+k)^{-1}\right),                                         & \textrm{if $rx_{\nu,k}\geq (1+c)\nu$,}\\
			\!O\left(\nu^{1/2}\left(k-\frac{G(r)}{r}\nu\right)^{-3/2}\right),  & \textrm{if $\nu+\nu^{1/3+\varepsilon}\leq rx_{\nu,k}<(1+c)\nu$,}\\
			\!O\left(\nu^{-2/3+3\varepsilon}\right),                       & \textrm{if $\nu-\nu^{1/3+\varepsilon}< rx_{\nu,k}< \nu+\nu^{1/3+\varepsilon}$,}\\
			\!O\left(\nu^{1/3}k^{-4/3}\right),                                   &\textrm{if $rx_{\nu,k}\leq \nu-\nu^{1/3+\varepsilon}$.}
		\end{array}
		\right.
	\end{equation*}
If $0\leq \nu\leq V$, there exists a positive integer $K$ such that if $k>K$ then \eqref{approximation1} holds with
	\begin{equation*}
		\tau_{\nu,k}=0
	\end{equation*}
	and
	\begin{equation*}
		R_{\nu,k}=O\left((\nu+k)^{-1}\right).
	\end{equation*}
	
%%%%%%%%%%%%%%%%%%%%%%%%%%%%%%%%%%%%%%%%%%%%%%%%%%%%%%%%%%%%%%%%%%%%%%%%%%%%%%%%5
	
Similar results hold for $\mathfrak{h}_{\nu,\delta}$. There exists a constant $c\in (0,1)$ such that for any $\varepsilon>0$  there exists an integer $V>|\delta|$ such that if $\nu>V$ then the positive zeros of $\mathfrak{h}_{\nu,\delta}$ satisfy
	\begin{equation}\label{approximation1NC}
		x''_{\nu, k}=F(\nu,k+\widetilde\tau_{\nu,k})+\widetilde R_{\nu,k},
	\end{equation}
	where
	\begin{equation}\label{translationNC}
		\widetilde\tau_{\nu,k}=\left\{
		\begin{array}{ll}
			0,                         & \textrm{if $rx''_{\nu,k}\geq \nu+\nu^{1/3+\varepsilon}$,}\\
			\psi_2\left(z''_{\nu,k}\right),  & \textrm{if $\nu-\nu^{1/3+\varepsilon}< rx''_{\nu,k}< \nu+\nu^{1/3+\varepsilon}$,}\\
			1/4,                       & \textrm{if $rx''_{\nu,k}\leq \nu-\nu^{1/3+\varepsilon}$,}
		\end{array}
		\right.
	\end{equation}	
and
	\begin{equation*}
		\widetilde R_{\nu,k}\!=\!\left\{
		\begin{array}{ll}
			\!O\left((\nu+k)^{-1}\right),                                         & \textrm{if $rx''_{\nu,k}\geq (1+c)\nu$,}\\
			\!O\left(\nu^{1/2}\left(k-\frac{G(r)}{r}\nu\right)^{-3/2}\right),  & \textrm{if $\nu+\nu^{1/3+\varepsilon}\leq rx''_{\nu,k}<(1+c)\nu$,}\\
			\!O\left(\nu^{-2/3+3\varepsilon}\right),                       & \textrm{if $\nu-\nu^{1/3+\varepsilon}< rx''_{\nu,k}< \nu+\nu^{1/3+\varepsilon}$,}\\
			\!O\left(\frac{\nu^{1/3}}{(k+1)^{4/3}}+\frac{\nu^{-1/3-\varepsilon/2}}{(k+1)^{1/3}}\right),                                   &\textrm{if $rx''_{\nu,k}\leq \nu-\nu^{1/3+\varepsilon}$.}
		\end{array}
		\right.
	\end{equation*}
If $|\delta|\leq \nu\leq V$, there exists  a positive integer $K$ such that if $k>K$ then \eqref{approximation1NC} holds with
\begin{equation*}
		\widetilde \tau_{\nu,k}=0
\end{equation*}
	and
	\begin{equation*}
		\widetilde R_{\nu,k}=O\left((\nu+k)^{-1}\right).
	\end{equation*}
\end{theorem}

\begin{remark}
We define $\tau_{\nu,k}=\widetilde\tau_{\nu,k}=1/4$ when $\nu\leq V$ and $k\leq K$. This will be used in the next section.
\end{remark}

\begin{proof}[Proof of Theorem \ref{approximation}]
If $\nu$ is sufficiently large, then $\nu/x''_{\nu, k}<R$ and
\begin{equation*}
x''_{\nu, k}=F\left(\nu, \mathcal{G}_{\nu}(x''_{\nu,k})\right).
\end{equation*}

If $rx''_{\nu,k}>\nu-\nu^{1/3+\varepsilon}$ then $x''_{\nu,k}>2\nu/(R+r)$ for sufficiently large $\nu$. By using the asymptotics in \eqref{thm111-1hnuNC} and the monotonicity of $\mathcal{G}_{\nu}$, we have
\begin{equation*}
	\frac{k+\widetilde\tau_{\nu,k}}{\nu}\geq \frac{\mathcal{G}_{\nu}(x''_{\nu,k})}{2\nu}\geq \frac{1}{R+r}G\left(\frac{R+r}{2}\right),
\end{equation*}
which, by \cite[Lemma 2.13]{GMWW:2019}, ensures that $\partial_y F(\nu, \theta)\asymp 1$ with $\theta$ between $k+\widetilde\tau_{\nu,k}$ and $\mathcal{G}_{\nu}(x''_{\nu,k})$. Then \eqref{approximation1NC} follows from the mean value theorem, \eqref{thm111-1hnuNC} and Corollary \ref{cor2}.

If $rx''_{\nu,k}\leq \nu-\nu^{1/3+\varepsilon}$ then $x''_{\nu,k}<\nu/r$. In this case we can obtain \eqref{approximation1NC} similarly if we notice that
\begin{equation*}
	\frac{k+\widetilde\tau_{\nu,k}}{\nu}\asymp \frac{\mathcal{G}_{\nu}(x''_{\nu,k})}{\nu}\leq \frac{G(r)}{r}
\end{equation*}
implies, by \cite[Lemma 2.13]{GMWW:2019} no matter whether $\theta/\nu$ is small or not, that
\begin{equation*}
\partial_y F(\nu, \theta)\lesssim \nu^{1/3}\theta^{-1/3}\lesssim \nu^{1/3}(k+1)^{-1/3}
\end{equation*}
with $\theta$ between $k+\widetilde\tau_{\nu,k}$ and $\mathcal{G}_{\nu}(x''_{\nu,k})$.

The case $|\delta|\leq \nu\leq V$ follows easily from the mean value theorem, Proposition \ref{prop1}, Corollary \ref{cor2} and \cite[Lemma 2.13]{GMWW:2019}.

The proof of \eqref{approximation1} is similar.
\end{proof}

%%%%%%%%%%%%%%%%%%%%%%%%%%%%%%%%%%%%%%%%%%%%%%%%%%%%%%%%%%%%%%%%%%%%%%%%%%%%%%%%%%%%%%%%%%%%%%%

\section{Properties of zeros of the cross-product \texorpdfstring{$\widetilde{\mathfrak{h}}_{\nu,\delta}$}{}} \label{sec3}

In this independent section we prove analogous results of Cochran \cite{cochran:1964} for the function $\widetilde{\mathfrak{h}}_{\nu,\delta}$. Recall that $\widetilde{\mathfrak{h}}_{\nu,\delta}(x)$ is defined by \eqref{444} originally for positive argument $x$. However, in view of its expression \eqref{555}, it is natural to extend its domain to the complex plane $\mathbb{C}$.

\begin{theorem} \label{thm333}
	There exists a large constant $C$ such that
	for any $\nu\ge0$ if $s\in\mathbb{N}$ satisfies $s>C(\nu^3+1)$,  then  within the circle  $|z|=(s+1/2)\pi/(R-r)$ the function $\widetilde{\mathfrak{h}}_{\nu,\delta}(z)$
	has $2s+2$ or $2s$ zeros  according to $\nu\neq |\delta|$ or $\nu=|\delta|$.
\end{theorem}

\begin{proof}
We follow Cochran's strategy in \cite{cochran:1964} to prove this theorem. We first observe that: if $\nu\neq |\delta|$ then $\widetilde{\mathfrak{h}}_{\nu,\delta}$  is analytic in $\mathbb{C}\setminus\{0\}$ with a pole at $0$ of the second order;  if $\nu=|\delta|$ then it is an entire function with $\widetilde{\mathfrak{h}}_{\nu,\delta}(0)\neq 0$. Indeed, the regularity of the function $\widetilde{\mathfrak{h}}_{\nu,\delta}$ in $\mathbb{C}\setminus\{0\}$  follows from the known regularity of the four cross-products of Bessel functions appearing in its definition (see \cite[P. 580]{cochran:1964} and \cite[P. 699]{Horsley}). Its behaviour about the origin is clear  if we write it into a Laurent series. For instance, for all non-integer $\nu\geq 0$ we have
 \begin{align*}
  \widetilde{\mathfrak{h}}_{\nu,\delta}(z)=\frac{\left(\nu^2-\delta^2\right)\left(R^{2\nu}-r^{2\nu}\right)}{\nu\pi r^{\nu+1}R^{\nu+1}}\frac{1}{z^2}
     +C_{\widetilde{\mathfrak{h}}_{\nu,\delta}}+\cdots
  \end{align*}
  with a constant term
  \begin{align*}
    C_{\widetilde{\mathfrak{h}}_{\nu,\delta}}= &\frac{\left(\nu^2-2\nu-\delta^2+2\delta\right)\left(R^{2\nu-2}-r^{2\nu-2}\right)}{4\nu(\nu-1)\pi r^{\nu-1}R^{\nu-1}}\\
    &+\frac{\left(-\nu^2-2\nu+\delta^2-2\delta\right)\left(R^{2\nu+2}-r^{2\nu+2}\right)}{4\nu(\nu+1)\pi r^{\nu+1}R^{\nu+1}},
  \end{align*}
by using the ascending series 9.1.10 and relations 9.1.2 and 9.1.27 in \cite[P. 358--361]{abram:1972}. The proofs for all integer $\nu\geq 0$ are similar, except that in addition to 9.1.10 and 9.1.27, we also need to use 9.1.5 and 9.1.11--9.1.13 in \cite{abram:1972}.

We next derive an asymptotics of $\widetilde{\mathfrak{h}}'_{\nu,\delta}/\widetilde{\mathfrak{h}}_{\nu,\delta}$ on the circle $\sigma:|z|=(s+1/2)\pi/(R-r)$ and apply  the argument principle to $\widetilde{\mathfrak{h}}_{\nu,\delta}$ to count its zeros in this circle. Let $s\in\mathbb{N}$ satisfy $s>C(\nu^3+1)$ with $C>0$ to be determined below. We apply Hankel's expansions (13.01) and (13.04) in \cite[P. 266--267]{olver:1997} (with $n=4$) to $\widetilde{\mathfrak{h}}_{\nu,\delta}(z)$ for $-\pi/2\le \arg z\le \pi/2$. In this process we also use relations between $J_{\nu}$, $Y_{\nu}$ and the Hankel functions and their recurrence relations. We then get an asymptotics of $\widetilde{\mathfrak{h}}_{\nu,\delta}$ as follows
\begin{align*}
  \widetilde{\mathfrak{h}}_{\nu,\delta}(z)=&\frac{-2}{\sqrt{Rr}\pi z}\Bigg(\sin((R-r)z)\left(1+\frac{\alpha_2}{z^2}+O\left(\frac{\mu^4+1}{|z|^4}\right)\right)\\
   & -\cos((R-r)z)\left(\frac{\alpha_1}{z}+\frac{\alpha_3}{z^3}+O\left(\frac{\mu^4+1}{|z|^4}\right)\right)\Bigg)
\end{align*}
with $\mu=4\nu^2$ and coefficients $\alpha_l=\alpha_l(\delta,R,r,\mu)=O(\mu^l+1)$ for $l=1,2,3$. By  the evenness of $\widetilde{\mathfrak{h}}_{\nu,\delta}$ the asymptotics of  $\widetilde{\mathfrak{h}}_{\nu,\delta}$ above holds  for all large $|z|$.

A straightforward calculation then shows that, for $z\in\sigma$,
\begin{align*}
\begin{split}
\frac{\widetilde{\mathfrak{h}}'_{\nu,\delta}(z)}{\widetilde{\mathfrak{h}}_{\nu,\delta}(z)}=&(R-r)\Bigg(\cot((R-r)z)+\frac{\alpha_1}{z}\left(\cot^2((R-r)z)+1\right)\\
&+\frac{\alpha_1}{z^2}\left(\left(\frac{1}{R-r}+\alpha_1\right)\cot((R-r)z)+\alpha_1\cot^3((R-r)z)\right)\\
&+\!\frac{1}{z^3}\!\left(\!\left(\!\alpha_3\!-\!\alpha_1\alpha_2\!+\!\alpha_1^3\!+\!\frac{\alpha_1^2}{R-r}\!\right)\!\cot^2((R-r)z)\!+\!\alpha_1^3\cot^4((R-r)z)\!\right)\\
    &+\frac{1}{z^3}\left(\alpha_3-\alpha_1\alpha_2-\frac{2\alpha_2}{R-r}\right)\Bigg)-\frac{1}{z}+O\left(\frac{\mu^4+1}{|z|^4}\right),
\end{split}
\end{align*}
where we have used the fact that $\cot((R-r)z)$ is bounded above and $\sin((R-r)z)$ is bounded away from zero on $\sigma$. We integrate this equality over the contour $\sigma$ and use  the residue theorem to evaluate the right hand side. For example, the first term  $ \cot((R-r)z)$ has poles at $z=m\pi/(R-r)$ of order $1$ with $m=0,\pm1,\pm2,\ldots,\pm s$ within the contour. Thus
\begin{equation*}
\frac{1}{2\pi i}\int_{\sigma}\cot((R-r)z)\,\mathrm dz
=\sum_{m=-s}^s\frac{1}{R-r} =\frac{2s+1}{R-r}.
\end{equation*}
Other parts can be computed similarly. To sum up, we obtain
\begin{equation*}
  \begin{split}
 \frac{1}{2\pi i}\int_{\sigma}\frac{ \widetilde{\mathfrak{h}}'_{\nu,\delta}(z)}{ \widetilde{\mathfrak{h}}_{\nu,\delta}(z)}\,\mathrm dz
   =2s+O\left(\frac{\mu+1}{s}\right)+O\left(\frac{\mu^4+1}{s^3}\right).
\end{split}
\end{equation*}
By the argument principle, the left hand side is equal to the number of zeros of $\widetilde{\mathfrak{h}}_{\nu,\delta}$ in $\sigma$ minus the number of its poles in $\sigma$. If $C$ is sufficiently large, then the right hand side has to be equal to $2s$. Recalling that $\widetilde{\mathfrak{h}}_{\nu,\delta}$ has only one pole at $0$ of the second order if and only if $\nu\neq|\delta|$, we get the desired result of the number of zeros.
\end{proof}

\begin{remark}
As a consequence of Theorem \ref{thm333} and the evenness of $\widetilde{\mathfrak{h}}_{\nu,\delta}(z)$, by using the argument in the proof of Proposition \ref{thm111}, one can show for sufficiently large  $\nu$  that zeros of $\widetilde{\mathfrak{h}}_{\nu,\delta}$ are all real and simple, and that positive zeros are all greater than $\nu/R$.  One cannot show for relatively small $\nu$ the same properties by a similar argument, since the asymptotics \eqref{case111-1NC} provides little information for  small $x$.

However, by using Theorem \ref{thm333} and the study on ultraspherical Bessel functions in Filonov, Levitin, Polterovich and Sher \cite{FLPS:2024}, we manage to show, even for small $\nu$, that $\widetilde{\mathfrak{h}}_{\nu,\delta}(z)$ has only real and simple zeros .
\end{remark}

\begin{theorem} \label{thm444}
For all $\delta\in\mathbb{R}$ and $\nu\geq |\delta|$, zeros of $\widetilde{\mathfrak{h}}_{\nu,\delta}(z)$ are all real and simple.
\end{theorem}

\begin{proof}
As treated in \cite[\S 3.2]{FLPS:2024}, we write
\begin{equation*}
j_{\nu}'(x)+iy_{\nu}'(x)=L_{\nu,\delta}(x)\left(\cos (\psi_{\nu,\delta}(x))+i\sin (\psi_{\nu,\delta}(x))\right), \textrm{ $x>0$},
\end{equation*}
with a positive modulus function $L_{\nu,\delta}$ and a continuous real phase function $\psi_{\nu,\delta}$ with the initial condition
\begin{equation*}
\lim_{x\rightarrow 0+}\psi_{\nu,\delta}(x)=\frac{\pi}{2}.
\end{equation*}
Therefore,
\begin{equation*}
\widetilde{\mathfrak{h}}_{\nu,\delta}(x)=-(Rr)^{\delta}x^{2\delta}L_{\nu,\delta}(Rx)L_{\nu,\delta}(rx)\sin\left(\psi_{\nu,\delta}(Rx)-\psi_{\nu,\delta}(rx)\right),\textrm{ $x>0$},
\end{equation*}
with a continuous new phase function $\psi_{\nu,\delta}(Rx)-\psi_{\nu,\delta}(rx)$ such that
\begin{equation*}
	\lim_{x\rightarrow 0+}\left(\psi_{\nu,\delta}(Rx)-\psi_{\nu,\delta}(rx)\right)=0.
\end{equation*}
By \cite[Lemma 3.2]{FLPS:2024},
\begin{equation*}
\psi_{\nu,\delta}(Rx)-\psi_{\nu,\delta}(rx)=\left(s+\frac{1}{2}\right)\pi+O\left(s^{-1}\right) \textrm{ with $x=\frac{(s+1/2)\pi}{R-r}$}.
\end{equation*}
By \cite[Lemma 3.1]{FLPS:2024}, when $\nu=|\delta|$ the phase function is always positive on the positive real line and thus its image of the interval $(0, (s+1/2)\pi/(R-r))$ for any large integer $s$ must contain points $\pi,2\pi,\ldots, s\pi$. When $\nu>|\delta|$, however,  the phase function is negative near $0$ and hence its image contains $0,\pi,2\pi,\ldots, s\pi$. This means that $\widetilde{\mathfrak{h}}_{\nu,\delta}$ has  in the interval  $(0, (s+1/2)\pi/(R-r))$ at least $s$ distinct zeros if $\nu=|\delta|$ and $s+1$ distinct zeros if $\nu>|\delta|$. Note that $\widetilde{\mathfrak{h}}_{\nu,\delta}(z)$ is an even function. The desired result then follows from Theorem \ref{thm333}.
\end{proof}

\begin{remark}\label{rm111}
When $\nu<|\delta|$ we still do not know how to prove that  $\widetilde{\mathfrak{h}}_{\nu,\delta}(z)$ has only real and simple zeros. The above argument shows that $\widetilde{\mathfrak{h}}_{\nu,\delta}(z)$ has at least $2s$ distinct real zeros within the circle $|z|=(s+1/2)\pi/(R-r)$, however, there are $2s+2$ zeros inside by Theorem \ref{thm333}. So two zeros are undetermined.  On the other hand, in applications we will take $\delta=d/2-1$ and $\nu=n+d/2-1$ with integers $d\geq 2$ and $n\geq 0$,  which do satisfy $\nu\geq|\delta|$.

These are reasons why we restrict our study of $\mathfrak{h}_{\nu,\delta}$ and $\widetilde{\mathfrak{h}}_{\nu,\delta}$ only to the case $\nu\geq |\delta|$.
\end{remark}

%%%%%%%%%%%%%%%%%%%      part3   part3    part3    part3 part3 part3 part3   %%%%%%%%%%%%%%%%%%%%%%%%%%%%%%%%%%%%%%%%%%%%%%%%%%%%%%%%%%%%
%%%%%%%%%%%%%%%%%%%%%%%% part3   part3    part3    part3 part3 part3 part3    %%%%%%%%%%%%%%%%%%%%%%%%%%%%%%%%%%%%%%%%%%%%%%%%%%%%%%%%%%

\section{From spectrum counting to lattice counting}\label{reduction-sec}

In Subsection \ref{subsec4.1} we deal with the shell case. We first transfer problems of counting spectrum into  certain problems of counting lattice points with various translations; we next transform the latter problems into problems of counting lattice points with uniform translations.

In Subsection \ref{subsec4.2} we present analogous results for the ball case, which are considerably easier to obtain.

\subsection{The shell case}\label{subsec4.1}

One can check by the standard separation of variables that for $d\geq 2$ the spectrum of the Dirichlet Laplacian associated with the shell $\mathbb{S}$ consists of the numbers $\omega_{n,k}^2$, $n\in\mathbb{Z}_+$, $k\in\mathbb{N}$, where $\omega_{n,k}$'s are positive zeros of the function $\mathfrak f_\nu(x)$ with
\begin{equation*}
\nu=n+\frac{d}{2}-1
\end{equation*}
\textbf{which we denote for short throughout Sections \ref{reduction-sec} and \ref{sec5}}. In other words
\begin{equation}
\omega_{n,k}=x_{\nu,k}.\label{s4-1}
\end{equation}
For each pair $(n,k)$ the number $\omega_{n,k}^2$ appears $m_n^d$ times in the spectrum, where the number $m_n^d$ is defined by
\begin{align*}
m_n^d:
=\left\{\begin{array}{cc}
              1 & \text{if $n=0$}, \\
              d & \text{if $n=1$},\\
              \binom{n+d-1}{d-1}-\binom{n+d-3}{d-1} &\text{if $n\ge 2$},
            \end{array}
\right.
\end{align*}
$m_{-1}^d:=0$ and we follow the convention that $\binom{l}{m}=0$ if $m>l$.
See Gurarie~\cite[\S4.5]{Gur:1992}.

In the Neumann case, the corresponding spectrum consists of the squares of positive zeros of $\mathfrak{h}_{\nu,\delta}$ with
\begin{equation*}
\delta=\frac{d}{2}-1,
\end{equation*}
that is,  it consists of the numbers $\widetilde \omega_{n,k}^2$, $n,k\in \mathbb{Z}_+$, where
\begin{equation}
	\widetilde \omega_{n,k}=x''_{\nu,k}\label{s4-2}
\end{equation}
with an additional definition $\widetilde \omega_{0,0}:=0$. For each pair $(n,k)$ the number $\widetilde \omega_{n,k}^2$ appears $m_n^d$ times in the spectrum.

We will apply to \eqref{s4-1} and \eqref{s4-2} Theorem \ref{approximation} which roughly tells us that $\omega_{n,k}$ and $\widetilde\omega_{n,k}$ correspond to points $(\nu, k-\tau_{\nu,k})$ and $(\nu, k+\widetilde\tau_{\nu,k})$ in $\mu\Omega$ respectively, where $\Omega$ denotes the closed domain bounded by the graph of $G$ and the axes. See Figure \ref{domainOmega}.
\begin{figure}[ht]
	\centering
	\includegraphics[width=0.6\textwidth]{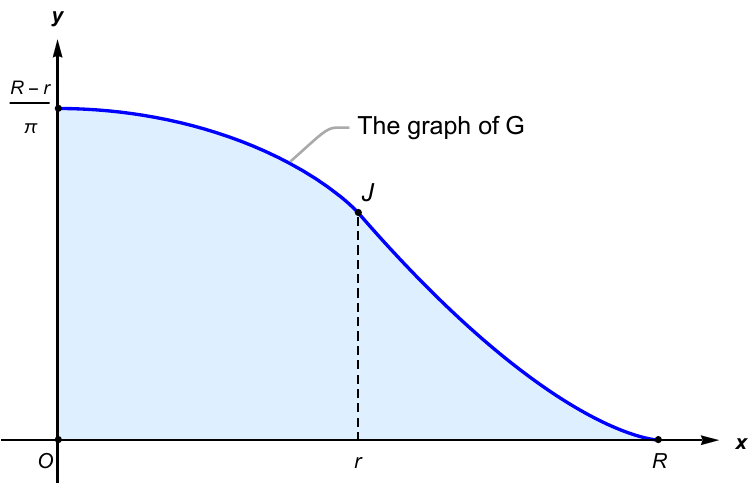}
	\caption{The domain $\Omega$.}
	\label{domainOmega}
\end{figure}

In the Dirichlet case, we write
\begin{equation*}
	\mathscr{N}_\mathbb{S}(\mu)=\mathscr{N}_\mathbb{S}^{\mathtt{D}}(\mu)=\sum_{l=0}^{\infty} \left(m_{l}^d-m_{l-1}^d\right)\mathscr{N}_l^{\mathtt{D}}(\mu),
\end{equation*}
where
\begin{equation*}
	\mathscr{N}_l^{\mathtt{D}}(\mu):=\#\{(n,k)\in \mathbb{Z}_+\times\mathbb{N} : \omega_{n,k}\leq \mu, n\geq l\}.
\end{equation*}
The superscript ``$\mathtt{D}$'' represents Dirichlet while ``$\mathtt{N}$''  represents Neumann. Correspondingly in this case we define a weighted lattice point counting function (associated with the domain $\Omega$)
\begin{equation*}
	\mathscr{P}_{\Omega}(\mu)=\mathscr{P}_{\Omega}^{\mathtt{D}}(\mu):=\sum_{l=0}^{\infty}\left(m_{l}^d-m_{l-1}^d\right)\mathscr{P}_l^{\mathtt{D}}(\mu),
\end{equation*}
where
\begin{equation*}
	\mathscr{P}_l^{\mathtt{D}}(\mu):=\#\left\{ \left(\nu, k-\tau_{\nu,k}\right)\in \mu\Omega : n\in \mathbb{Z}_+, n\geq l, k\in\mathbb{N}\right\}.
\end{equation*}

In the Neumann case, we have analogously
\begin{equation*}
\mathscr{N}_\mathbb{S}(\mu)=\mathscr{N}_\mathbb{S}^{\mathtt{N}}(\mu)=\sum_{l=0}^{\infty} \left(m_{l}^d-m_{l-1}^d\right)\mathscr{N}_l^{\mathtt{N}}(\mu),
\end{equation*}
where
\begin{equation*}
	\mathscr{N}_l^{\mathtt{N}}(\mu):=\#\{(n,k)\in \mathbb{Z}_+\times\mathbb{Z}_+ : \widetilde\omega_{n,k}\leq \mu, n\geq l\},
\end{equation*}
and we define
\begin{equation*}
	\mathscr{P}_{\Omega}(\mu)=\mathscr{P}_{\Omega}^{\mathtt{N}}(\mu):=\sum_{l=0}^{\infty}\left(m_{l}^d-m_{l-1}^d\right)\mathscr{P}_l^{\mathtt{N}}(\mu),
\end{equation*}
where
\begin{equation*}
	\mathscr{P}_l^{\mathtt{N}}(\mu):=\#\left\{ \left(\nu, k+\widetilde\tau_{\nu,k}\right)\in \mu\Omega : n\in \mathbb{Z}_+, n\geq l, k\in\mathbb{Z}_+\right\}.
\end{equation*}

We observe that the summands in the above four sums are all equal to zero if $l> R\mu-\frac{d-2}{2}$, and that $m_{l}^2-m_{l-1}^2$ equals $1$ if $l=0,1$, and $0$ otherwise, and for $d\geq 3$, $l\in\mathbb{Z}_+$,
\begin{equation}\label{bino-bound}
m_{l}^d-m_{l-1}^d=\binom{l+d-2}{d-2}-\binom{l+d-4}{d-2}  =O_d\left(l^{d-3}+1\right).
\end{equation}

As did in ``the boundary parts'' in \cite[Theorem 3.1]{colin:2011} we can control the difference between the number of eigenvalues and the corresponding number of lattice points.

\begin{lemma}\label{lemma4.1}
There exists a constant $C>0$ such that
\begin{equation*}
\left|\mathscr{N}_l^*(\mu)-\mathscr{P}_l^*(\mu)\right|\leq \mathscr{P}_l^*\left(\mu+C\mu^{-0.4}\right)-\mathscr{P}_l^*\left(\mu-C\mu^{-0.4}\right)+O\left(\mu^{0.6}\right)
\end{equation*}
for all $l\in \mathbb{Z}_+$ and $*\in\{ \mathtt{D},\mathtt{N}\}$.
\end{lemma}

\begin{proof}
We only prove the Neumann case here, as the Dirichlet case is similar. For  $k\in\mathbb{Z}_+$, we define
\begin{equation}
	\mathscr{N}^{\mathtt{N}}_{l,k}(\mu):=\#\{n\in \mathbb{Z}_+ : x''_{\nu,k}\leq \mu, n\geq l\}\label{countfcn1NC}
\end{equation}
and
\begin{equation}
	\mathscr{P}^{\mathtt{N}}_{l,k}(\mu):=\#\left\{ n\in \mathbb{Z}_+ : \left(\nu, k+ \widetilde\tau_{\nu,k}\right)\in \mu\Omega, n\geq l\right\}.\label{countfcn2NC}
\end{equation}
Then $\mathscr{N}^{\mathtt{N}}_l(\mu)$ and $\mathscr{P}^{\mathtt{N}}_l(\mu)$ are sums of \eqref{countfcn1NC} and \eqref{countfcn2NC} respectively over (finitely many) $k\in \mathbb{Z}_+$.  Using Theorem \ref{approximation} and properties of $F$ we get
\begin{align}
&\left|\mathscr{N}^{\mathtt{N}}_{l,k}(\mu)-\mathscr{P}^{\mathtt{N}}_{l,k}(\mu) \right|\leq \#\bigg(\left\{n\in\mathbb{Z}_+ : \max\{n,k\}>A\right\}\cap \nonumber\\
&\quad \left\{n\in\mathbb{Z}_+ : \mu-|\widetilde R_{\nu,k}|\leq F(\nu,k+\widetilde \tau_{\nu,k})\leq\mu+|\widetilde R_{\nu,k}|, n\geq l \right\}\bigg),\label{setNC}
\end{align}
where $A$ is a fixed sufficiently large constant such that when $\max\{n,k\}>A$ the asymptotics of $x''_{\nu,k}$ in Theorem \ref{approximation} applies.

We next use Theorem \ref{approximation} and estimates of derivatives of $F$ (in \cite[Lemma 2.13]{GMWW:2019}) to estimate $|\mathscr{N}^{\mathtt{N}}_{l,k}(\mu)-\mathscr{P}^{\mathtt{N}}_{l,k}(\mu)|$ by estimating sizes of the set in  \eqref{setNC}. We need to determine the form of $\widetilde R_{\nu,k}$ in different ranges of $k$. Note that we always have $\nu\leq R\mu$ by Proposition \ref{prop2}.

If $0\leq k\leq \mu^{1/4}$ then
\begin{equation}
\left|\mathscr{N}^{\mathtt{N}}_{l,k}(\mu)-\mathscr{P}^{\mathtt{N}}_{l,k}(\mu)\right|\lesssim \mu^{1/3}(k+1)^{-4/3}.\label{444NC}
\end{equation}
Indeed, since $k/\mu$ is less than $G(r)$ we have $\nu\asymp \mu$. This, together with Proposition \ref{estofhx''} and Corollary \ref{cor2}, gives
\begin{equation*}
\frac{G(\nu/x''_{\nu,k})}{\nu/x''_{\nu,k}}=\frac{k+O(1)}{\nu}\lesssim \mu^{-3/4},
\end{equation*}
which implies that $\nu/x''_{\nu,k}\rightarrow R-$ as $\mu\rightarrow \infty$  and thus $rx''_{\nu,k}\leq \nu-\nu^{1/3+\varepsilon}$. Therefore $\widetilde R_{\nu,k}=O(\nu^{1/3}(k+1)^{-4/3})$ and $\widetilde \tau_{\nu,k}=1/4$. Using the estimate of $\partial_x F$ we get \eqref{444NC}.

If $\mu^{1/4}< k\leq \mu^{4/7}$, by a similar argument we have $\widetilde R_{\nu,k}=O(\nu^{1/3}k^{-4/3})=O(1)$, $\widetilde \tau_{\nu,k}=1/4$ and thus
\begin{equation*}
\left|\mathscr{N}^{\mathtt{N}}_{l,k}(\mu)-\mathscr{P}^{\mathtt{N}}_{l,k}(\mu) \right|\lesssim 1.
\end{equation*}

If $\mu^{4/7}< k\leq G(r)\mu-C_1$ for a sufficiently large constant $C_1$ then
\begin{equation}\label{666NC}
\left|\mathscr{N}^{\mathtt{N}}_{l,k}(\mu)-\mathscr{P}^{\mathtt{N}}_{l,k}(\mu)\right|\le \mathscr{P}^{\mathtt{N}}_{l,k}(\mu+C\mu^{-3/7})-\mathscr{P}^{\mathtt{N}}_{l,k}(\mu-C\mu^{-3/7})
\end{equation}
for some constant $C$. Indeed, from the definition of the set \eqref{setNC} we observe that the point $(\nu,k+\widetilde\tau_{\nu,k})$ is contained in a tubular neighborhood of $\mu\partial \Omega$ of width much less than $1$ (because of $\nu\asymp \mu$ and Remark \ref{cor2-3}). For large $\mu$ the tubular neighborhood between $y=G(r)\mu$ and $y=G(r)\mu-C_1$ is close to a parallelogram. Hence if $k\leq G(r)\mu-C_1$ for a  sufficiently large $C_1$ then $\nu\geq r\mu$. As a result,
\begin{equation}
\frac{k}{\nu}\leq \frac{G(r)}{r}-\frac{C_1}{\nu}.\label{555NC}
\end{equation}
However, if $rx''_{\nu,k}\geq \nu+\nu^{1/3+\varepsilon}$ then, by Proposition \ref{estofhx''} and the monotonicity of $G$, we have
\begin{equation*}
\frac{k}{\nu}=\frac{G(\nu/x''_{\nu,k})}{\nu/x''_{\nu,k}}+O\left(\nu^{-1-\frac{3}{2}\varepsilon} \right)>\frac{G(r)}{r}+O\left(\nu^{-1-\frac{3}{2}\varepsilon}\right),
\end{equation*}
which contradicts with \eqref{555NC}. Thus $rx''_{\nu,k}< \nu+\nu^{1/3+\varepsilon}$. This implies that $\widetilde R_{\nu,k}$ is either $O(\nu^{-2/3+3\varepsilon})$ or $O(\nu^{1/3}k^{-4/3}+\nu^{-1/3}k^{-1/3})$, both of which are of size $O(\mu^{-3/7})$. We immediately get \eqref{666NC}.

If $G(r)\mu-C_1< k\leq G(r)\mu+\mu^{0.6}$, we still have $\nu\asymp \mu$. By using the trivial estimate $\widetilde R_{\nu,k}=O(1)$, \cite[Lemma 2.13]{GMWW:2019} and the mean value theorem we get
\begin{equation*}
	\left|\mathscr{N}^{\mathtt{N}}_{l,k}(\mu)-\mathscr{P}^{\mathtt{N}}_{l,k}(\mu) \right|\lesssim 1.
\end{equation*}

If $k>G(r)\mu+\mu^{0.6}$ then
\begin{equation}
\left|\mathscr{N}^{\mathtt{N}}_{l,k}(\mu)-\mathscr{P}^{\mathtt{N}}_{l,k}(\mu) \right|\leq \mathscr{P}^{\mathtt{N}}_{l,k}(\mu+C\mu^{-0.4})-\mathscr{P}^{\mathtt{N}}_{l,k}(\mu-C\mu^{-0.4})\label{888NC}
\end{equation}
for some constant $C$. As argued in the proof of \eqref{666NC}, we conclude that $\nu<r\mu$. Then
\begin{equation*}
k-\frac{G(r)}{r}\nu>\mu^{0.6}.
\end{equation*}
With this lower bound, we can show that
\begin{equation}
	\widetilde R_{\nu,k}=O(\mu^{-0.4}),\label{ine1}
\end{equation}
which gives the inequality \eqref{888NC} immediately. To obtain \eqref{ine1}, we first notice that $\widetilde R_{\nu,k}$ has possibly four forms (see Theorem \ref{approximation}). The first one $O((\nu+k)^{-1})=O(\mu^{-1})$ is small enough. The second one is of size $O(\mu^{-0.4})$ by the above lower bound. We next observe that if $\nu/\mu$ is sufficiently small, by Proposition \ref{estofhx''} and the monotonicity of $G$, then $rx''_{\nu,k}\geq (1+c)\nu$ and $\widetilde R_{\nu,k}$ is of the first form. Therefore, it remains to consider the case when the ratio $\nu/\mu\gtrsim 1$. Hence the third and fourth forms of $\widetilde R_{\nu,k}$ are $\lesssim \mu^{-2/3+3\varepsilon}<\mu^{-0.4}$.

Summing the above bounds of $|\mathscr{N}^{\mathtt{N}}_{l,k}(\mu)-\mathscr{P}^{\mathtt{N}}_{l,k}(\mu)|$ over $k$ yields the desired inequality.
\end{proof}

One can then transfer the spectrum counting problem to a lattice point counting problem via the following result.  By using Lemma \ref{lemma4.1} and the bound in \eqref{bino-bound}, we readily get the following in both the Dirichlet and Neumann cases.

\begin{proposition}\label{difference1}
There exists a constant $C>0$ such that
\begin{equation*}
\left|\mathscr{N}_\mathbb{S}(\mu)-\mathscr{P}_{\Omega}(\mu)\right|\leq  \mathscr{P}_{\Omega}(\mu+C\mu^{-0.4})-\mathscr{P}_{\Omega}(\mu-C\mu^{-0.4})+O\left(\mu^{d-2+0.6}\right).
\end{equation*}
\end{proposition}

\begin{remark}
In spectral theory problems of counting eigenvalues are sometimes transformed into problems of counting lattice points. However, the argument in the proof of Lemma \ref{lemma4.1} reminds us that the reverse is also possible. For example, we can derive an inequality using eigenvalue counting to obtain lattice point counting:
\begin{equation*}
\left|\mathscr{N}_\mathbb{S}(\mu)-\mathscr{P}_{\Omega}(\mu)\right|\leq  \mathscr{N}_\mathbb{S}(\mu+C\mu^{-0.4})-\mathscr{N}_\mathbb{S}(\mu-C\mu^{-0.4})+O\left(\mu^{d-2+0.6}\right).
\end{equation*}
This provides us with another possible novel way to study certain specific problems of lattice point counting.
\end{remark}

Notice that $\mathscr{P}_{\Omega}(\mu)$ is about counting lattice points with weights and different translations. We will transfer it into standard lattice point problems with the same translations and meanwhile estimate the differences caused by such transformations. In fact, we will move every point $(\nu, k-\tau_{\nu,k})$ to $(\nu, k-1/4)$ and  every $(\nu, k+\widetilde\tau_{\nu,k})$ to $(\nu, k+1/4)$. Accordingly, we define the following new counting functions: firstly in the Dirichlet case
\begin{equation*}
	\mathscr{Q}_{\Omega}(\mu)=\mathscr{Q}_{\Omega}^{\mathtt{D}}(\mu):=\sum_{l=0}^{\infty}\left(m_{l}^d-m_{l-1}^d\right)\mathscr{Q}_l^{\mathtt{D}}(\mu),
\end{equation*}
where
\begin{equation*}
	\mathscr{Q}_l^{\mathtt{D}}(\mu):=\#\{ (\nu,k-1/4)\in \mu\Omega : n\in \mathbb Z_+,n\ge l,k\in\mathbb{N}\},
\end{equation*}
and secondly in the Neumann case
\begin{equation*}
	\mathscr{Q}_{\Omega}(\mu)=\mathscr{Q}_{\Omega}^{\mathtt{N}}(\mu):=\sum_{l=0}^{\infty}\left(m_{l}^d-m_{l-1}^d\right)\mathscr{Q}^{\mathtt{N}}_l(\mu),
\end{equation*}	
where
\begin{equation*}
	\mathscr{Q}_l^{\mathtt{N}}(\mu):=\#\{ (\nu,k+1/4)\in \mu\Omega : n\in \mathbb Z_+,n\ge l,k\in \mathbb Z_+\}.
\end{equation*}

To quantify the difference between $\mathscr{P}_{\Omega}(\mu)$ and $\mathscr{Q}_{\Omega}(\mu)$, we need to count points in certain bands of width $1/4$. For $0<L\leq R\mu$, we define bands on $[0, L]$ as follows
\begin{equation*}
\mathcal{B}_{L}^{\mathtt{D}}=\left\{(x,y)\in\mathbb{R}^2 : 0\leq x\leq L, \,  \mu G\left(\frac{x}{\mu} \right)< y\leq \mu G\left(\frac{x}{\mu} \right)+\frac{1}{4} \right\}
\end{equation*}
and
\begin{equation*}
\mathcal{B}_{L}^{\mathtt{N}}=\left\{(x,y)\in\mathbb{R}^2 : 0\leq x\leq L,\,\mu G\left(\frac{x}{\mu} \right)-\frac{1}{4}< y\leq \mu G\left(\frac{x}{\mu} \right) \right\}.
\end{equation*}
These bands are formed by translating up/down part of the boundary $\mu\partial\Omega$ by $1/4$. We define associated counting functions
\begin{equation*}
\mathscr{Q}_{\mathcal{B}_{r\mu}^{\mathtt{D}},l}=\# \left\{(\nu,k)\in \mathcal{B}_{r\mu}^{\mathtt{D}}:n\in \mathbb{Z}_+,n\ge l, k\in \mathbb N \right\}
\end{equation*}
and
\begin{equation*}
\mathscr{Q}_{\mathcal{B}_{r\mu}^{\mathtt{N}},l}=\# \left\{(\nu,k)\in \mathcal{B}_{r\mu}^{\mathtt{N}} : n\in \mathbb{Z}_+,n\ge l, k\in \mathbb{Z}_+ \right\}.
\end{equation*}
Both functions equal to zero trivially when $l+d/2-1>r\mu$.

\begin{lemma} \label{difference3}
	\begin{equation*}
	\mathscr{Q}_l^{*}(\mu)=\mathscr{P}_l^{*}(\mu)\pm \mathscr{Q}_{\mathcal{B}_{r\mu}^{*},l}+O(\mu^{1/3+\varepsilon})
	\end{equation*}
	for $l\in \mathbb Z_+$ and $*\in\{ \mathtt{D},\mathtt{N}\}$, where we take the sign ``$+$'' (resp., ``$-$'') when $*=\mathtt{D}$ (resp., $*=\mathtt{N}$).
\end{lemma}

\begin{remark}
From the proof of this lemma, we observe that $\mathscr{Q}_l^{*}(\mu)=\mathscr{P}_l^{*}(\mu)$ if $l+d/2-1>r\mu+C'\mu^{1/3+\varepsilon}$.
\end{remark}

\begin{proof}[Proof of Lemma \ref{difference3}]
We will only give a brief proof of the Neumann case. See \cite[Proposition 3.2]{GMWW:2019} for a proof of the Dirichlet case.

In the process of moving the points $(\nu, k+\widetilde\tau_{\nu,k})$ up to $(\nu, k+1/4)$, some of these points may get out of the domain $\mu\Omega$ but none enters the domain. The difference $\mathscr{P}_l^{\mathtt{N}}(\mu)-\mathscr{Q}_l^{\mathtt{N}}(\mu)$ equals the number of points $(\nu, k+\widetilde\tau_{\nu,k})\in \mathcal{B}_{R\mu}^{\mathtt{N}}$ with $n\geq l$ that leave the domain  $\mu\Omega$.

The points $(\nu, k+\widetilde\tau_{\nu,k})\in \mathcal{B}_{R\mu}^{\mathtt{N}}$ with
$\widetilde\tau_{\nu,k}=0$ are all above the line $OJ$ (see Figure \ref{domainOmega}) because of $k/\nu>G(r)/r$ as a consequence of Proposition \ref{estofhx''} and Corollary  \ref{cor2}. These points get out of $\mu\Omega$ surely.

The points $(\nu, k+\widetilde\tau_{\nu,k})\in \mathcal{B}_{R\mu}^{\mathtt{N}}$ with
$\widetilde\tau_{\nu,k}=1/4$ are all below the line $OJ$ because of  $(k+1/4)/\nu<G(r)/r$ by Proposition \ref{estofhx''} and Corollary  \ref{cor2} again. These points remain unmoved.

The points $(\nu, k+\widetilde\tau_{\nu,k})\in \mathcal{B}_{R\mu}^{\mathtt{N}}$ with
$0<\widetilde\tau_{\nu,k}<1/4$ may get out of $\mu\Omega$. We do not know whether they are above or below the line $OJ$ but we do know that they satisfy $|\nu-r\mu|\leq C'\mu^{1/3+\varepsilon}$  for some large constant $C'$.  Indeed, by Theorem \ref{approximation}, if $\nu-\nu^{1/3+\varepsilon}<rx''_{\nu,k}<\nu+\nu^{1/3+\varepsilon}$ then
\begin{equation*}
x''_{\nu, k}=F(\nu,k+\widetilde\tau_{\nu,k})+O\left(\nu^{-2/3+3\varepsilon}\right)=\mu+O(1),
\end{equation*}
where the last equality follows from the fact that  $(\nu, k+\widetilde\tau_{\nu,k})$ must be  contained in a cone (in the first quadrant with its vertex at the origin) away from the axes (by Proposition \ref{estofhx''}) and hence in $\mu\Omega\setminus (\mu-O(1))\Omega$. Plugging this formula into the above inequality of $x''_{\nu,k}$ yields the desired range of $\nu$.

Based on the above facts, we conclude that the points $(\nu,k+\widetilde\tau_{\nu,k})$ in the band $\mathcal{B}_{r\mu-C'\mu^{1/3+\varepsilon}}^{\mathtt{N}}$ are all with $\widetilde\tau_{\nu,k}=0$; the points $(\nu,k+\widetilde\tau_{\nu,k})$ in $\mathcal{B}_{R\mu}^{\mathtt{N}}\setminus \mathcal{B}_{r\mu+C'\mu^{1/3+\varepsilon} }^{\mathtt{N}}$ are all with $\widetilde\tau_{\nu,k}=1/4$. Some of the points $(\nu,k+\widetilde\tau_{\nu,k})\in \mathcal{B}_{R\mu}^{\mathtt{N}}$ with $|\nu-r\mu|< C'\mu^{1/3+\varepsilon}$ may get out of $\mu\Omega$ but its number is relatively small and of size $O(\mu^{1/3+\varepsilon})$. Collecting all these facts leads to the desired equality.
\end{proof}

As a consequence of Lemma \ref{difference3} and the bound in \eqref{bino-bound}, we obtain the following easily.
\begin{proposition}\label{difference4}
	\begin{equation*}
		\mathscr{P}_{\Omega}^*(\mu)=\mathscr{Q}_{\Omega}^*(\mu)\mp \sum_{0\leq l\leq r\mu-\frac{d-2}{2}}\left(m_l^d-m_{l-1}^d\right)\mathscr{Q}_{\mathcal{B}_{r\mu}^{*},l}+O\left(\mu^{d-2+\frac{1}{3}+\varepsilon}\right),
	\end{equation*}
where $*\in\{ \mathtt{D},\mathtt{N}\}$ and we take the sign ``$-$'' (resp., ``$+$'') when $*=\mathtt{D}$ (resp., $*=\mathtt{N}$).
\end{proposition}

By Propositions \ref{difference1} and \ref{difference4}, in order to obtain an asymptotics of $\mathscr{N}_{\mathbb{S}}^*(\mu)$ with $*\in\{ \mathtt{D},\mathtt{N}\}$, we just need to study lattice point counting functions $\mathscr{Q}_{\Omega}^*(\mu)$ and $\mathscr{Q}_{\mathcal{B}_{r\mu}^{*},l}$. We will do that in Section \ref{sec5}.

%%%%%%%%%%%%%%%%%%%      part3   part3    part3    part3 part3 part3 part3   %%%%%%%%%%%%%%%%%%%%%%%%%%%%%%%%%%%%%%%%%%%%%%%%%%%%%%%%%%%%
%%%%%%%%%%%%%%%%%%%%%%%% part3   part3    part3    part3 part3 part3 part3    %%%%%%%%%%%%%%%%%%%%%%%%%%%%%%%%%%%%%%%%%%%%%%%%%%%%%%%%%%

\subsection{The ball case} \label{subsec4.2}

As in Section \ref{intro}, we let $\mathscr{N}_\mathbb{B}(\mu)$ denote the eigenvalue counting function for the Dirichlet/Neumann Laplacian associated with the ball $\mathbb{B}$.

Let $\Omega_0$ denote the closed domain bounded by the axes and the graph of
\begin{equation*}
G_0(x)=Rg(x/R),\quad 0\leq x\leq R.
\end{equation*}
See Figure \ref{domainOmega0}.
\begin{figure}[ht]
	\centering
	\includegraphics[width=0.6\textwidth]{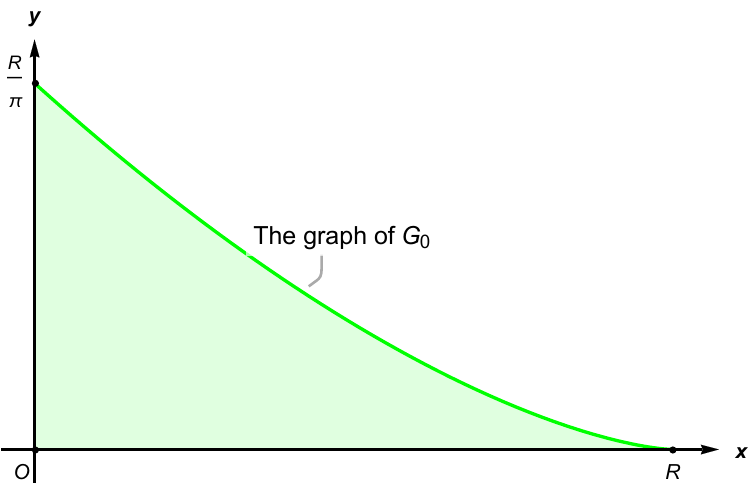}
	\caption{The domain $\Omega_0$.}
	\label{domainOmega0}
\end{figure}
We define a weighted lattice point counting function (associated with the domain $\Omega_0$) corresponding to the Dirichlet  Laplacian
\begin{equation*}
	\mathscr{Q}_{\Omega_0}(\mu)=\mathscr{Q}_{\Omega_0}^{\mathtt{D}}(\mu):=\sum_{l=0}^{\infty}\left(m_{l}^d-m_{l-1}^d\right)\mathscr{Q}_{\Omega_0,l}^{\mathtt{D}}(\mu),
\end{equation*}
where
\begin{equation*}
	\mathscr{Q}_{\Omega_0,l}^{\mathtt{D}}(\mu):=\#\left\{ \left(\nu, k-1/4\right)\in \mu\Omega_0 : n\in \mathbb{Z}_+, n\geq l, k\in\mathbb{N}\right\}.
\end{equation*}
We define an analogous version corresponding to the Neumann  Laplacian
\begin{equation*}
	\mathscr{Q}_{\Omega_0}(\mu)=\mathscr{Q}_{\Omega_0}^{\mathtt{N}}(\mu):=\sum_{l=0}^{\infty}\left(m_{l}^d-m_{l-1}^d\right)\mathscr{Q}_{\Omega_0,l}^{\mathtt{N}}(\mu),
\end{equation*}
where
\begin{equation*}
	\mathscr{Q}_{\Omega_0,l}^{\mathtt{N}}(\mu):=\#\left\{ \left(\nu, k+1/4\right)\in \mu\Omega_0 : n\in \mathbb{Z}_+, n\geq l, k\in\mathbb{Z}_+\right\}.
\end{equation*}

We have the following ``ball'' version of Proposition \ref{difference1}.
\begin{proposition} \label{prop4.7}
	There exists a constant $C>0$ such that
	\begin{equation*}
		\left|\mathscr{N}_\mathbb{B}(\mu)-\mathscr{Q}_{\Omega_0}(\mu)\right|\leq  \mathscr{Q}_{\Omega_0}(\mu+C\mu^{-3/7})-\mathscr{Q}_{\Omega_0}(\mu-C\mu^{-3/7})+O\left(\mu^{d-2+4/7}\right).
	\end{equation*}
\end{proposition}

The Dirichlet case has already been proved in \cite[Theorem 3.2]{Guo}. The proof for the Neumann case follows a similar pattern. Below, we will outline the proof and highlight the key differences. We may assume $R=1$ without loss of generality.

In the Neumann case, we need to study positive zeros of
\begin{equation*}
	j_{\nu,\delta}'(x)=x^{-\delta}\left( J_{\nu}'(x)-\delta x^{-1}J_{\nu}(x) \right) \quad \textrm{with $\delta\in\mathbb{R}$ and $\nu\geq |\delta|$}.
\end{equation*}
The term $\delta x^{-1}J_{\nu}(x)$ can be absorbed into error terms of the asymptotics of  $J_{\nu}'(x)$. To verify this, particularly when $\nu<x<(1+c)\nu$ for any sufficiently small $c>0$, we need to use $xg(\nu/x)\asymp \nu^{-1/2}(x-\nu)^{3/2}$ by \eqref{bound-zeta+}.

\begin{lemma}
	For any $c>0$ and $\nu\ge 0$, if $x\geq \max\{(1+c)\nu, 10\}$ then
	\begin{equation*}
		j_{\nu,\delta}'(x)=-\left(\frac{2}{\pi}\right)^{1/2} \frac{\left(x^2-\nu^2\right)^{1/4}}{x^{1+\delta}} \left(\sin\left( \pi
		xg\left(\frac{\nu}{x}\right)-\frac{\pi}{4}\right)+O_c\left(x^{-1}\right)\right).
	\end{equation*}

For any sufficiently small $c>0$ and sufficiently large $\nu$, if $\nu<x<(1+c)\nu$ then
\begin{equation*}
	j_{\nu,\delta}'(x)=\frac{-2^{2/3}}{(3\pi)^{1/6}}\frac{\left(x^2-\nu^2\right)^{1/4}}{x^{1+\delta}\left(x g\left(\nu/x \right)\right)^{1/6}}
	\left(\!\mathrm{Ai}'\left(\!-\left(\frac{3\pi}{2} x g\left(\frac{\nu}{x} \right)\right)^{\!\frac{2}{3}}\right)+O\left(\nu^{-\frac{2}{3}}\right)\!\right)
\end{equation*}
when $xg(\nu/x)\leq 1$, and
	\begin{equation*}
		j_{\nu,\delta}'(x)=-\left(\frac{2}{\pi}\right)^{\!1/2} \frac{\left(x^2-\nu^2\right)^{1/4}}{x^{1+\delta}} \left(\sin\left( \pi
		xg\left(\frac{\nu}{x}\right)-\frac{\pi}{4}\right)+O\left(\!\left(xg\left(\frac{\nu}{x}\right)\right)^{\!-1}\right)\!\right)
	\end{equation*}
when $xg(\nu/x)>1$.
\end{lemma}

We denote positive zeros of $j_{\nu,\delta}'$  by $a'_{\nu, \delta, k}$ (with the convention of beginning with $k=0$ if $\nu>|\delta|$ and with $k=1$ if $\nu=|\delta|$). By the way, $j_{\nu,\delta}'$ has only real zeros when $\nu\geq |\delta|$ since  $zJ_{\nu}'(z)-\delta J_{\nu}(z)$ has only real zeros (see  Watson \cite[P. 482]{watson:1966}), and all its nonzero zeros are simple by Dixon's theorem (see Watson \cite[P. 480]{watson:1966}).

\begin{proposition} \label{prop4.9}
There exists a constant $c\in (0,1)$ such that for any $\delta\in\mathbb{R}$ and all sufficiently large $\nu$,
	\begin{equation*}
	a'_{\nu, \delta, k}g\left(\frac{\nu}{a'_{\nu, \delta, k}}\right)=k+\frac{1}{4}+R'_{\nu,k},
	\end{equation*}
where the remainder $R'_{\nu,k}$ satisfies $|R'_{\nu,0}|<1/8$,	  $|R'_{\nu,k}|<1/4$, $k\in \mathbb{N}$ and
		\begin{equation*}
		R'_{\nu,k}=
		\left\{\begin{array}{ll}
			O((\nu+k)^{-1}),   &  \textrm{if $a'_{\nu, \delta, k}\ge (1+c)\nu$,}\\
			O((k+1)^{-1}),    & \textrm{if $\nu<a'_{\nu, \delta, k}< (1+c)\nu$.}
		\end{array}\right.
	\end{equation*}

	For any $\delta\in\mathbb{R}$ and $V\in\mathbb{N}$ there exists a constant $K>0$ such that if $|\delta|\leq \nu\leq V$ and $k\geq K$ then
	\begin{equation*}
		a'_{\nu, \delta, k}g\left(\frac{\nu}{a'_{\nu, \delta, k}}\right)=k+\frac{1}{4}+O\left(\frac{1}{\nu+k}\right).
	\end{equation*}
\end{proposition}

The proof of this proposition is merely a repetition of the proofs of Propositions \ref{thm111} and \ref{thm222}, except that it needs the use of the following two facts: if $s\in\mathbb{N}$ is sufficiently large, then in the interval $(0, \pi(s+\frac{1}{2}\nu+\frac{1}{2}))$ the function $j_{\nu,\delta}'(x)$ has $s+1$ or $s$ zeros according to $\nu>|\delta|$ or $\nu=|\delta|$; if $\nu\geq |\delta|$, $k\in\mathbb{N}$ and $\max\{\nu,k\}$ is sufficiently large, then
\begin{equation*}
a'_{\nu, \delta, k}\gtrsim \nu+k.
\end{equation*}
The first fact follows easily from \cite[Lemmas 3.1 and 3.2]{FLPS:2024}. The second one follows from the known lower bound of positive zeros of $J_{\nu}$ (McCann \cite{McCann:1977}) and the fact that $j_{\nu,\delta}$ and $j'_{\nu,\delta}$ have interlacing  positive zeros by Dixon's theorem (see Watson \cite[P. 480]{watson:1966}).

We can then derive the following bounds and approximations of zeros $a'_{\nu, \delta, k}$ like those in  Proposition \ref{prop2} and Theorems \ref{thm2.20} and \ref{approximation}. Proposition \ref{prop4.10} has actually been proved implicitly in the proof of the previous proposition.

\begin{proposition}  \label{prop4.10}
	We have the following bounds for $a'_{\nu, \delta, k}$.
	\begin{enumerate}
		\item If  $\nu\geq|\delta|$ and $k\geq 2$, then
		\begin{equation*}
			a'_{\nu, \delta, k}>\sqrt{\nu^2+\pi^2\left(k-\frac{5}{4} \right)^2}.
		\end{equation*}
		
		\item If $\nu\geq|\delta|$ and  $k$ is sufficiently large, then
		\begin{equation*}
			\pi\left(k+\frac{\nu}{2}-\frac{1}{2}\right)<h^{-1}_\nu\left(k\right)<a'_{\nu, \delta, k}<h^{-1}_\nu\left(k+\frac{3}{8}\right)<\pi\left(k+\frac{\nu}{2}+\frac{1}{2}\right),
		\end{equation*}
		where $h_{\nu}(x)=xg(\nu/x)$.
	\end{enumerate}
\end{proposition}

McMahon gave an asymptotics of $a'_{\nu, \delta, k}$ when $\nu=n+\frac{1}{2}$ and $\delta=\frac{1}{2}$ with $n$ fixed and $k$ large (see \cite[P. 441]{abram:1972}). But we could not find in the literature any generalization of McMahon's expansion  of $a'_{\nu, \delta, k}$ for other values of $\nu$ and $\delta$. Here we give such a generalization, a one-term asymptotics. It follows directly from a Taylor expansion of the function $g$ at $0$ and Proposition \ref{prop4.9}.

\begin{theorem}\label{thm4.12}
	For any fixed $\nu\geq |\delta|$ and sufficiently large $k$, we have
	\begin{equation*}
		a'_{\nu, \delta, k}=\pi \left(k+\frac{1}{2}\nu+\frac{1}{4} \right)+O_{\nu}\left(\frac{1}{k} \right).
	\end{equation*}
\end{theorem}

\begin{remark}\label{remark4.13}
	In fact, we believe that a slight modification of the argument in this paper can give us a McMahon-type asymptotics of $a'_{\nu, \delta, k}$ rather than just a one-term asymptotics.
\end{remark}

Theorem \ref{thm4.11} follows from the mean value theorem and Proposition \ref{prop4.9}.

\begin{theorem} \label{thm4.11}
Let $F_g$ be the Minkowski functional of the graph of $g$. There exists a small constant $c>0$ and an integer $V>|\delta|$ such that if $\nu>V$ then
	\begin{equation}
a'_{\nu, \delta, k}=F_g(\nu,k+1/4)+\mathfrak{R}'_{\nu,k},\label{rem}
	\end{equation}
	where
	\begin{equation*}
		\mathfrak{R}'_{\nu,k}=\left\{
		\begin{array}{ll}
			O\left((\nu+k)^{-1}\right),                                         & \textrm{if $a'_{\nu, \delta, k}\geq (1+c)\nu$,}\\
			O\left(\nu^{1/3}(k+1)^{-4/3}\right),                                   &\textrm{if $a'_{\nu, \delta, k}<(1+c)\nu$.}
		\end{array}
		\right.
	\end{equation*}
If $|\delta|\leq \nu\leq V$ and $k$ is sufficiently large then \eqref{rem} holds with
	\begin{equation*}
		\mathfrak{R}'_{\nu,k}=O\left((\nu+k)^{-1}\right).
	\end{equation*}
\end{theorem}

The spectrum under consideration consists of squares of
\begin{equation*}
\omega'_{n,k}:=a'_{\nu,\delta,k}, \quad n,k\in \mathbb{Z}_+,
\end{equation*}
with $\nu=n+\frac{d}{2}-1$, $\delta=\frac{d}{2}-1$, and an additional definition $\omega'_{0,0}:=0$. For each pair $(n,k)$ the number $(\omega'_{n,k})^2$ appears $m_n^d$ times in the spectrum. Repeating the argument in \cite[Section 3]{Guo} gives the Neumann case of Proposition \ref{prop4.7}.

%%%%%%%%%%%%%%%%%%%%%%%%%%%  part4  part4   %%%%%%%%%%%%%%%%%%%%%%%%%%%%%%%%%%%%%%%%%%%%%%%%%%%%%%%%%%%%%%%%%%%%
%%%%%%%%%%%%%%%%%%%%%%%%%%%  part4  part4   %%%%%%%%%%%%%%%%%%%%%%%%%%%%%%%%%%%%%%%%%%%%%%%%%%%%%%%%%%%%%%%%%%%%

\section{Lattice counting and Weyl's law}\label{sec5}

\subsection{Weyl's law for shells}\label{sec5.1}

The main task in this subsection is to study two associated lattice point problems, $\mathscr{Q}_{\Omega}^*(\mu)$ and $\mathscr{Q}_{\mathcal{B}_{r\mu}^{*},l}$,  defined in Subsection \ref{subsec4.1}. The shell part of Theorem \ref{specthm} follows immediately from Proposition \ref{difference1} and Theorem \ref{thm Ndu}.

We first study the lattice counting associated with both the domain $\Omega$ (when $r>0$; see Figure \ref{domainOmega}) and the domain $\Omega_0$ (when $r=0$; see Figure \ref{domainOmega0}). In the following lemma,  when $r=0$, as a slight abuse of notation, $\Omega$, $G$ and $\mathscr{Q}_l^{*}(\mu)$  represent $\Omega_0$, $G_0$ and $\mathscr{Q}_{\Omega_0,l}^*(\mu)$ respectively. The results for $\Omega_0$ will be used in Subsection \ref{sec5.2}.

Let $d\geq 2$ be a fixed integer. For $l\in\mathbb{Z}_+$ with $l>2-d$, we define
\begin{equation*}
(\mu\Omega)_l=\{(x,y)\in\mu\Omega: x\ge l+(d-3)/2\}
 \end{equation*}
to be a subset of $\mu\Omega$ in $\mathbb{R}^2$; when $d=2$ and $l=0$, we define
$(\mu\Omega)_0$ to be the union of $\mu\Omega$ and a rectangle $[-1/2,0]\times[0,\mu G(0)]$. Then $\mathscr{Q}_l^{\mathtt{D}}(\mu)$ and $\mathscr{Q}_l^{\mathtt{N}}(\mu)$ are the number of points $(\nu,k-1/4)$ and $(\nu,k+1/4)$ in $(\mu\Omega)_l$, respectively.

\begin{lemma}\label{theorem:no-in-D}
For $0\leq r < R$ and $d\ge 2$, we have
\begin{equation}\label{final-statmentNC}
\mathscr{Q}_l^{*}(\mu)=
\left|(\mu\Omega)_l\right|+\left(c_*- \frac{1}{2}\right)\left(R\mu-l-\frac{d-3}{2}\right)+O\left(\mu^{2/3}\right)
\end{equation}
with $0\le l<R\mu-d/2+1$, $*\in\{ \mathtt{D},\mathtt{N}\}$, $c_{\mathtt{D}}=1/4$ and $c_{\mathtt{N}}=3/4$.
If either $l\ge  r\mu-d/2+1$ or the boundary curve of $\Omega$ has  a tangent in $J$ with rational slope (i.e. $\pi^{-1}\arccos(r/R)\in \mathbb{Q}$), the remainder estimate can be improved to
\begin{equation*}
O_{\epsilon}\left(\mu^{2\theta^*+\epsilon}\right).
\end{equation*}
\end{lemma}

\begin{remark}
The remainder $O(\mu^{2/3})$ is a consequence of a standard second derivative estimate of van der Corput. The above-mentioned improved estimate follows from Theorem  \ref{expo sum} which requires the boundedness of certain first derivatives. This condition is not satisfied if we count lattice points near to the boundary point $\mu J$ along lines parallel to the axes. However, this obstacle can be avoided if we count them along lines parallel to the tangent in $J$ with rational slope. See also \cite[Section 4]{GMWW:2019}.
\end{remark}

\begin{proof}[Proof of Lemma \ref{theorem:no-in-D}]
%%%%%%%%%%%%%%%%%%%% Neumann Laplacian  %%%%%%%%%%%%%%%%%%%%%%%%%%%%%%%%%%%%%%

For $l\in \mathbb{Z}_+$ and $c\in[0,1)$, we study the following counting function
\begin{equation*}
\mathscr{Q}_l(\mu):=\#\{(\nu,k-c)\in \mu\Omega : n\in \mathbb Z_+, n\ge l, k\in \mathbb N\},
\end{equation*}
that is, we count the number of points $(\nu,k-c)$ in $\mu \Omega$. In particular $\mathscr{Q}_l(\mu)$ with $c=1/4$ (resp., $3/4$) gives $\mathscr{Q}_l^{\mathtt{D}}(\mu)$ (resp., $\mathscr{Q}_l^{\mathtt{N}}(\mu)$). If we define
\begin{align*}
\Omega_1&:=\{(x,y)\in\Omega: y\leq G(r)\},\\
\Omega_2&:=\{(x,y)\in\Omega: y>G(r)\}
\end{align*}
and
\begin{equation}
\mathscr{Q}_{\Omega_i,l}(\mu):=\#\{(\nu,k-c)\in \mu\Omega_i:n\in \mathbb Z_+,n\ge l,k\in \mathbb N\},\label{countfcn3}
\end{equation}
then
\begin{equation*}
\mathscr{Q}_l(\mu)=\mathscr{Q}_{\Omega_1,l}(\mu)+\mathscr{Q}_{\Omega_2,l}(\mu).
\end{equation*}
We do this splitting because the boundary curve (given by the graph of $G$)  consists of two parts when $r>0$. When $r=0$ we have $\Omega_2=\emptyset$ and hence $\mathscr{Q}_l(\mu)=\mathscr{Q}_{\Omega_1,l}(\mu)$.

We claim that
\begin{equation}
\begin{split}
\mathscr{Q}_{\Omega_1,l}(\mu)=&\left|\{(x,y)\in(\mu \Omega)_l :y\leq  G(r)\mu \}\right|-L_{l}\\
    &+\left(c-\frac{1}{2}\right)\left(R\mu-l-\frac{d-3}{2}\right)+
    	O_{\epsilon}\left(\mu^{2\theta^*+\epsilon}\right),
  \end{split}\label{ND1NC}
\end{equation}
where
\begin{equation*}
	\begin{split}
		L_{l}=\left\{\begin{array}{ll}
			0&\text{if $l+d/2-1\geq r\mu$}, \\
			\psi(G(r)\mu+c)\left(r\mu-l-\frac{d-3}{2}\right) &\text{if $l+d/2-1< r\mu$}.
		\end{array}
		\right.
	\end{split}
\end{equation*}
When $r>0$ we also claim that if  $l+d/2-1\geq r\mu$ then $\mathscr{Q}_{\Omega_2,l}(\mu)=0$; if $l+d/2-1< r\mu$ then
\begin{equation}
\mathscr{Q}_{\Omega_2,l}(\mu) =\left|\{(x,y)\in(\mu \Omega)_l :y> G(r)\mu \}\right|+L_{l}+O\left(\mu^{2/3}\right),\label{ND2NC}
\end{equation}
where the remainder can be improved to $O_{\epsilon}(\mu^{2\theta^*+\epsilon})$ in the rational case (i.e. $\pi^{-1}\arccos(r/R)\in \mathbb{Q}$). It is easy to verify results of this lemma based on these claims.

It remains to prove these claims. We first prove \eqref{ND1NC} when $l+d/2-1<r\mu$. In $\mu\Omega_1$ we count lattice points along lines parallel to the $x$-axis. Then
\begin{equation}
\mathscr{Q}_{\Omega_1,l}(\mu)=\sum_{c<k\leq G(r)\mu +c}\left(\left\lfloor\mu H\left(\frac{k-c}{\mu}\right)-\frac{d}{2}+1\right\rfloor-l+1\right), \label{sum3}
\end{equation}
where $H:[0,G(r)]\to[r,R]$ represents the inverse function of $G$ restricted to $[r,R]$. With
\begin{equation*}
\psi(t)=t-1/2-\lfloor t\rfloor
\end{equation*}
the sawtooth function, $\mathscr{Q}_{\Omega_1,l}(\mu)$ is equal to the difference of
\begin{equation}
\sum_{c<k\leq G(r)\mu +c}\left(\mu H\left(\frac{k-c}{\mu}\right)-l-\frac{d-3}{2}\right) \label{sum1}
\end{equation}
and
\begin{equation}
\sum_{c<k\leq G(r)\mu +c}\psi\left(\mu H\left(\frac{k-c}{\mu}\right)-\frac{d}{2}+1\right). \label{sum2}
\end{equation}

Applying the Euler--Maclaurin summation formula to \eqref{sum1} yields
\begin{align*}
\eqref{sum1}=&\int_1^{G(r)\mu+c}\!\!\! \mu H\left(\frac{y-c}{\mu}\right)-l-\frac{d-3}{2} \textrm{d}y+\int_1^{G(r)\mu+c} \!\!\!\psi(y) H'\left(\frac{y-c}{\mu}\right) \textrm{d}y\\
  &+\frac{1}{2}\left(\mu H\left(\frac{1-c}{\mu}\right)-l-\frac{d-3}{2} \right)-\psi\left(G(r)\mu+c\right)\left(r\mu-l-\frac{d-3}{2} \right).
\end{align*}
By the asymptotics
\begin{equation*}
H(y)=R+O\left(y^{2/3}\right) \textrm{ as $y\rightarrow 0+$},
\end{equation*}
the first integral is equal to
\begin{equation*}
\left|\{(x,y)\in(\mu \Omega)_l :y\leq  G(r)\mu \}\right|+(c-1)\left(R\mu-l-\frac{d-3}{2}\right)+O(\mu^{1/3}).
\end{equation*}
By the second mean value theorem and $H'(y)\asymp y^{-1/3}$ (see \cite[Lemma 4.6]{GMWW:2019}) the second integral is of size $O(\mu^{1/3})$. Hence
\begin{align*}
\eqref{sum1}=&\left|\{(x,y)\in(\mu \Omega)_l :y\leq  G(r)\mu \}\right|+\left(c-\frac{1}{2}\right)\left(R\mu-l-\frac{d-3}{2}\right)\\
&-L_l+O(\mu^{1/3}).
\end{align*}

For the sum \eqref{sum2}, the part with $k\leq \mu^{2\theta^*}$ is of size  $O(\mu^{2\theta^*})$ and the rest part is divided into sums of the form
\begin{equation}
\sum_{M<k\leq M'\leq 2M}\psi\left(NF\left(\frac{k}{M}\right)\right),\label{psisumNC}
\end{equation}
where $M=2^j\mu^{2\theta^*}\lesssim\mu$, $N=M^{2/3}\mu^{1/3}$ and
\begin{equation}
F(x)=\left(\frac{\mu}{M}\right)^{2/3}H\left(\frac{M}{\mu} x-\frac{c}{\mu}\right)-\frac{1}{N}\left(\frac{d}{2}-1\right).\label{F1}
\end{equation}
By Theorem \ref{expo sum} with $T=MN$,
\begin{equation*}
	\eqref{psisumNC}\lesssim_{\epsilon} \left(M^{5/3}\mu^{1/3}\right)^{\theta^*+\epsilon}.
\end{equation*}
Summing over $j$ yields
\begin{equation*}
	\eqref{sum2}\lesssim_{\epsilon} \mu^{2\theta^*+2\epsilon}.
\end{equation*}
Combining the above results of \eqref{sum1} and \eqref{sum2} yields \eqref{ND1NC} when $l+d/2-1<r\mu$.

To prove \eqref{ND1NC} for $l+d/2-1\geq r\mu$, we just need to replace the summation domain in \eqref{sum3} by $c<k\leq G(\frac{l+d/2-1}{\mu})\mu+c$ and argue similarly as above. We omit the details.

If $r>0$ we also need to compute $\mathscr{Q}_{\Omega_2,l}(\mu)$. It is obvious that $\mathscr{Q}_{\Omega_2,l}(\mu)=0$ if $l+d/2-1\geq r\mu$. We will assume that $l+d/2-1<r\mu$ below.

We first prove \eqref{ND2NC} with the error term $O(\mu^{2/3})$.  We now count lattice points along lines parallel to the $y$-axis. Then
\begin{align}
	\mathscr{Q}_{\Omega_2,l}(\mu)
	=&\sum_{l\le n< r\mu-d/2+1}\left(\left\lfloor\mu G\left(\frac{n+d/2-1}{\mu}\right)+c\right\rfloor- \left\lfloor G(r)\mu+c\right\rfloor\right)\nonumber\\
	=&\sum_{l\le n< r\mu-d/2+1}\left(\mu G\left(\frac{n+d/2-1}{\mu}\right)-\mu G\left(r\right) \right)\label{sum4}\\
	&-\sum_{l\le n<  r\mu-d/2+1}\psi\left(\mu G\left(\frac{n+d/2-1}{\mu}\right)+c\right)+L_l+O(1).\label{sum5}
\end{align}
Applying the Euler--Maclaurin summation formula to the sum in \eqref{sum4} and the second mean value theorem yields
\begin{equation*}
\eqref{sum4}=\left|\{(x,y)\in(\mu \Omega)_l :y> G(r)\mu \}\right|+O(1).
\end{equation*}
Applying van der Corput's second derivative estimate (see \cite[Satz 5]{corput:1923}) to the sum in \eqref{sum5}  with $f(x)=\mu G(\frac{x+d/2-1}{\mu})+c$ yields the sum is bounded by
\begin{align*}
& \int_{l}^{r\mu-d/2}\left|f''(x)\right|^{\frac{1}{3}}\,\textrm{d}x+\max_{l\leq x\leq r\mu-d/2}|f''(x)|^{-\frac{1}{2}}+1\\
\leq  &  \mu^{2/3}\int_{0}^{r}\left|G''(x)\right|^{\frac{1}{3}}\,\textrm{d}x+\mu^{1/2}\max_{0\leq x\leq r}|G''(x)|^{-\frac{1}{2}}+1\lesssim \mu^{2/3}.
\end{align*}
We therefore get \eqref{ND2NC} easily.

We next prove \eqref{ND2NC} with the error term $O_{\epsilon}(\mu^{2\theta^*+\epsilon})$ in the rational case. Denote the rational slope of the tangent line at $J$ by  $G'(r)=-a/q<0$ with $a$, $q$ positive and relatively prime. Let
\begin{equation*}
\mathcal{T}:=\left\{(x,y)\in\mathbb{R}^2: 0\leq x<r, G(r)< y\leq G(r)+aq^{-1}(r-x)\right\}
\end{equation*}
denote the triangle bounded by $y$-axis, $y=G(r)$ and the tangent line at $J$, and
\begin{equation*}
\Omega^*_2:=\left\{(x,y)\in\mathbb{R}^2: 0\leq x<r,G(x)< y\leq G(r)+aq^{-1}(r-x)\right\}
\end{equation*}
be the domain $\mathcal{T}\setminus \Omega_2$. Thus
\begin{equation}\label{label-ND2NC}
	\mathscr{Q}_{\Omega_2,l}(\mu)=\mathscr{Q}_{\mathcal{T},l}(\mu)-\mathscr{Q}_{\Omega_2^*,l}(\mu),
\end{equation}
where counting functions $\mathscr{Q}_{\mathcal T,l}(\mu)$ and $\mathscr{Q}_{\Omega_2^{*},l}(\mu)$ are both defined by \eqref{countfcn3} with $\Omega_i$ replaced by $\mathcal T$ and $\Omega_2^{*}$ respectively.

Concerning $\mathscr{Q}_{\Omega_2^{*},l}(\mu)$, we count points along lines $l_t$:
\begin{equation*}
a(x-d/2+1)+q(y+c)=t,\quad t\in \mathbb{Z}.
 \end{equation*}
Observe that $l_t$ contains points from the lattice $\mathbb{Z}^2+(d/2-1,-c)$ if and only if $t\in \mathbb{Z}$.  The line $l_t$ intersects the lower boundary curve of $\mu\Omega_2^*$ between endpoints $(l+d/2-1,\mu G((l+d/2-1)/\mu))$ and $(r\mu,G(r)\mu)$ at a unique point if $t\in[\mu q\beta_l,\mu q\gamma]$, where
\begin{equation*}
\beta_l=G\left(\frac{l+d/2-1}{\mu}\right)+\left(c+\frac{a}{q}l\right)\frac{1}{\mu}
\end{equation*}
and
\begin{equation*}
\gamma=G(r)+\frac{a}{q}r+\left(c+\frac{a}{q}\left(1-\frac{d}{2}\right)\right)\frac{1}{\mu}.
\end{equation*}
We denote the $x$-coordinate of the intersection point by $\mu T(t/(\mu q))$ where $T$ is a strictly increasing function from $[\beta_l,\gamma]$ to $[(l+d/2-1)/\mu,r]$. It satisfies the equation
\begin{align}\label{definition-TNC}
G(T(y))+\frac aq T(y)+\left(c+\frac{a}{q}\left(1-\frac{d}{2}\right)\right)\frac{1}{\mu}=y, \qquad y\in[\beta_l,\gamma].
\end{align}
For every $t_0\in\{0,\dots,q-1\}$ let $x_0\in\{0,\dots,q-1\}$ be such that $ax_0\equiv t_0\pmod q$. If $t\equiv t_0\pmod q$ the points $(n+d/2-1,k-c)$ on $l_t$ are those with $n=x_0+qm$, $m\in\mathbb{Z}$. Therefore the number of $(\nu,k-c)$ with $n\ge l$ in $\mu\Omega_2^*\cap l_t$ is equal to the number of integers $m$ such that
\begin{equation*}
\frac{l-x_0}{q}\leq m<\frac{1}{q}\left(\mu T\left(\frac{t}{\mu q}\right)-x_0-\frac{d}{2}+1\right).
\end{equation*}
Thus
\begin{equation*}
\mathscr{Q}_{\Omega_2^*,l}(\mu)=S_1+S_2+S_3
\end{equation*}
with
\begin{equation*}
S_1:=\sum_{\mu q\beta_l<t\leq\mu q \gamma}
\left(\frac{\mu}{q}T\left(\frac{t}{\mu q}\right)-\frac{l+d/2-1}{q}\right),
\end{equation*}
\begin{equation*}
	S_2:=\sum_{t_0=0}^{q-1}\psi\left(\frac{x_0-l}{q}\right)\left(\mu(\beta_l-\gamma)+O(1)\right)
\end{equation*}
and
\begin{equation*}
	S_3:=\sum_{t_0=0}^{q-1}\sum_{\mu  \beta_l-t_0q^{-1}<m\leq\mu  \gamma-t_0q^{-1}}
	\psi\left(-\frac{\mu}{q} T\left(\frac{t_0+qm}{\mu q}\right)+\frac{x_0+d/2-1}{q}\right).
\end{equation*}

Applying the Euler--Maclaurin summation formula and \eqref{definition-TNC} to the sum $S_1$ yields that
\begin{align*}
S_1=&\left|\left\{(x,y)\in \mu \Omega_2^* : x\geq  l+\frac{d}{2}-1\right\}\right|-\frac{\psi(\mu q \gamma)}{q}\left(r\mu-\left(l+\frac{d}{2}-1\right)\right)\\
&+\frac{1}{q^2}\int_{\mu q\beta_l}^{\mu q \gamma}T'\left(\frac{x}{\mu q}\right)\psi(x)\,\mathrm dx.
\end{align*}
The last integral is of order $O(\mu^{1/3})$. Indeed, by splitting it into two parts over $[\mu q\beta_l,\mu q \gamma-1)$ and $[\mu q \gamma-1,\mu q \gamma]$ respectively and using the second mean value theorem and \cite[Lemma 4.7 and (4.15) on P. 33]{GMWW:2019}, we have it bounded by
\begin{equation*}
\sup_{\beta_l\leq y\leq\gamma-1/(\mu q)}\left|T'(y)\right|+\mu\left(T(\gamma)-T\left(\gamma-\frac1{\mu q}\right)\right)\lesssim\mu^{1/3}.
\end{equation*}

Since
\begin{align}\label{complete-psi-sumNC}
\sum_{m=0}^{q-1}\psi\left((x+m)/q\right)=\psi(x),
\end{align}
we have
\begin{equation*}
S_2=\mu(\gamma-\beta_l)/2+O(1).
\end{equation*}

In the inner sum of $S_3$, the part with $\lfloor \mu \gamma\rfloor-\mu^{2\theta^*}<m\leq \mu  \gamma$ is of order $O(\mu^{2\theta^*})$ trivially and the rest part is divided into sums of the form
\begin{equation}
\sum_{\lfloor \mu \gamma\rfloor-2M\leq \lfloor \mu \gamma\rfloor-M'< m\leq \lfloor \mu \gamma\rfloor-M}\psi\left(-\frac{\mu}{q} T\left(\frac{t_0+qm}{\mu q}\right)+\frac{x_0+d/2-1}{q}\right) \label{sum7}
\end{equation}
with $M=2^j\mu^{2\theta^*}\lesssim \mu$. Do a substitution $\widetilde{m}:=\lfloor\mu \gamma\rfloor- m$ and rewrite it as
\begin{equation*}
\eqref{sum7}=\sum_{M\leq \widetilde{m}<M'\leq 2M}\psi\left(NF\left(\frac{\widetilde{m}}{M}\right)\right),
\end{equation*}
where $N=\mu^{1/3}M^{2/3}q^{-1}$ and
\begin{equation}
F(x)=-\left(\frac{\mu}{M}\right)^{2/3}T\left(\gamma-\frac{M}{\mu}x+\frac{c_0}{\mu}\right)+\frac{x_0+d/2-1}{qN}, \quad x\in[1,2], \label{F2}
\end{equation}
with $c_0=t_0/q+\lfloor\gamma\mu\rfloor-\gamma\mu$. Based on the size of derivatives of $T$ (in \cite[Lemma 4.7]{GMWW:2019}), we have $|F^{(j)}(x)|\asymp 1$ for $j=1,2,3$. By Theorem \ref{expo sum} with $T=MN$, we have
\begin{equation*}
\eqref{sum7}\lesssim_{\epsilon} \left(M^{5/3}\mu^{1/3}\right)^{\theta^*+\epsilon}.
\end{equation*}
Summing over $j$ gives
\begin{equation*}
	S_3\lesssim_{\epsilon} \mu^{2\theta^*+2\epsilon}.
\end{equation*}

Combining results of $S_1$, $S_2$ and $S_3$ yields
\begin{equation}
\begin{split}\label{ND2-label2NC}
\mathscr{Q}_{\Omega_2^*,l}(\mu)=&\left|\left\{(x,y)\in \mu \Omega_2^* : x\geq  l+\frac{d}{2}-1\right\}\right|-\frac{r\mu-l}{q}\psi(\mu q \gamma)\\
    &+\frac{\mu}{2}(\gamma-\beta_l)+O_{\epsilon}\left(\mu^{2\theta^*+\epsilon}\right).
\end{split}
\end{equation}

Concerning $\mathscr{Q}_{\mathcal{T},l}(\mu)$ we notice that
\begin{align*}
	\mathscr{Q}_{\mathcal{T},l}(\mu)=&\sum_{l\le n< r\mu-d/2+1}\left(\left\lfloor G(r)\mu-\frac{a}{q}(\nu-r\mu)+c\right\rfloor- \left\lfloor G(r)\mu+c\right\rfloor\right)\\
	=&\sum_{l\leq n<r\mu-d/2+1}\left(\frac{a}{q}\left(r\mu-\nu \right)-\psi\left(\gamma\mu-\frac{a}{q}n\right)\right)+L_l+O(1).
\end{align*}
By the Euler--Maclaurin formula, the first sum is equal to
\begin{equation*}
\left|\left\{(x,y)\in \mu\mathcal{T} : x\geq  l+\frac{d}{2}-1\right\}\right|-\frac{a}{2q}\left(l+\frac{d}{2}-1-r\mu\right)+O(1).
\end{equation*}
By the relation \eqref{complete-psi-sumNC}, the second sum is equal to $(r\mu-l)q^{-1}\psi(\mu q\gamma)+O(1)$. Hence
\begin{equation}
\begin{split}
\mathscr{Q}_{\mathcal{T},l}(\mu)=&\left|\left\{(x,y)\in \mu\mathcal{T} : x\geq  l+\frac{d}{2}-1\right\}\right|-\frac{r\mu-l}{q}\psi(\mu q\gamma)\\
  &-\frac{a}{2q}\left(l+\frac{d}{2}-1-r\mu\right)+L_l+O(1).
\end{split}	\label{sum6}
\end{equation}

By \eqref{label-ND2NC}, \eqref{ND2-label2NC} and \eqref{sum6} we readily get \eqref{ND2NC} with the error term $O_{\epsilon}(\mu^{2\theta^*+\epsilon})$ in the rational case. This finishes the proof of  Lemma \ref{theorem:no-in-D}.
\end{proof}

We next study the lattice counting associated with the band $\mathscr{Q}_{\mathcal{B}_{r\mu}^{*},l}$.

\begin{lemma}\label{no-in-band}
For $r>0$ and $d\geq 2$, we have
\begin{equation*}
	\mathscr{Q}_{\mathcal{B}_{r\mu}^{*},l}=\frac{1}{4}\left(r\mu-l-\frac{d-3}{2}\right)+O\left(\mu^{2/3} \right)    
\end{equation*}
with $0\leq l\leq r\mu-d/2+1$ and $*\in\{ \mathtt{D},\mathtt{N}\}$. If the boundary curve of $\Omega$ has  a tangent in $J$ with rational slope, the remainder estimate can be improved to
\begin{equation*}
	O_{\epsilon}\left(\mu^{2\theta^*+\epsilon}\right).
\end{equation*}
\end{lemma}

\begin{proof}
It suffices to find the size of the set
\begin{equation*}
\left\{(\nu,k-c)\in \mathcal{B}_{r\mu}^{\mathtt{D}}:n\in \mathbb{Z}_+,n\ge l, k\in \mathbb N \right\},
\end{equation*}
since by taking $c=0$ and $3/4$ we get $\mathscr{Q}_{\mathcal{B}_{r\mu}^{\mathtt{D}},l}$ and $\mathscr{Q}_{\mathcal{B}_{r\mu}^{\mathtt{N}},l}$ respectively. We just need to count the number of points $(\nu,k-c)$ in $\mu\Omega_2$ and in $\mu\Omega_2\cup \mathcal{B}_{r\mu}^{\mathtt{D}}$, and then find their difference. However, the former number is already given by \eqref{ND2NC}, while the latter number can be easily obtained by repeating the computation of \eqref{ND2NC} in the proof of Lemma \ref{theorem:no-in-D}.
\end{proof}

In the last part of this subsection, we derive an asymptotics of $\mathscr{Q}_{\Omega}(\mu)$ by Lemma \ref{theorem:no-in-D}, based on which and also on Lemma \ref{no-in-band}  and Proposition \ref{difference4} we derive an asymptotics of $\mathscr{P}_{\Omega}(\mu)$. To conclude, the shell part of Theorem \ref{specthm} follows immediately from the asymptotics of $\mathscr{P}_{\Omega}(\mu)$ and Proposition \ref{difference1}.

\begin{theorem}\label{thm Ndu}
For $0<r<R$ and $d\geq 2$, we have
\begin{equation*}
\mathscr{Q}_{\Omega}^*(\mu)=\frac{R^d-r^d}{2^{d}(\Gamma(d/2+1))^{2} }\mu^d\mp\frac{R^{d-1}}{2\cdot(d-1)!}\mu^{d-1}+O\left(\mu^{d-2+\frac{2}{3}}\right)
\end{equation*}
and
\begin{equation*}
	\mathscr{P}_{\Omega}^*\textbf{}(\mu)=\frac{R^d-r^d}{2^{d}(\Gamma(d/2+1))^{2} }\mu^d\mp\frac{R^{d-1}+r^{d-1}}{2\cdot(d-1)!}\mu^{d-1}+O\left(\mu^{d-2+\frac{2}{3}}\right),
\end{equation*}
where we take the sign ``$-$'' (resp., ``$+$'') when $*=\mathtt{D}$ (resp., $*=\mathtt{N}$). If the boundary curve of $\Omega$ has  a tangent in $J$ with rational slope, both remainder estimates can be improved to
\begin{equation*}
	O_{\epsilon}\left(\mu^{d-2+2\theta^*+\epsilon}\right).
\end{equation*}
\end{theorem}

\begin{proof}
We claim that the following asymptotics hold
\begin{equation}
		\sum_{0\leq l<R\mu-\frac{d-2}{2}} \left(m_{l}^d-m_{l-1}^d\right)\left|(\mu\Omega)_l\right|=\frac{R^d-r^d}{2^{d}\left(\Gamma\left(d/2+1\right)\right)^{2}} \mu^d+O\left(\mu^{d-2} \right) \label{mainterm}
\end{equation}
and
\begin{equation}
	\sum_{0\leq l< a\mu-\frac{d-2}{2}}\left(m_{l}^d-m_{l-1}^d\right)\left(a\mu-l-\frac{d-3}{2}\right)=\frac{2(a\mu)^{d-1}}{(d-1)!}+O\left( \mu^{d-2}\right)     \label{area}
\end{equation}
with $a>0$. It is easy to check that, for $d\geq 3$ and $l\geq 1$,
\begin{equation}
	m_{l}^d-m_{l-1}^d=\frac{2}{(d-3)!}\left(l+\frac{d-3}{2}\right)^{d-3}+
	\mathcal{E}_l, \label{999}
\end{equation}
where $\mathcal{E}_l=0$ if $3\leq d\leq 5$ and $\mathcal{E}_l=O(l^{d-5})$ if $d\geq 6$, and that $m_{0}^d-m_{-1}^d=1$. We also know that $m_{l}^2-m_{l-1}^2$ equals $1$ if $l=0,1$, and $0$ otherwise. Based on the above claim (with $a=R,r$) and estimates, Lemmas \ref{theorem:no-in-D} and \ref{no-in-band}  and Proposition \ref{difference4}, it is straightforward to verify the desired asymptotics of $\mathscr{Q}_{\Omega}(\mu)$ and $\mathscr{P}_{\Omega}(\mu)$.

We will assume $d\geq 3$ below. When $d=2$ the above claim follows directly from Lemma \ref{theorem:no-in-D} and the fact  $|\Omega|=(R^2-r^2)/8$.

We first verify the asymptotics \eqref{mainterm}.  By \eqref{999} its left side
is equal to
\begin{equation}
\sum_{0<l< R\mu-\frac{d-2}{2}}f(l)+\sum_{0< l< R\mu-\frac{d-2}{2}} \mathcal{E}_l \left|(\mu\Omega)_l\right|+|\Omega|\mu^2+O(\mu), \label{mainterm1}
\end{equation}
where
\begin{equation*}
f(x)=\frac{2}{(d-3)!}\left(x+\frac{d-3}{2}\right)^{d-3}\int_{x+\frac{d-3}{2}}^{R\mu} \mu G\left(\frac{t}{\mu} \right)  \, \textrm{d}t.
\end{equation*}
It is trivial that the second sum in \eqref{mainterm1} is at most $O(\mu^{d-2})$. By the Euler--Maclaurin summation formula, the first sum in \eqref{mainterm1} is equal to
\begin{equation}
\int_{0}^{R\mu-\frac{d-2}{2}} f(t) \, \textrm{d}t+\int_{0}^{ R\mu-\frac{d-2}{2}} \psi(t)f'(t) \, \textrm{d}t-\frac{1}{2}f(0)+O\left(\mu^{d-3}\right).\label{mainterm2}
\end{equation}
 By using integration by parts, a change  of variables and properties of the beta and gamma functions, we get
\begin{align*}
\int_{0}^{R\mu-\frac{d-2}{2}} f(t) \, \textrm{d}t&=\frac{2\mu^d}{(d-2)!}\int_{0}^{R} t^{d-2}G(t)\, \textrm{d}t+O\left(\mu^{d-2} \right) \\
     &=\frac{R^d-r^d}{2^{d}\left(\Gamma\left(d/2+1\right)\right)^{2}} \mu^d+O\left(\mu^{d-2} \right).
\end{align*}
By the second mean value theorem, the second term in \eqref{mainterm2} is of size $O(\mu^{d-2})$. It is also clear that $f(0)/2$ is of size $O(\mu^2)$ if $d\geq 4$ and equal to $|\Omega|\mu^2$ if $d=3$. Plugging these results in \eqref{mainterm1} gives us \eqref{mainterm}.

It remains to verify the asymptotics \eqref{area}. By \eqref{bino-bound}
and summation by parts, the left side of \eqref{area} is equal to
\begin{equation*}
2\sum_{0\leq l\leq \lfloor a\mu-\frac{d+2}{2}\rfloor}\binom{l+d-2}{d-2}+O\left( \mu^{d-2}\right).
\end{equation*}
Simplifying the sum of binomial coefficients reduces this to
\begin{equation*}
2\binom{\lfloor a\mu-\frac{d+2}{2}\rfloor+d-1}{d-1}+O\left( \mu^{d-2}\right)=\frac{2(a\mu)^{d-1}}{(d-1)!}+O\left( \mu^{d-2}\right),
\end{equation*}
as desired.
\end{proof}

%%%%%%%%%%%%%%%%%%%%%%%%%%%%%%%%%%%%%%%%%%%%%%%%%%%%%%%%%%%%%%%%%%%%%%%%%%%%%%%%%%%%%%%%%%%%%%%%%%%

\subsection{Weyl's law for balls}\label{sec5.2}

In order to obtain the asymptotics of $\mathscr{N}_\mathbb{B}(\mu)$, based on the results of Subsection \ref{subsec4.2}, we need to study the lattice counting function  $\mathscr{Q}_{\Omega_0}(\mu)$. By Lemma \ref{theorem:no-in-D} with $r=0$, as argued in the proof of Theorem \ref{thm Ndu},  we obtain the following.

\begin{theorem} \label{thm5.5}
For $d\geq 2$, we have
\begin{equation*}
\mathscr{Q}_{\Omega_0}^*(\mu)=\frac{R^d}{2^{d}(\Gamma(d/2+1))^{2} }\mu^d\mp\frac{R^{d-1}}{2\cdot(d-1)!}\mu^{d-1}+O_{\epsilon}\left(\mu^{d-2+2\theta^*+\epsilon}\right),
\end{equation*}
where we take the sign ``$-$'' (resp., ``$+$'') when $*=\mathtt{D}$ (resp., $*=\mathtt{N}$).
\end{theorem}

Finally, the ball part of Theorem \ref{specthm} follows immediately from Theorem \ref{thm5.5} and Proposition \ref{prop4.7}.

\begin{remark}
As a final remark on the Weyl's law for disks, we would like to point out some misunderstandings in  a recent paper \cite{Huxley:2024} by Huxley.
	
On \cite[P. 106]{Huxley:2024}, Huxley wrote ``\textit{A recent paper [6] by Guo, M\"{u}ller, Wang and Wang considers the more difficult problem of the Dirichlet eigenvalues of an annulus for which the ratio $r/R$ of the radii is a rational number... They claim that the estimate for Dirichlet eigenvalues of the disc follows by letting $r$ tend to zero...}''. This is not what Guo, M\"{u}ller, Wang and Wang claimed in their paper. On the contrary, they stated clearly (on \cite[P. 5]{GMWW:2019}, right above Theorem 1.4) that ``\textit{Its rigorous proof relies on the reduction step from the eigenvalue counting to the lattice point counting (see [6, Section 3], [10, Section 6] and its variant in Section 3), Theorem 4.1 (together with the symmetry of the domain D) and...}''. In fact, the focus in Guo, M\"{u}ller, Wang and Wang \cite{GMWW:2019} is the annulus case rather than the disk case whose proof is similar but much simpler. Hence \cite{GMWW:2019} did not provide a detailed proof for the disk case but a one-sentence sketch. A subsequent paper \cite{Guo} (which was not quoted in Huxley \cite{Huxley:2024}) focuses on extending the result of disks to high-dimensional balls for the Dirichlet Laplacian. A detailed proof for the disk case is essentially contained therein. By the way, \cite{GMWW:2019} obtained an improved remainder estimate for annuli under the assumption that $\pi^{-1}\arccos(r/R)\in \mathbb{Q}$ rather than under the assumption $r/R\in\mathbb{Q}$ as Huxley stated.
\end{remark}

%%%%%%%%%%%%%%%%%%%%%%%%%%%%%%%%%%%%%%%%%%%%%%%%%%%%%%%%%%%%%%%%%%%%%%%%%%%%%%%%%%%%%%%%%%%%%%%%%%%

\section{Estimates of rounding error sums} \label{sec6}

This section is devoted to estimating sums of rounding errors. By applying the result of Theorem \ref{expo sum} to the analysis in the previous section, we can obtain  the refined bounds presented in this paper.

Recently Li and the last author \cite{LY2023} made a progress on the Gauss circle problem (among other results) by improving Huxley's exponent $131/208=0.6298\cdots$ (in \cite{Huxley:2003}) to $2\theta^*=0.6289\cdots$. Their work contains a new exponential sum estimate \cite[Theorem 4.2]{LY2023}.

Let us explain how the estimate of exponential sums was obtained. It was first observed by Bourgain \cite{Bourgain:2017} that the decoupling theory of harmonic analysis can be brought into the study of the first spacing problem. He derived an essentially sharp $L^{12}$ estimate of an exponential sum whose Fourier support is a curve in $\mathbb R^4$. This exponent $12$ is essentially optimal, indicating the first spacing problem in estimating $|\zeta(\frac{1}{2}+it)|$ cannot be further improved.  Then, in Bourgain and Watt \cite[Section 5]{BW2018}, they showed that the double large sieve inequality can be generalized  by  combining the locally constant property of exponential sums and H\"older's inequality. Next, in \cite{BWpreprint},  the same authors tried to bring this new idea of using the variant double large sieve inequality into the study of the first spacing problem of the circle and divisor problems. However, there was a technical issue with their (3.11)---a positive power of $N$ was missing when they applied the broad-and-narrow dichotomy argument, and thus their result on the first spacing problem in \cite{BWpreprint} does not hold. Subsequently, in \cite{LY2023}, Li and the last author applied Bourgain and Watt's variant double large sieve to generalize the setting from $G_4$ to $G_q$, and obtained a key estimate of $G_q$ with $q$ a little bit larger than $4$ (while in contrast Huxley used a bound of $G_4$ in \cite{Huxley:2003}). They have combined recent advancements of the small cap decoupling theory for cones (made by Guth and Maldague \cite{GM:2022}) with results on some diophantine counting problems to obtain better bounds in the first spacing problem. At last, by combining with Huxley's work on the second spacing problem, they obtained their improved pointwise estimate of two-dimensional exponential sums.

We intended to apply \cite[Theorem 4.2]{LY2023} to the lattice point counting problems encountered in this paper. Unfortunately, the assumption
\begin{equation} \label{ex cond}
	\left| F'(x)F^{(3)}(x)-3F''(x)^2\right| \gtrsim 1 \textrm{ for $x\asymp 1$}
\end{equation}
is not satisfied by the functions $F$ we have in \eqref{F1} and \eqref{F2}.  There is no such problem in \cite{LY2023} because the functions $F$, encountered in the  study of the circle and divisor problems, are variants of the reciprocal function $1/x$ that do satisfy the assumption \eqref{ex cond}. And meanwhile, assuming \eqref{ex cond} helps reduce the number of cases to be discussed. It is noted that Bourgain and Watt also made this assumption in their estimate of exponential sums in \cite{BWpreprint} for the same reason.

To overcome this obstacle, inspired by Huxley's \cite[Proposition 3]{Huxley:2003}, we examined all the details of the proof of \cite[Theorem 4.2]{LY2023} and discovered that, even  without the assumption \eqref{ex cond}, its conclusion still holds within a certain limited range, which suffices  for our needs in this paper. Based on this discovery, the following theorem can be formulated.  Essentially, the proof of this theorem is the same as those of \cite[Theorems 1.2 and 4.2]{LY2023}. For completeness, a sketchy proof will be provided below, with a particular emphasis on clarifying  why the assumption \eqref{ex cond} is unnecessary in the range \eqref{range a}.  For more details, we refer the interested readers to \cite[Section 5]{BWpreprint} and \cite[Sections 4 and 5]{LY2023}).

Let $T,M$ be two large positive parameters and $F$ a real-valued function defined on $[1/2,2]$, three times continuously differentiable,  satisfying $|F^{(j)}(x)|\asymp 1$ for $j=1,2,3$. We have the following bound for sums of rounding errors formed with the sawtooth function $\psi$.

\begin{theorem}  \label{expo sum}
	If $M$ is in the range
	\begin{equation}   \label{range a}
		T^{\frac{141}{328}+c}\le M \le T^{\frac{1}{2}}
	\end{equation}
	for some small absolute constant $c>0$, then for any $\epsilon>0$ we have
	\begin{equation} \label{stf}
		\sum_{m=M}^{M_2} \psi\left( \frac{T}{M} F\left( \frac{m}{M} \right) \right)=O_\epsilon \left(T^{\theta^*+\epsilon}\right),
	\end{equation}
	where $M_2$ is an integer in the range $M\le M_2\le 2M-1$ and $\theta^*=0.314483\cdots$ (as defined in \cite[Definition 1.1]{LY2023}) is the opposite of the unique solution in the interval $[-0.35,-0.3]$ to the equation
	\begin{equation} \label{definition of theta}
		-\frac{8}{25}x-\frac{1}{200}\left(\sqrt{2(1-14x)}-5\sqrt{-1-8x}\right)^2+\frac{51}{200}=-x.
	\end{equation}
\end{theorem}

\begin{remark}
	One can compare this result with Huxley's \cite[Case (A) of Proposition 3]{Huxley:2003} with $\kappa=3/10$. If $M$ falls within the range \eqref{range a}, then it lies within the range (1.8) but not within (1.12) of Huxley's Proposition 3. Consequently, the bound  \eqref{stf} offers an improvement over  Huxley's bound (1.25).
\end{remark}

\begin{proof}[Proof of Theorem 6.1.]
	One can follow the approach of handling the sums in \cite[(5.1) and (5.2)]{LY2023} to treat the sum in equation \eqref{stf}. First of all, one may apply to \eqref{stf} the truncated Fourier expansion of the sawtooth function
	\begin{equation*}  
		\psi(t)=\text{Im} \sum_{1\le h\le Y}  \frac{e(ht)}{\pi h}+O\left(\frac{1}{1+\|t\|Y}\right) \,\,  \textrm{with $Y=MT^{-\theta^*}$},
	\end{equation*}
	and then one is left to estimate the exponential sum
	\begin{equation*}
		S=\sum_{h} \phi\left(\frac{h}{H}\right) \sum_{m} \chi\left(\frac{m}{M}\right) e\left(\frac{hT}{M} F\left(\frac{m}{M}\right)\right),
	\end{equation*}
	where $H$ is a dyadic number in $[1,Y]$ and $\phi, \chi$ are simply indicator functions of subintervals of $[1/2,2]$. An application of the Bombieri--Iwaniec method would transform the estimate  to that of a bilinear form
	\begin{equation*}
		\mathop{\sum}\limits_{\substack{L\le l\le 2L \\K\le k\le 2K}} e\left( \vec{x}_{\frac{a}{r}}\cdot \vec{y}_{(k,l)} \right),
	\end{equation*}
	where
	\begin{equation*} 
		\vec{x}_{\frac{a}{r}}=\left( \frac{\overline{a}}{r}, \frac{\overline{a}c}{r}, \frac{1}{\sqrt{\mu r^3}} \frac{\kappa}{\sqrt{\mu r^3}} \right)
	\end{equation*}
	(see \cite[Subsection 4.4]{LY2023} for the meaning of the notations) and
	\begin{equation*}
		\vec{y}_{(k,l)}=\left(k,lk,l\sqrt{k},\frac{l}{\sqrt{k}}\right)
	\end{equation*}
	are vectors in $\mathbb R^4$. Here, we have omitted the constant multiples in front and those negligible error terms. Applying Bourgain and Watt's variant of the double large sieve inequality in \cite[Section 5]{BW2018} reduces the problem into the following first and second spacing problems.

	The first spacing problem seeks for the bound of
	\begin{equation*}
		G_q:=\left\|\sum_{k\sim K}\sum_{l\sim L} a_{kl}e(lx_1+klx_2+l\sqrt{k}x_3)\right\|_{L^q_{\#}\left[|x_1|\le 1,|x_2|\le 1,|x_3|\le \frac{1}{\eta L\sqrt{K}}\right]},
	\end{equation*}
	where $a_{kl}$ are arbitrary coefficients such that $|a_{kl}|\le 1$, $q\ge 4$,  the $L^q_{\#}$-norm is given by
	\begin{equation*}
		\|f\|_{L^p_{\#}(B)}=\left(\frac{1}{|B|}\int_{B}|f|^p\right)^{1/p},
	\end{equation*}
	the parameters $K,L$ are integers, $\eta>0$, and they satisfy
	\begin{equation*} 
		1\le L<K\le \frac{1}{\eta}\le KL.
	\end{equation*}
	One would directly resort to \cite[Proposition 3.1]{LY2023} for this part which is irrelevant to the assumption \eqref{ex cond} at all. It is in this part one can combine Guth and Maldague's work on the small cap decoupling theory for cones in \cite{GM:2022} with results on some diophantine counting problems.

	The second spacing problem asks for the number of pairs $(a/r, a_1/r_1)$ with $a,a_1\asymp A$ and $r,r_1\asymp Q$ such that
	\begin{align*}
		\left\|\frac{\overline{a}}{r}-\frac{\overline{a_1}}{r_1} \right\| & \lesssim \frac{1}{KL},
		\\ \left\|\frac{\overline{a}c}{r}-\frac{\overline{a_1}c_1}{r_1} \right\| & \lesssim \frac{1}{L},
		\\ \left| \frac{1}{\sqrt{\mu r^3}}-\frac{1}{\sqrt{\mu_1 r_1^3}} \right| & \lesssim \frac{1}{L\sqrt{K}},
		\\ \left| \frac{\kappa}{\sqrt{\mu r^3}}-\frac{\kappa_1}{\sqrt{\mu_1 r_1^3}} \right| & \lesssim \frac{\sqrt{K}}{L}.
	\end{align*}
	They can be further simplified into the form
	\begin{align*}
		\left\|\frac{\overline{a}}{r}-\frac{\overline{a_1}}{r_1} \right\| & \le \Delta_1  \textrm{ with $\Delta_1$ much less than $1$,} 
		\\ \left\|\frac{\overline{a}c}{r}-\frac{\overline{a_1}c_1}{r_1} \right\| & \le \Delta_2,  
		\\ \left| \frac{\mu_1 r_1^3}{\mu r^3} -1 \right| & \le \Delta_3,   
		\\ | \kappa-\kappa_1 | & \le \Delta_4  
	\end{align*}
	with parameters $\Delta_i$ properly chosen. One can then use Huxley's \cite[Lemmas 3.3 and 3.4]{Huxley:2003}  for this part. As a matter of fact, the assumption \eqref{ex cond} is only needed in the case $M\gtrsim T^{\frac{181}{328}}(\log T)^{\frac{2907}{45920}}$ to ensure these Huxley's results on the second spacing problem to hold. Clearly, the parameter $M$ in \eqref{range a} falls outside this specific range. Hence this part is irrelevant to  the assumption \eqref{ex cond} either.
	
	As a final step, by integrating results for the two parts above, one can prove \eqref{stf} in the same way as Li and the last author did for \cite[Theorem 1.2]{LY2023}.
\end{proof}

%%%%%%%%%%%%%%%%%%%%%%%%%%%%%%%%%%%%%%%%%%%%%%%%%%%%%%%%%%%%%%%%%%%%%%%%%%%%%%%%%%%%%%%%%%%
%%%%%%%%%%%%%%%%%%%%%%%%%%%%%%%%%%%%%%%%%%%%%%%%%%%%%%%%%%%%%%%%%%%%%%%%%%%%%%%%%%%%%%%%%%%

\appendix

\section{Useful asymptotics of Bessel functions}

For the convenience of readers we first quote Olver's uniform asymptotic expansions of Bessel functions of large order
(see \cite[p. 368--369]{abram:1972} or \cite{olver:1954}):
\begin{equation}
J_\nu(\nu z)\sim \left(\frac{4\zeta}{1-z^2}\right)^{\!\! 1/4}\! \! \left(\frac{\mathrm{Ai}(\nu^{2/3}\zeta)}{\nu^{1/3}}
\sum_{k=0}^{\infty}\frac{a_k(\zeta)}{\nu^{2k}}+\frac{\mathrm{Ai}^{\prime}(\nu^{2/3}\zeta)}{\nu^{5/3}}
\sum_{k=0}^{\infty}\frac{b_k(\zeta)}{\nu^{2k}}\right),  \label{jnuse111}
\end{equation}
\begin{equation}
Y_\nu(\nu z) \sim -\left(\frac{4\zeta}{1-z^2}\right)^{\!\! 1/4}\!\!\left(\frac{\mathrm{Bi}(\nu^{2/3}\zeta)}
{\nu^{1/3}}\sum_{k=0}^{\infty}\frac{a_k(\zeta)}{\nu^{2k}}+\frac{\mathrm{Bi}^{\prime}(\nu^{2/3}\zeta)}{\nu^{5/3}}
\sum_{k=0}^{\infty}\frac{b_k(\zeta)}{\nu^{2k}}\right),  \label{ynuse111}
\end{equation}
\begin{equation}\label{jnuse111NC}
J_\nu'(\nu z)\sim -\frac{2}{z}\left(\frac{1-z^2 }{4\zeta}\right)^{\!\! 1/4}\!\!\left(\frac{\mathrm{Ai}^{\prime}(\nu^{2/3}\zeta)}{\nu^{2/3}}
\sum_{k=0}^{\infty}\frac{d_k(\zeta)}{\nu^{2k}}+\frac{\mathrm{Ai}(\nu^{2/3}\zeta)}{\nu^{4/3}}
\sum_{k=0}^{\infty}\frac{c_k(\zeta)}{\nu^{2k}}\right)
\end{equation}
and
\begin{equation}\label{ynuse111NC}
Y_\nu'(\nu z) \sim \frac{2}{z}\left(\frac{1-z^2 }{4\zeta}\right)^{\! \! 1/4}\!\!\left(\frac{\mathrm{Bi}^{\prime}(\nu^{2/3}\zeta)}{\nu^{2/3}}
\sum_{k=0}^{\infty}\frac{d_k(\zeta)}{\nu^{2k}}+\frac{\mathrm{Bi}(\nu^{2/3}\zeta)}
{\nu^{4/3}}\sum_{k=0}^{\infty}\frac{c_k(\zeta)}{\nu^{2k}}\right),
\end{equation}
where the notation $\sim$ is as defined by 3.6.15 in \cite[P. 15]{abram:1972} and $\zeta=\zeta(z)$ is given by
\begin{equation}
\frac{2}{3}(-\zeta)^{3/2}=\int_{1}^z\frac{\sqrt{t^2-1}}{t}\,\mathrm{d}t=
\sqrt{z^2-1}-\arccos\left(\frac{1}{z}\right)\label{def-zeta1}
\end{equation}or
\begin{equation}
\frac{2}{3}\zeta^{3/2}=\int_{z}^1\frac{\sqrt{1-t^2}}{t}\,\mathrm{d}t=\ln \frac{1+\sqrt{1-z^2}}{z}-\sqrt{1-z^2}.\label{def-zeta2}
\end{equation}
Here the branches are chosen so that $\zeta$ is real when $z$ is positive. $\mathrm{Ai}$ and $\mathrm{Bi}$ denote the Airy functions of the first and second kind. For the definitions and sizes of the coefficients $a_k$, $b_k$, $c_k$ and $d_k$ see \cite[p. 368--369]{abram:1972}. In particular it is known that $a_0(\zeta)=d_0(\zeta)=1$, $b_0(\zeta)$ is bounded, and if $\zeta$ is bounded above then $c_0(\zeta)$ is bounded.
It is easy to check the following expansions of \eqref{def-zeta1} and \eqref{def-zeta2}.
If $z\rightarrow 1+$ then
\begin{equation}\label{bound-zeta+}
(-\zeta(z))^{3/2}=\sqrt{2}(z-1)^{3/2}-\frac{9\sqrt{2}}{20}(z-1)^{5/2}+O\left((z-1)^{7/2}\right).
\end{equation}
If $z\rightarrow 1-$ then
\begin{equation}\label{bound-zeta-}
(\zeta(z))^{3/2}=\sqrt{2}(1-z)^{3/2}+\frac{9\sqrt{2}}{20}(1-z)^{5/2}+O\left((1-z)^{7/2}\right).
\end{equation}

%%%%%%%%%%%%%%%%%%%%%%%%%%%%%%%%%%%%%%%%%%%%%%%%%%%%%%%%%%%%%%%%%%%%%%%%%%%%%%%%%%%%%%%%%%%

We also have the following three lemmas about asymptotics of Bessel functions. Recall that $g(x)=\left(\sqrt{1-x^2}-x\arccos x\right)/\pi$.

\begin{lemma} \label{app-1}
For any $c>0$ and $\nu\ge 0$, if $x\geq \max\{(1+c)\nu, 10\}$ then
\begin{equation}\label{jnasy}
J_{\nu}(x)=\left(\frac{2}{\pi}\right)^{1/2}\left(x^2-\nu^2\right)^{-1/4} \left(\sin\left( \pi
xg\left(\frac{\nu}{x}\right)+\frac{\pi}{4}\right)+O_c\left(x^{-1}\right)\right),
\end{equation}
\begin{equation}\label{ynasy}
Y_{\nu}(x)=\left(\frac{2}{\pi}\right)^{1/2}\left(x^2-\nu^2\right)^{-1/4} \left(\sin\left( \pi
xg\left(\frac{\nu}{x}\right)-\frac{\pi}{4}\right)+O_c\left(x^{-1}\right)\right),
\end{equation}
\begin{equation}\label{jnasyNC}
J_{\nu}'(x)=\left(\frac{2}{\pi}\right)^{\! 1/2}\!\! x^{-1}\!\left(x^2-\nu^2\right)^{1/4} \!\left(\sin\left( \pi
xg\left(\frac{\nu}{x}\right)+\frac{3\pi}{4}\right)+O_c\left(x^{-1}\right)\right)
\end{equation}
and
\begin{equation}\label{ynasyNC}
Y_{\nu}'(x)=\left(\frac{2}{\pi}\right)^{\! 1/2}\!\!x^{-1}\!\!\left(x^2-\nu^2\right)^{1/4} \!\!\left(\sin\left( \pi
xg\left(\frac{\nu}{x}\right)+\frac{\pi}{4}\right)+O_c\left(x^{-1}\right)\right).
\end{equation}
\end{lemma}

These standard asymptotics can be proved by using the method of stationary phase. All proofs are similar. For example see \cite[Lemma 2.1]{Guo} for a proof of \eqref{jnasy} and \cite[Appendix A]{GMWW:2019} for a proof of \eqref{ynasy} with non-negative integers $\nu$.

\begin{lemma} \label{app-2}
For any $c>0$ and all sufficiently large $\nu$, if $\nu < x\le(1+c)\nu$ and $xg(\nu/x)\geq 1$ then
\begin{equation}\label{Bessel22}
J_{\nu}(x)=\frac{(2/\pi )^{1/2}}{(x^2-\nu^2)^{1/4}}\left(\sin\left(\pi
xg\left(\frac{\nu}{x}\right)+\frac{\pi}{4}\right)+O\left(\left(xg\left(\frac{\nu}{x}\right)\right)^{-1}\right)\right),
\end{equation}
\begin{equation}\label{Bessel221}
Y_{\nu}(x)=\frac{(2/\pi )^{1/2}}{(x^2-\nu^2)^{1/4}}\left(\sin\left(\pi
xg\left(\frac{\nu}{x}\right)-\frac{\pi}{4}\right)+O\left(\left(xg\left(\frac{\nu}{x}\right)\right)^{-1}\right)\right),
\end{equation}
\begin{equation}\label{Bessel22NC}
J'_{\nu}(x)=\left(\frac{2}{\pi}\right)^{\!\! 1/2}\frac{(x^2-\nu^2)^{1/4}}{x}\left(\sin\left(\pi
xg\left(\frac{\nu}{x}\right)+\frac{3\pi}{4}\right)+O\left(\left(xg\left(\frac{\nu}{x}\right)\right)^{-1}\right)\right)
\end{equation}
and
\begin{equation}\label{Bessel221NC}
Y'_{\nu}(x)=\left(\frac{2}{\pi}\right)^{\!\! 1/2}\frac{(x^2-\nu^2)^{1/4}}{x}\left(\sin\left(\pi
xg\left(\frac{\nu}{x}\right)+\frac{\pi}{4}\right)+O\left(\left(xg\left(\frac{\nu}{x}\right)\right)^{-1}\right)\right).
\end{equation}
\end{lemma}

The asymptotics in Lemma \ref{app-2} are consequences of Olver's asymptotics of Bessel functions  and well-known asymptotics of Airy functions. Proofs are similar. For example see \cite[Lemma 2.2]{Guo} for a proof of \eqref{Bessel22}.

%%%%%%%%%%%%%%%%%%%%%%%%%%%%%%%%%%%%%%%%%%%%%%%%%%%%%%%%%%%%%%%%%%%%%%%%%%%%%%%%%%%%%%%%%%%%5

The following is an analogue of the formula 9.3.4 in \cite[p. 366]{abram:1972}.

\begin{lemma}\label{9.3.4analogue}
For any $\epsilon>0$ and all sufficiently large $\nu$,
\begin{enumerate}
\item if $0\leq w\leq \nu^{\epsilon}$ then
\begin{align*}
J_\nu\left(\nu+w\nu^{1/3}\right)&=2^{1/3}\nu^{-1/3}\mathrm{Ai}\left(-2^{1/3}w\right)+O\left(\nu^{-1+2.25\epsilon}\right),\\
Y_\nu\left(\nu+w\nu^{1/3}\right)&=-2^{1/3}\nu^{-1/3}\mathrm{Bi}\left(-2^{1/3}w\right)+O\left(\nu^{-1+2.25\epsilon}\right),\\
J'_\nu\left(\nu+w\nu^{1/3}\right)&=-2^{2/3}\nu^{-2/3}\mathrm{Ai}'\left(-2^{1/3}w\right)+O\left(\nu^{-4/3+2.75\epsilon}\right),\\
Y'_\nu\left(\nu+w\nu^{1/3}\right)&=2^{2/3}\nu^{-2/3}\mathrm{Bi}'\left(-2^{1/3}w\right)+O\left(\nu^{-4/3+2.75\epsilon}\right);
\end{align*}

\item if $-\nu^{\epsilon}\leq w\leq 0$ then
\begin{align*}
J_\nu\left(\nu+w\nu^{1/3}\right)&=2^{1/3}\nu^{-1/3}\mathrm{Ai}\left(-2^{1/3}w\right)\left(1+O\left(\nu^{-2/3+2.5\epsilon}\right)\right),\\
Y_\nu\left(\nu+w\nu^{1/3}\right)&=-2^{1/3}\nu^{-1/3}\mathrm{Bi}\left(-2^{1/3}w\right)\left(1+O\left(\nu^{-2/3+2.5\epsilon}\right)\right),\\
J'_\nu\left(\nu+w\nu^{1/3}\right)&=-2^{2/3}\nu^{-2/3}\mathrm{Ai}'\left(-2^{1/3}w\right)\left(1+O\left(\nu^{-2/3+2.5\epsilon}\right)\right),\\
Y'_\nu\left(\nu+w\nu^{1/3}\right)&=2^{2/3}\nu^{-2/3}\mathrm{Bi}'\left(-2^{1/3}w\right)\left(1+O\left(\nu^{-2/3+2.5\epsilon}\right)\right).
\end{align*}
\end{enumerate}
\end{lemma}

\begin{proof}
All these asymptotics can be proved similarly. As an example we sketch the proof of $J'_\nu$ below. See \cite[Lemma A.1]{GMWW:2019} for the proof of $J_{\nu}$ and $Y_{\nu}$.

Let us first consider the case $-\nu^{\epsilon}\leq w<0$. (If $\omega=0$, the desired asymptotics follows immediately from  \eqref{jnuse111NC}.) If we write $z=1+w\nu^{-2/3}$, then by \eqref{jnuse111NC} we have
\begin{equation*}
\begin{split}
&J_{\nu}'\left(\nu+w\nu^{1/3}\right)=J_{\nu}'(\nu z)= \\
&-\frac{2}{z}\left(\frac{1-z^2 }{4\zeta}\right)^{\!\! 1/4}\nu^{-2/3}\left(\mathrm{Ai}'\left(\nu^{2/3}\zeta\right)
\left(1+O\left(\nu^{-2}\right) \right)+\mathrm{Ai}\left(\nu^{2/3}\zeta\right)O\left(\nu^{-2/3}\right)\right),
\end{split}
\end{equation*}
where $\zeta=\zeta(z)>0$, determined by \eqref{def-zeta2}, is such that
\begin{equation}
\nu^{2/3}\zeta=-2^{1/3}w+\frac{3}{10}2^{1/3}w^2 \nu^{-2/3}\left(1+O\left(|w|\nu^{-2/3}\right)\right) \label{app-3}
\end{equation}
implied by \eqref{bound-zeta-}. Thus
\begin{equation}
\left(\frac{1-z^2}{4\zeta}\right)^{1/4}=2^{-1/3}\left(1+O\left(|w|\nu^{-2/3}\right) \right).\label{firstfactor}
\end{equation}

By using the well-known asymptotics of $\mathrm{Ai}(r)$ and $\mathrm{Ai}'(r)$ for $r>0$ (see \cite[p. 448]{abram:1972}), we get
\begin{equation*}
\left|\mathrm{Ai}''\left(\eta\right)\right|=\left|\eta\mathrm{Ai}\left(\eta\right) \right|\lesssim |w|^{1/2}\left|\mathrm{Ai}'\left(-2^{1/3}w\right)\right|
\end{equation*}
for $-2^{1/3}\omega\leq \eta\leq \nu^{2/3}\zeta$.  Hence, by the mean value theorem, we get
\begin{equation}
\mathrm{Ai}'\left(\nu^{2/3}\zeta\right)=\mathrm{Ai}'\left(-2^{1/3}w\right)+
O\left(\mathrm{Ai}'\left(-2^{1/3}w\right)\nu^{-2/3+2.5\varepsilon} \right). \label{app-4}
\end{equation}

Applying \eqref{app-3}, \eqref{firstfactor} and \eqref{app-4} to $J_{\nu}'\left(\nu+w\nu^{1/3}\right)$ gives the desired asymptotics.

To handle the case $0< w\leq \nu^{\epsilon}$ we just need to slightly modify the first part. Again we write $z=1+w\nu^{-2/3}$. The corresponding $\zeta=\zeta(z)<0$, determined by \eqref{def-zeta1}, still satisfies \eqref{app-3} and \eqref{firstfactor}. By using the well-known asymptotics of $\mathrm{Ai}(-r)$ and $\mathrm{Ai}'(-r)$ for $r>0$ (see \cite[p. 448--449]{abram:1972}), we get
\begin{equation*}
\mathrm{Ai}'\left(\nu^{2/3}\zeta\right)=\mathrm{Ai}'\left(-2^{1/3}w\right)+O\left(\nu^{-2/3+2.75\varepsilon} \right).
\end{equation*}
We then easily get the desired asymptotics.
\end{proof}

%%%%%%%%%%%%%%%%%%%%%%%%%%%%%%%%%%%%%%%%%%%%%%%%%%%%%%%%%%%%%%%%%%%%%%%%%%%%%%%%%%%%%%%%%%%%%%%%

\subsection*{Acknowledgments}
Jingwei Guo would like to thank Wolfgang M\"uller for helpful communication. J. Guo is supported by the NSFC (no. 12341102). T. Jiang is supported by the University Annual Scientific Research Plan of Anhui Province (2022AH050173) and the Talent  Cultivation Foundation of Anhui Normal University (2022xjxm036).   Z. Wang is supported by the National Key R and D Program of China (2020YFA0713100) and the NSFC (no. 12171446). X. Yang is partially supported by the Campus Research Board Award RB24028 of UIUC.

%%%%%%%%%%%%%%%%%%%%%%%%%%%%%%%%%%%%%%%%%%%%%%%%%%%%%%%%%%%%%%%%%%%%%%%%%%%%%%%%%%%%%%%%%%%%%%%

\end{document}